\documentclass[reqno,11pt,letterpaper]{amsart}

\usepackage{amsmath,amssymb,amsthm,graphicx,mathrsfs,url}
\usepackage[usenames,dvipsnames]{color}
\usepackage[colorlinks=true,linkcolor=Red,citecolor=Green]{hyperref}
\usepackage{amsxtra}
\usepackage{graphicx}
\usepackage{tikz}
\relax

\numberwithin{equation}{section}

\newcommand{\be}{\begin{equation}}
\newcommand{\ee}{\end{equation}}

\newcommand{\ov}{\overline}

\renewcommand{\Re}{\mathop{\rm Re}\nolimits}
\renewcommand{\Im}{\mathop{\rm Im}\nolimits}

\theoremstyle{plain}

\newtheorem{thm}{Theorem}[section]
\newtheorem{prop}{Proposition}[section]
\newtheorem{cor}[prop]{Corollary}
\newtheorem{lem}[prop]{Lemma}
\newtheorem{definition}[prop]{Definition}

\theoremstyle{definition}

\newtheorem{rem}[prop]{Remark}

\numberwithin{equation}{section}

\def\squarebox#1{\hbox to #1{\hfill\vbox to #1{\vfill}}}

\newcommand{\noi}{\noindent}

\usepackage{amsxtra}

\ifx\pdfoutput\undefined
\DeclareGraphicsExtensions{.pstex, .eps}
\else
\ifx\pdfoutput\relax
\DeclareGraphicsExtensions{.pstex, .eps}
\else
\ifnum\pdfoutput>0
\DeclareGraphicsExtensions{.pdf}
\else
\DeclareGraphicsExtensions{.pstex, .eps}
\fi
\fi
\fi

\title[Sharp decay rate for the damped wave equation with convex-shaped damping]
{Sharp decay rate for the damped wave equation with convex-shaped damping}
\author{Chenmin Sun}
\address{CY Cergy-Paris Universit\'e, Laboratoire de Math\'ematiques AGM, UMR  8088 du CNRS, 2 av. Adolphe Chauvin
	95302 Cergy-Pontoise Cedex, France }
\email{chenmin.sun@cyu.fr}

\usepackage{amssymb}
\usepackage{amsmath, amsthm, amsopn, amsfonts}

\def\11{{\rm 1~\hspace{-1.4ex}l} }
\def\R{\mathbb R}
\def\C{\mathbb C}
\def\Z{\mathbb Z}
\def\N{\mathbb N}

\def\T{\mathbb T}

\begin{document}

	\begin{abstract}
We revisit the damped wave equation on two-dimensional torus where the damped region does not satisfy the geometric control condition. It was shown in \cite{AL14} that, for sufficiently regular damping, the damped wave equation is stale at a rate sufficiently close to $t^{-1}$. 
We show that if the damping vanishes like a H\"older function $|x|^{\beta}$, and in addition, the boundary of the damped region is locally strictly convex with positive curvature, the wave is stable at rate $t^{-1+\frac{2}{2\beta+7}}$, which is better than the known optimal decay rate $t^{-1+\frac{1}{\beta+3}}$ for strip-shaped dampings of the same H\"older regularity. Moreover, we show by example that the  decay rate is optimal. This illustrates the fact that the sharp energy decay rate depends not only on the order of vanishing of the damping, but also on the shape of the damped region. The main ingredient of the proof is the averaging method (normal form reduction) developed by Hitrik and Sj\"ostrand (\cite{Hi1}\cite{Sj}).
\end{abstract}   
	
	\maketitle 
%	\setlength{\parskip}{0.3em}  
%	\tableofcontents

\section{Introduction} 

\subsection{Background}
Let $(M,g)$ be a compact Riemannian manifold with the Beltrami-Laplace operator $\Delta$. Consider the damped wave equation 
\begin{equation}\label{DWave1} 
\begin{cases}
& \partial_t^2u-\Delta u+a(z)\partial_tu=0,\; \text{ in }\R_+\times M,\\
& (u,\partial_tu)|_{t=0}=(u_0,u_1),\; \text{ in }M,
\end{cases}	
\end{equation}
where $a(z)\geq 0$ is the damping. 
The well-posedness of \eqref{DWave1} is a consequence of the Lumer-Philips theorem and the maximal dissipative property of the generator
\begin{align}\label{matrixform} 
\mathcal{L}=\left(\begin{matrix}
0  &\mathrm{Id} \\
\Delta  &-a(z)
\end{matrix}
\right)
\end{align}
on the Hilbert space $\mathcal{H}:=H^1(M)\times L^2(M)$. 
For a solution $(u,\partial_tu)\in H^1(M)\times L^2(M)$, the energy  defined by
$$ E[u](t):=\frac{1}{2}\|\nabla u(t)\|_{L^2(M)}^2+\frac{1}{2}\|\partial_tu(t)\|_{L^2(M)}^2$$is decreasing in time:
$$ \frac{d}{dt}E[u](t)=-\int_{M}a|\partial_tu|^2\leq 0.
$$
A basic question is the decay rate of the energy as $t\rightarrow+\infty$.

It was proved by Rauch-Taylor \cite{RaT} ($\partial M=\emptyset$) and by Bardos-Lebeau-Rauch \cite{BLR} ($\partial M\neq \emptyset$) that, for continuous damping $a\in C(\ov{M})$, if the set  $\omega=\{a>0\}$ verifies the geometric control condition (GCC), then there exists $\alpha_0>0$ such that the uniform stabilization holds:
\begin{align}\label{unifomr} 
 E[u](t)\leq E[u](0)e^{-\alpha_0t},\;\forall t\geq 0.
\end{align}

If (GCC) for $\omega=\{a>0\}$ is not satisfied, there are very few cases that the uniform stabilization \eqref{unifomr} holds (see \cite{BG17} and \cite{Zh})\footnote{If the trapped rays are all grazing on the boundary $\partial\{a>0\}$, the uniform stabilization \eqref{unifomr} can be characterized (see \cite{BG17} and \cite{Zh}) by relative positions of the grazing trapped rays and ${a>0}$.}. Lebeau \cite{Le} constructed examples with arbitrary slowly decaying initial data in the energy space $H^1(M)\times L^2(M)$. Nevertheless, if the initial data is more regular, say in $H^2(M)\times H^1(M)$, the uniform decay rate $1/\log(1+t)$ holds (\cite{Le}). Since then, intensive research activities focus on possible improvement of the logarithmic decay rate for regular initial data, in special geometric settings.

Beyond (GCC), the determination of better decay rate for special manifolds $M$ and special dampings depends on at least the following factors from the  existing literature:
\begin{itemize}
	\item[(a)] The dynamical properties for the geodesic flow of the underlying manifold $M$.
	\item[(b)] The dimension of the trapped rays as well as relative positions between trapped rays and the boundary $\partial\{a>0\}$ of the damped region.
	\item[(c)] Regularity and the vanishing properties of the damping $a$ near $\partial\{a>0\}$.
\end{itemize}
It is known that the energy decay rate is linked to the averaged function along the geodesic flow $\varphi_t$: 
$$ \rho\mapsto \mathcal{A}_T(a)(\rho):=\frac{1}{T}\int_0^Ta\circ\varphi_t(\rho)dt,\quad \rho\in T^*M.
$$
Indeed, (GCC) is equivalent to the lower bound $\mathcal{A}_T(\rho)\geq c_0>0$ for some $T>0$ large enough on the sphere bundle $S^*M$. Roughly speaking, when the geodesic flow is ``unstable'', one may improve the energy decay rate (see \cite{No} for more detailed explanation and references therein). As an illustration of (a), when $M$ is a compact hyperbolic surface, Jin \cite{Ji} shows the exponential energy decay rate for regular data living in $H^{2}(M)\times H^1(M)$. In this direction, we refer also \cite{BuC},\cite{Ch},\cite{CSVW},\cite{Riv} and references therein.

The polynomial decay rate is the intermediate situation between the logarithmic decay rates and the exponential decay rates, exhibited in less chaotic geometry like the flat torus and bounded domains (see \cite{LiR}\cite{BuHi}\cite{Ph07}\cite{AL14} and references therein), where the generalized geodesic flows are unstable. We refer \cite{LLe},\cite{BZu} for other situations of polynomial stabilization, where the undamped region is a submanifold.

We point out that the factor (a) is almost decisive for the observability (and exact controllability) of wave and Schr\"odinger equations.
Comparing with the observability for the wave equation where (GCC) is the only criteria (see \cite{BLR} \cite{BG97}), the stabilization problem is more complicated. Indeed, it was shown in \cite{AL14} (Theorem 2.3) that the observability for the Schr\"odinger semigroup in some time $T>0$ implies automatically that the damped wave is stable at rate $t^{-\frac{1}{2}}$. However, this decay rate is not optimal in general. On the two-dimensional torus, if the damping function is regular enough and vanishing nicely, the decay rate can be very close to $t^{-1}$ ( \cite{BuHi}\cite{AL14}). Even when the damping is the indicator of a vertical (or horizontal) strip, the optimal decay rate is known to be $t^{-\frac{2}{3}}$ (the lower bound was obtained by Nonnenmacher in \cite{AL14} and the upper bound was obtained in \cite{St}). These results provide evidences of factors (b) and (c) mentioned previously.
As explained in \cite{AL14}, the significant difference to the controllability problem is that, there is no general monotonicity property of the type: $a_1\leq a_2$ implies the decay rate associated to $a_2$ is better (or worse) than the decay rate associated to $a_1$. 

In this article, we revisit the polynomial stabilization for wave equations on flat torus. Our main result reveals that, with the same vanishing order, the curvature of the boundary of the damped region also affects the energy decay rate of damped wave equations.

\subsection{The main result}
We concern the polynomial decay rate for \eqref{DWave1} on the two-dimensional flat torus %\footnote{We take the periodic to be $2\pi$ only for simplicity of statement of the main results. The main result remains valid for any general torus $\R^2/(2\pi A\Z\times 2\pi BZ)$, $A,B>0$.}
 $M=\T^2:=\R^2/(2\pi\Z)^2$:
\begin{align}\label{DWaveT2} 
	\begin{cases}
		& \partial_t^2u-\Delta u+a(z)\partial_tu=0,\; \text{ in }\R_+\times \T^2,\\
		& (u,\partial_tu)|_{t=0}=(u_0,u_1),\; \text{ in }\T^2.
	\end{cases}	
\end{align}
To present the main result, we introduce some definitions.
\begin{definition}
	We say that \eqref{DWaveT2} is \emph{stable at rate $t^{-\alpha}$}, if there exists $C>0$, such that all the solution $u$ with initial data $(u_0,u_1)\in \mathcal{H}^2:=H^2(\T^2)\times H^1(\T^2)$ satisfies
$$ (E[u](t))^{\frac{1}{2}}\leq Ct^{-\alpha}\|(u_0,u_1)\|_{\mathcal{H}^2}.
$$
We say that the rate $t^{-\alpha}$ is optimal, if
moreover
$$ \limsup_{t\rightarrow+\infty}t^{\alpha}\!\!\!\!\sup_{0\neq (u_0,u_1)\in\mathcal{H}^2}\frac{(E[u](t))^{\frac{1}{2}}}{\|(u_0,u_1)\|_{\mathcal{H}^2}}>0.
$$
\end{definition} 

Next we introduce the class of damping that we will consider:
\begin{definition}
Let $m,k\in\N$, $\sigma>0$ and $k\sigma<1$. The function class $\mathcal{D}^{m,k,\sigma}(\T^d)$ is defined by:
$$ \mathcal{D}^{m,k,\sigma}:=\{f\in C^m(\T^d): |\partial^{\alpha}f|\lesssim_{\alpha,\sigma} |f|^{1-|\alpha|\sigma},\forall |\alpha|\leq k \}.
$$
\end{definition}
Note that $\mathcal{D}^{m,k,\sigma_1}\subset\mathcal{D}^{m,k,\sigma_2}$, if $\sigma_1<\sigma_2$ and $k\sigma_2<1$. 
This class contains non-negative functions which vanish like H\"older functions. One typical example is 
$$ a_1(z)=b(z)(\max\{0,0.1-|z| \})^{\frac{1}{\sigma}}\in\mathcal{D}^{m,m,\sigma}(\T^d),
$$
where $\sigma<\frac{1}{m}$,  $b\in C^{m}(\T^d)$ and $\inf_{\T^d}b>0$. The associated damped region is $\{z\in\T^d:a_1(z)>0 \}=\{z\in\T^d: |z|<0.1 \}$ is a disc. Another example is the strip-shaped damping
$ a_2(z)=a_2(x)
$ such that $\{z\in\T^2: a_2(z)>0\}:=(-0.1,0.1)_x\times\T_y$ and for some $m\geq 4$,
$$ \frac{d^m}{dx^m}a_2(x)\leq 0 \text{ near } x=0.1 \text{ and } \frac{d^m}{dx^m}a_2(x)\geq 0 \text{ near } x=-0.1.
$$
It was shown in Lemma 3.1 of \cite{BuHi} that $a_2\in \mathcal{D}^{m,m,\frac{1}{m}}(\T^2)$.

For $a\in \mathcal{D}^{m,k,\sigma}$, we denote by $\Sigma_a:=\partial\{a(z)>0\}$.  Let $\T_{A,B}^2:=\R^2/(2\pi A\times 2\pi B)$ be a general flat torus defined via the covering map $\pi_{A,B}:\R^2\mapsto \T_{A,B}^2$.
\begin{definition}
An open set $\omega\subset\T_{A,B}^2$  is said to be \emph{locally strictly convex with positive curvature}, if the boundary of each component of $\pi_{A,B}^{-1}(\omega)\subset\R^2$ is $C^2$ and has strictly positive curvature, as a closed curve in $\R^2$. Sometimes, we also say that the boundary is locally strictly convex.
\end{definition}
Our main result is the following:
\begin{thm}\label{Thm:main}
Let $\beta>4$, $m\geq 10$. Assume that $a\geq 0,$ $ a\in\mathcal{D}^{m,2,\frac{1}{\beta}}$ and the open set $\omega:=\{z\in\T^2: a(z)>0 \}$ is locally strictly convex with positive curvature. Assume that $a(z)$ is locally H\"older of order $\beta$ near $\partial\omega$, in the sense that there exists $R_0>1$, such that
$$ \frac{1}{R_0}\mathrm{dist}(z,\partial\omega)^{\beta}\leq a(z)\leq R_0\mathrm{dist}(z,\partial\omega)^{\beta},\; \text{ for }z\in\omega \text{ near }\partial\omega.
$$
Then the damped wave equation \eqref{DWaveT2} is stable at rate $t^{-1+\frac{2}{2\beta+7}}$.Moreover, the decay rate $t^{-1+\frac{2}{2\beta+7}}$ is optimal in the following sense: there exists a damping $a_0(z)$, satisfying all the hypothesis above with $\beta=m\geq 10$, such that the associated damped wave equation \eqref{DWaveT2} cannot be stable at rate $t^{-1+\frac{2}{2\beta+7}-\epsilon}$, for any $\epsilon>0$.
	
\end{thm}

\begin{rem}
As a comparison, if $a(z)=a(x)$ depends only on one direction (supported on the vertical strip $\omega$) and is locally H\"older of order $\beta$ near $\partial\omega$, the optimal stable rate is $t^{-1+\frac{1}{\beta+3}}$ (see \cite{Kl}\cite{DKl}) which is worse than $t^{-1+\frac{2}{2\beta+7}}$. Our result provides examples that, with the same local H\"older regularity, smaller damped regions better stabilize the wave equation. To the best knowledge of the author, Theorem \ref{Thm:main} also provides the first example where not only the vanishing order of the damping can affect the stable rate, but also the shape of the boundary of the damped region.
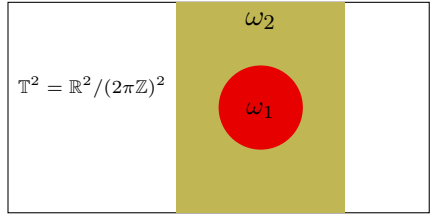
\begin{center} 
	\begin{tikzpicture}[scale=1.4]
		\draw (0,0) rectangle (4,2); 
		\fill[yellow!70!black] (1.6,0)--(3.2,0)--(3.2,2)--(1.6,2);
		\fill[red!90!black] (2.4,1) circle[radius=0.4];
		\draw (2.4,1.65)
		node[above]{$\omega_2$};
		\draw (2.4,0.8) node[above]{\small{$\omega_1$}};
	%	\draw[->>] (0.5,0)--(0.5,0.5);
%		\draw[->>] (0.5,0.5)--(0.5,1);
%		\draw[->>] (0.5,1)--(0.5,1.5);
%		\draw[->>] (0.5,1.5)--(0.5,2);
		\draw (1.8,-1)		node[above]{\small{
		$a_1(z)=(0.1-|z|)_+^{\beta},\; a_2(z)=(0.5-|x|)_+^{\beta}$}};
		\draw (0.8,1)		node[above]{\tiny{$\T^2=\R^2/(2\pi\Z)^2$}};
		\draw (1.8,-0.5)
		node[above]{\small{The damping $a_1$ generates better decay rate than $a_2$  } };
	\end{tikzpicture}
\end{center}
\end{rem}
%\begin{rem}
%	It would be interesting to investigate whether the stable rate $t^{-1+\frac{2}{2\beta+7}}$  in Theorem \ref{Thm:main} is optimal.   
%\end{rem}

\begin{rem}
In the case of rectangular H\"older regular damping $a(z)=a(x)$, the optimal decay rate is shown to be $t^{-1+\frac{1}{\beta+3}}$ (\cite{Kl}\cite{DKl}) for all $\beta\geq 0$. In our result, the regularity assumptions $\beta>4,m\geq 10$ in the first part and $\beta\geq 10$ in the second part of Theorem \ref{Thm:main} might be relaxed, as far as tools of microlocal analysis can still be applied (for example, the paradifferential calculus for low-regularity symbols).
We leave open the problem to determine whether the optimal decay rate $t^{-1+\frac{2}{2\beta+7}}$ in Theorem \ref{Thm:main} can be extended to rough dampings (i.e. small $\beta\geq 0$), as well as the case where $a(z)$ is an indicator function for a strictly convex subset.
\end{rem}

\begin{rem}
	It was shown in \cite{AL14} that, when the damping satisfies $|\nabla a|\leq Ca^{1-\frac{1}{\beta}}$ for large enough $\beta$, then \eqref{DWaveT2} is stable at rate $t^{-1+\frac{4}{\beta+4}}$ (Theorem 2.6 of \cite{AL14}). With an additional assumption on $|\nabla^2 a|\leq Ca^{1-\frac{2}{\beta}}$, our proof of Theorem \ref{Thm:main} essentially provides an alternative proof of this rougher stable rate. Indeed, we need one more condition for $|\nabla^2 a|$ only to perform the normal form reduction in the Section \ref{normalformreduction}. Once reduced to the one-dimensional setting, we are able to apply the same argument of Burq-Hitrik \cite{BuHi}.
\end{rem}

%\begin{rem}
%In \cite{Kl2}, the author considered damping functions that are sums of squares and are invariant in one direction. 
%\end{rem}

\begin{rem}
For the reason of exhibition, we have used a contradiction argument and the notion of semiclassical defect measures in the proof of Theorem \ref{Thm:main}. Comparing to the argument of \cite{AL14}, we do not make use of second semiclassical measures. 
\end{rem}

Finally we give some microlocal interpretation of Theorem \ref{Thm:main}. It is known that the decay rate of the damped wave equation is related to the time average along geodesics (see \cite{No}). As the damping has conormal singularities at the boundary $\partial\omega$ of the damped region, the decay rate depends more precisely on the reflected and transmitted energy of waves concentrated on trapped rays. If along the conormal direction, the damping is more regular, its interaction with transversal free waves
is weaker (analog to the high-low frequency interaction), hence the transmission effect is stronger than the reflection, and consequently, the decay rate is better. When the boundary $\partial\omega$ of the damped region is convex, the average of the damping along any direction gains $\frac{1}{2}$ local H\"older regularity, near the vanishing points along the transversal direction (see Proposition \ref{convexaveraging} for details). This heuristic indicates that, with the same local H\"older regularity near the boundary (of the damped region), the convex-shaped damping  has better stable rate for \eqref{DWaveT2}, than strip-shaped dampings (that are invariant along one direction).

\subsection{Resolvent estimate}

The proof of Theorem \ref{Thm:main} relies on Borichev-Tomilov's criteria of the polynomial semi-group decay rate and the corresponding resolvent estimate for $\mathcal{L}$ given in \eqref{matrixform}:
\begin{prop}[\cite{BoT}]\label{thm:resolvent}
	We have Spec$(\mathcal{L})\cap i\R=\emptyset$. Then, the following statements are equivalent:
\begin{align*}
&(a) \hspace{0.3cm}		\big\|(i\lambda-\mathcal{L})^{-1}\big\|_{\mathcal{L}(\mathcal{H})}\leq C|\lambda|^{\frac{1}{\alpha}} \text{ for all }\lambda\in\R, |\lambda|\geq 1;\\
&(b) \hspace{0.3cm}
\text{The damped wave equation \eqref{DWave1} is stable at rate }t^{-\alpha}. 
\end{align*}
\end{prop}
By Proposition \ref{thm:resolvent} and Proposition 2.4 of \cite{AL14}, the proof of Theorem \eqref{Thm:main} can be reduced to the following semiclassical resolvent estimate:
\begin{thm}\label{thm:resolvent1}
Let $\beta>4$, $m\geq 10$. Assume that $a\geq 0,$ $ a\in\mathcal{D}^{m,2,\frac{1}{\beta}}$ and the open set $\omega:=\{z\in\T^2: a(z)>0 \}$ is locally strictly convex with positive curvature. Assume that $a(z)$ is locally H\"older of order $\beta$ near $\partial\omega$, in the sense that there exists $R_0>1$, such that
$$ \frac{1}{R_0}\mathrm{dist}(z,\partial\omega)^{\beta}\leq a(z)\leq R_0\mathrm{dist}(z,\partial\omega)^{\beta},\; \text{ for }z\in\omega \text{ near }\partial\omega.
$$
Then there exist $h_0\in(0,1)$ and $C_0>0$, such that for all $0<h\leq h_0$,
$$ \|(-h^2\Delta-1+iha(z))^{-1}\|_{\mathcal{L}(L^2(\T^2))}\leq C_0h^{-2-\frac{2}{2\beta+5}}.
$$
Moreover, there exists a damping function $a_0(z)$, satisfying the hypothesis above with $\beta=m\geq 10$, such that
$$ \|(-h^2\Delta-1+iha_0(z))^{-1}\|_{\mathcal{L}(L^2(\T^2))}\geq Ch^{-2-\frac{2}{2\beta+5}}.
$$

\end{thm}
As a comparison, let us recall the main resolvent estimate in \cite{DKl}\cite{Kl}, corresponding to the optimal energy decay rate $t^{-1+\frac{1}{\beta+3}}$ when the damping $a(z)$ depends only on $x$ variable and is locally H\"older of order $\gamma$:
\begin{thm}[\cite{DKl}\cite{Kl}]\label{ProceedingAMS} 
	Let $\gamma\geq 0$. Assume that $W=W(x)\geq 0$ and $\{W>0\}$ is disjoint unions of vertical strips $(\alpha_j,\beta_j)_x\times\T_y$ and $\{W\geq 0\}\neq \T^2$. Assume moreover that for each $j\in\{1,\cdots,l\}$,
	$$ C_1V_j(x)\leq W(x)\leq C_2V_j(x) \text{ on } (\alpha_j,\beta_j),
	$$
	where $V_j(x)>0$ are continuous functions on $(\alpha_j,\beta_j)$, satisfying
	\begin{equation}
		V_j(x)=\left\{
		\begin{aligned}
			&(x-\alpha_j)^{\gamma},\; \alpha_j<x<\frac{3\alpha_j+\beta_j}{4}, \\
			&(\beta_j-x)^{\gamma},\; \frac{\alpha_j+3\beta_j}{4}<x<\beta_j.
		\end{aligned}
		\right.
	\end{equation}
	Then there exist $h_0\in(0,1)$ and $C>0$ such that for all $0<h<h_0$,
	$$ \|(-h^2\Delta-1+ihW(x))^{-1}\|_{\mathcal{L}(L^2)}\leq Ch^{-2-\frac{1}{\gamma+2}}.
	$$
	Furthermore, the above resolvent estimate is optimal, in the sense that there exists quasi-modes $(u_h)_{0<h\leq h_0}$, such that
	$$ \|u_h\|_{L^2}=1,\quad \|(-h^2\Delta-1+ihW(x))u_h\|_{L^2(\T^2)}=O(h^{2+\frac{1}{\gamma+2}}).
	$$
\end{thm}

In the rest of the article, we will prove Theorem \ref{thm:resolvent1}. The proof is based on a contradiction argument. This leads to the fact that the semiclassical measure associated to quasi-modes $(u_h)_{h>0}$ of order $o(h^{2+\frac{2}{2\beta+5}})$ is non-zero along finitely many closed trajectories with periodic directions on the phase space. On the other hand, it turns out that the restriction of semiclassical measure to any periodic direction is zero, which leads to a contradiction. This analysis follows from a second microlocalization procedure and will be achieved in three major steps:
\begin{itemize}
	\item In Section \ref{sec:high}, using the positive commutator method, we show that the semiclassical measure corresponding to the transversal high frequency part of scale $\gtrsim h^{-\frac{1}{2}-\frac{1}{2\beta+5}}$ is zero.
	\item In Section \ref{normalformreduction}, we deal with scales for transversal low frequencies. Using the averaging method, we transfer quasi-modes $(u_h)_{h>0}$ to new quasi-modes $(v_h)_{h>0}$, satisfying new equations that commute with the vertical derivative. This allows us to reduce the problem to the one-dimensional setting. \emph{This is the key part of the proof}, for which we need several elementary properties of the averaging operator, presented in Section \ref{function}.
	\item In Section \ref{sec:1D}, we prove the reduced one-dimensional resolvent estimate (Proposition \ref{1DresolventHolder}). 
\end{itemize}
Furthermore, we prove in Section \ref{lowerbound} the lower bound in Theorem \ref{thm:resolvent1}.
At the end of this article, we add two appendices. In Appendix A, we reproduce the proof of Theorem \ref{ProceedingAMS} in order to be self-contained and to fix some gaps in the paper of \cite{DKl}. In Appendix B, we review several technical results about the semiclassical pseudo-differential calculus, needed in Section \ref{normalformreduction}.

\subsection*{Acknowledgment}
The author is supported by the program: ``Initiative d'excellence Paris
Seine'' of CY Cergy-Paris Universit\'e and the ANR grant ODA (ANR-18-CE40- 0020-01). The author would like to thank anonymous referee's suggestions that help to improve this article. 
%%%%%%%%%%%%%%%%%%%%%%%%%%%%%%%%%%%%%%%%%%%%%%%%%%%%%%%%%%%%%%%%%%%%%%%%%%%%%%%%%%%%%%

%%%%%%%%%%%%%%%%%%%%%%%%%%%%%%%%%%%%%%%%%%%%%%%%%%%%%%%%%%%%%%%%%%%%%%%%%%%%

%%%%%%%%%%%%%%%%%%%%%%%%%%%%%%%%%%%%%%%%%%%%%%%%%%%%%%%%
\section{Contradiction argument and the first microlocalization}

\subsection{A priori estimate and the contradiction argument}
We will adapt basic conventions for notations in the semiclassical analysis (\cite{EZB}). Denote by $P_h=-h^2\Delta-1$ and we fix the parameter $\sigma=\frac{1}{\beta}$ throughout this article. We denote by $\delta_h$ a small parameter such that $\delta_h\rightarrow 0$ as $h\rightarrow 0$. We denote by $\hbar=h^{\frac{1}{2}}\delta_h^{\frac{1}{2}}$ a second semiclassical parameter. For the proof of Theorem \ref{thm:resolvent1}, we fix $\delta_h=h^{\frac{2}{2\beta+5}}.$
Theorem \ref{thm:resolvent1} is the consequence of the following key proposition:
\begin{prop}\label{resolvent1}
Let $u_h$ be a sequence of quasi-modes of width $h^{2}\delta_h$ with $\delta_h=h^{\frac{2}{2\beta+5}}$ i.e.
$$ (P_h+iha)u_h=f_h=o_{L^2}(h^{2}\delta_h).
$$
Then if $u_h=O_{L^2}(1)$, we have $u_h=o_{L^2}(1)$.
\end{prop}
The proof of Proposition \ref{resolvent1} will occupy the rest of this article. We argue by contradiction. First we delete all the zero elements in a given sequence of $u_h$. Then, up to extracting a subsequence and renormalization, we may assume that 
\begin{align}\label{contradiciton}
\|u_h\|_{L^2(\T^2)}=1.
\end{align}
 The following a priori estimate is simple: 
\begin{lem}\label{apriori}
We have the following a priori estimates:
\begin{itemize}
	\item[(a)] $ \|a^{1/2}u_h\|_{L^2}=o(h^{\frac{1}{2}}\delta_h^{\frac{1}{2}})=o(\hbar)$.
	\item[(b)]  $\|h\nabla u_h\|_{L^2}^2-\|u_h\|_{L^2}^2=o(h^{2}\delta_h)$.
\end{itemize}
\end{lem}
\begin{proof}
 Multiplying the equation $(P_h+iha)u_h=f_h$ by $\ov{u}_h$, and integrating by part, we get
 $$ \|h\nabla u_h\|_{L^2}^2-\|u_h\|_{L^2}^2+ih(au_h,u_h)_{L^2}=(f_h,u_h)_{L^2}.
 $$ 
 Taking the imaginary part and the real part, we obtain (a) and (b), with respectively. 
\end{proof}

Since the sequence $(u_h)_{h>0}$ is bounded $L^2(\T^2)$, there exist a subsequence, still denoted as $(u_h)_{h>0}$, and a Radon measure $\mu$ on $T^*\T^2$, such that 
for any symbol $a\in C_c^{\infty}(T^*\T^2)$, there holds
\begin{align}\label{def:measure}
	\lim_{h\rightarrow 0}(\mathrm{Op}_h(a)u_h,u_h)_{L^2}=\langle\mu,a\rangle.
\end{align}
Below, we will denote $\mu$ the semiclassical defect measure associated to this  subsequence $(u_h)_{h>0}$.
 For the proof of this existence of semi-classical measure, one may consult Chapter 5 of \cite{EZB}.

\begin{lem}\label{firstconcentration}
 we have 
 $$\mathrm{supp}(\mu)\{(z,\zeta)\in T^*\T^2: |\zeta|=1\} \text{ and } \mu|_{\omega\times\R^2}=0,$$
 where $\omega=\{z\in\T^2:a>0\}$.
\end{lem}
\begin{proof}
This property follows from the standard elliptic regularity which only requires quasi-mode for $P_h$ of order $O_{L^2}(h)$. The damping term $ihau_h$ can be roughly treated as an error $O_{L^2}(h)$. For example, one can consult Theorem 5.4 of \cite{EZB} for a proof.
\end{proof}
Let $\varphi_t$ be the geodesic flow on $T^*\T^2$. We recall the following invariant property of the semiclassical measure:
\begin{lem}\label{firstpropagation} 
The semiclassical measure $\mu$ is invariant by the flow $\varphi_t$, i.e.
$$ \varphi_t^*\mu=\mu.
$$
\end{lem}
\begin{proof}
This property holds true for quasi-mode of $P_h$ of order $o_{L^2}(h)$. From (a) of Lemma \ref{apriori}, we have $f_h-ihau_h=o_{L^2}(h)$. The proof then follows from a standard propagation argument (see for example Theorem 5.5 of \cite{EZB}). 
\end{proof}

\subsection{Reducing to periodic trapped directions}\label{changingcoordinate} 

In order to perform the finer analysis near trapped rays, we need to do a change of coordinate, following \cite{BZ4} and \cite{AL14}. 
The spirit of the second-microlocalization here is also close to the work \cite{AM14}.

 By identifying $\T^2=\R^2/(2\pi\Z)^2$, we decompose $\mathbb{S}^1$ as rational directions
$$ \mathcal{Q}:=\{\zeta\in\mathbb{S}^1: \zeta=\frac{(p,q)}{\sqrt{p^2+q^2}},\; (p,q)\in\Z^2,\; \mathrm{gcd}(p,q)=1 \}
$$
and irrational directions $\mathcal{R}:=\mathbb{S}^1\setminus\mathcal{Q}$. Since the orbit of an irrational direction is dense, by Lemma \ref{firstconcentration} and Lemma \ref{firstpropagation}, we have $$\mu=\mu|_{\T^2\times\mathcal{Q}}=\sum_{\zeta_0\in\mathcal{Q}}\mu_{\zeta_0}.$$
It suffices to show that, for each $\zeta_0=\frac{(p_0,q_0)}{\sqrt{p_0^2+q_0^2}} \in\mathcal{Q}$, the restricted measure $\mu_{\zeta_0}$ is zero\footnote{In fact, we only need to consider finitely many $\zeta_0\in\mathcal{Q}$, since when $p_0^2+q_0^2$ is large enough, the associated periodic direction is close to an irrational direction and the trajectory will eventually enter $\omega$.  }. Denote by $\Lambda_0$, the rank 1 submodule of $\Z^2$ generated by $\Xi_0=(p_0,q_0)$. Denote by $$\Lambda_0^{\perp}:=\{\zeta\in\R^2: \zeta\cdot \Xi_0=0 \}
$$
the dual of the submodule $\Lambda_0$. Denote by
$$ \T_{\Xi_0}^2:=(\R\Lambda_0/(2\pi\Lambda_0))\times (\Lambda_0^{\perp}/(2\pi\Z)^2\cap\Lambda_0^{\perp}).
$$
Then we have a natural smooth covering map 
$ \pi_{\Xi_0}:\T_{\Xi_0}^2\mapsto \T^2
$
of degree $p_0^2+q_0^2$. The pullback of a $2\pi\times 2\pi$ periodic function $f$ satisfies
$$ (\pi_{\Xi_0}^*f)(X+k\tau,Y+l\tau)=(\pi_{\Xi_0}^*f)(X,Y),\quad k,l\in\Z, (X,Y)\in\R^2,
$$
where
$ \tau=2\pi\sqrt{p_0^2+q_0^2}.
$ 
	
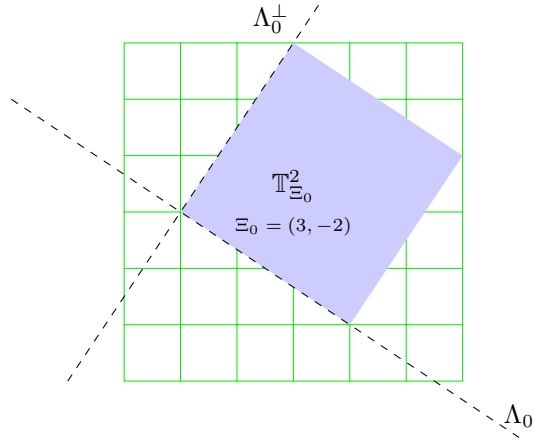
\begin{center} 
	\begin{tikzpicture}[scale=1.5]
	\draw[step=0.5,help lines,green!90!black] (-1.5,-1.5) grid (1.5, 1.5); 
	\fill[blue!20] (-1,0)--(0,1.5)--(1.5,0.5)--(0.5,-1);
	\draw[dashed]
	(-2.5,1)--(2,-2);
	\draw (0,0)
	node[above]{\small{$\T_{\Xi_0}^2$} };
	\draw (0,-0.3)
	node[above]{\tiny{$\Xi_0=(3,-2)$}};
	\draw (2,-2)
	node[above]{\small{$\Lambda_0$}};
	\draw[dashed]
	(-2,-1.5)--(0.25,1.875);
	\draw (-0.2,1.5)
	node[above]{\small{$\Lambda_0^{\perp}$ }};
	\end{tikzpicture}
\end{center}

By pulling back to the torus $\T_{\Xi_0}^2$, we can identify the sequence $(u_h)\subset L^2(\T^2)$ as $(\pi_{\Xi_0}^{*}u_h) \subset L^2(\T_{\Xi_0}^2)$ and in this new coordinate system, $\zeta_0=\frac{\Xi_0}{|\Xi_0|}=(0,1)$. The semi-classical defect measure $\mu$ on $T^*\T^2$ is the pushforward of the semi-classical measure associated to $(\pi_{\Xi_0}^{*}u_h)$. 
Since the period of the torus $\T_{\Xi_0}^2$ has no influence of the analysis in the sequel\footnote{As we fix one periodic direction $\zeta_0$ and consider the semi-classical limit $h\rightarrow 0$, one does not need to worry about the fact that the period $2\pi\sqrt{p_0^2+q_0^2}$ may be very large. }, we will still use the notation $\T^2$ to stand for $\T_{\Xi_0}^2$, the variables $z=(x,y), \zeta=(\xi,\eta)$ to stand for variables $Z=(X,Y), \Xi$ on $T^*T_{\Xi_0}^2$, and assuming the period to be $2\pi$ for simplicity. The only thing that will change is that the pre-image of the damping $\pi_{\Xi_0}^{-1}(\omega)$ is now a disjoint union of $p_0^2+q_0^2$ copies of $\omega$ on $\T_{\Xi_0}^2$, and each component is still strictly convex with positive curvature. For this reason, in the hypothesis of Proposition \ref{convexaveraging}, we assume that the boundary of $\{a>0\}$ is made of disjoint unions of strictly convex curves.

	%	\draw (0,0)
	%	node[above]{$\omega_0$};
	%	\draw (2.4,0.8) node[above]{\small{supp$(a)$}};
	%	\draw[->>] (0.5,0)--(0.5,0.5);
	%	\draw[->>] (0.5,0.5)--(0.5,1);
	%	\draw[->>] (0.5,1)--(0.5,1.5);
	%	\draw[->>] (0.5,1.5)--(0.5,2);
	%	\draw (1.2,1)
	%	node[above]{\small{$\zeta_0=(0,1)$}};

%%%%%%%%%%%%%%%%%%%%%%%%%%%%%%%%%%%%%%%%%%%%%%%%%%%%%%%%%%%%%%%%%%

\section{Analysis of the transversal high frequencies}\label{sec:high}

 Recall that $\zeta_0=(0,1)$, and our goal is to show that $\mu\mathbf{1}_{\zeta=\zeta_0}=0$. In this section, we deal with the transversal high frequencies of size $\mathcal{O}(\hbar^{-1})$ and use the positive commutator method to show these portions are propagated into the flowout of the damped region $\omega$.

 For the geodesic flow $\varphi_t$ on $T^*\T^2$ and $\zeta\in\mathbb{S}^1$, we denote by $\varphi_t(\cdot,\zeta): \T^2\rightarrow \T^2$ the projection of the flow map $\varphi_t$.
 By shifting the coordinate, we may assume $$\omega_0:=I_0\times\T_y\subset \bigcup_{t\in[0,2\pi]}\varphi_t(\cdot,\zeta_0)(\omega),
 $$
 where $$I_0= (-\sigma_0,\sigma_0)\subset \pi_1(\{z:a(z)\geq c_0>0\}),\; \text{ for some $\sigma_0<\frac{\pi}{100}$}
 $$ and $\pi_1:\T^2\mapsto \T_x$ the canonical projection. Therefore, there exist $\epsilon_0>0, c_0>0$ sufficiently small and $T_0>0$, such that for any $|\zeta|=1, z_0\in\omega_0, |\zeta-\zeta_0|\leq \epsilon_0$,
 \begin{align}\label{GCCverticle} 
  \int_0^{T_0}(a\circ\varphi_t)(z_0,\zeta)dt\geq c_0>0.
 \end{align}
\begin{center} 
	\begin{tikzpicture}[scale=1.5]
\draw (0,0) rectangle (4,2); 
\fill[green!70!black] (2,0)--(2.8,0)--(2.8,2)--(2,2);
\fill[yellow!90!black] (2.4,1) circle[radius=0.7];
\draw (2.4,1.65)
node[above]{$\omega_0$};
\draw (2.4,0.8) node[above]{\small{supp$(a)$}};
\draw[->>] (0.5,0)--(0.5,0.5);
\draw[->>] (0.5,0.5)--(0.5,1);
\draw[->>] (0.5,1)--(0.5,1.5);
\draw[->>] (0.5,1.5)--(0.5,2);
\draw (1.2,1)
node[above]{\small{$\zeta_0=(0,1)$}};
\end{tikzpicture}
\end{center}
% To simplify the notation, for any function $b$, we denote by
%$$ \mathcal{A} (b)(x)
%=\mathcal{A} (b)_{\zeta_0}(x)
%=\frac{1}{2\pi}\int_{-\pi}^{\pi}b(x,y)dy.
%$$

To microlocalize the solution near $\zeta_0$, we pick $\psi_0\in C_c^{\infty}(\R)$ and consider $u_h^1:=\psi_0\big(\frac{hD_x}{\epsilon_0}\big)u_h$, then
$$ (P_{h}+iha)u_h^{1}=f_h^1:=\psi_0\big(\frac{hD_x}{\epsilon_0}\big)f_h-ih\big[\psi_0\big(\frac{hD_x}{\epsilon_0}\big),a\big]u_h.
$$
Let $\mu_1$ be the semiclassical measure of $u_h^1$, then $\mu_1=|\psi\big(\frac{\xi}{\epsilon_0}\big)|^2\mu$.
\begin{lem}\label{u_h^1}
We have
$$ \|a^{\frac{1}{2}}u_h^1\|_{L^2}=o(h^{\frac{1}{2}}\delta_h^{\frac{1}{2}}),\quad \|f_h^1\|_{L^2}=o(h^{2}\delta_h).
$$
\end{lem}
\begin{proof}
It suffices to show that $f_h^1=o_{L^2}(h^{2}\delta_h)$, and the first assertion follows from the same proof of (a) of Lemma \ref{apriori}. By the symbolic calculus,
$$ i\big[\psi_0\big(\frac{hD_x}{\epsilon_0}\big),a\big]=\epsilon_0^{-1}h\mathrm{Op}_h\big(\psi_0'\big(\frac{\xi}{\epsilon_0}\big)\partial_xa \big)+O_{L^2}(\epsilon_0^{-2}h^2).
$$ 
From the pointwise inequality for the non-negative function $a$:
\begin{align}\label{pointwise} 
|\nabla a(x)|^2\leq 2\|a\|_{W^{2,\infty}}a(x),
\end{align}
we have, for some $C>0$,
$$ Ca-\big|\epsilon_0^{-1}\psi_0'\big(\frac{\xi}{\epsilon_0}\big)\partial_xa\big|^2\geq 0 \text{ on }T^*\T^2.
$$
Therefore, by the sharp G$\mathring{\mathrm{a}}$rding inequality, we get
$$
 \big\|\epsilon_0^{-1}\mathrm{Op}_h\big(\psi_0'\big(\frac{\xi}{\epsilon_0}\big)\partial_xa \big)u_h\big\|_{L^2}\leq C|(au_h,u_h)_{L^2}|^{\frac{1}{2}}+Ch^{\frac{1}{2}}\|u_h\|_{L^2}.
$$
Together with (a) of Lemma \ref{apriori}, this implies that
$$ \big\|ih\big[\psi_0\big(\frac{hD_x}{\epsilon_0}\big),a\big]u_h\big\|_{L^2}=o(h^{\frac{5}{2}})=o(h^2\delta_h).
$$
The proof of Lemma \ref{u_h^1} is complete.
\end{proof}

Recall that $\hbar=h^{\frac{1}{2}}\delta_h^{\frac{1}{2}}$. Let $\psi\in C_c^{\infty}(\R)$, and consider 
\begin{align}\label{vhwh}
v_h=\psi(\hbar D_x)u_h^1,\quad  w_h=(1-\psi(\hbar D_x))u_h^1.\end{align} 
In this decomposition, $w_h$ corresponds to the transversal high frequency part, while $v_h$ corresponds to the transversal low frequency part for which will be treated in next sections.
Note that
\begin{align}\label{eq:vh}
(P_h+iha)v_h=\psi(\hbar D_x)f_h^1-ih[\psi(\hbar D_x),a]u_h^1=:r_{1,h}
\end{align} 
and
\begin{align}\label{eq:w_h}
(P_h+iha)w_h=(1-\psi(\hbar D_x))f_h^1+ih[\psi(\hbar D_x),a]u_h^1=:r_{2,h}.
\end{align}
We need to show that the commutator term $h[\psi(\hbar D_x),a]u_h^1$ can be viewed as the remainder:
\begin{lem}\label{newerro} 
We have
$$ \|r_{1,h}\|_{L^2}+\|r_{2,h}\|_{L^2}=o(h^{2}\delta_h)=o(h\hbar^2).
$$
Consequently, from Lemma \ref{apriori},
$$  \|a^{\frac{1}{2}}v_h\|_{L^2}+\|a^{\frac{1}{2}}w_h\|_{L^2}=o(h^{\frac{1}{2}}\delta_h^{\frac{1}{2}})=o(\hbar).
$$
\end{lem}
\begin{proof}
 According to the symbolic calculus,
$$ i[\psi(\hbar D_x),a]=\hbar\mathrm{Op}_{\hbar}(\psi'(\xi)\partial_xa)+C\hbar^2\mathrm{Op}_{\hbar}(\partial_{\xi}^2\psi\cdot\partial_x^2a)+\mathcal{O}_{\mathcal{L}(L^2)}(\hbar^3).
$$
Using the fact that $a\in \mathcal{D}^{m,2,\sigma}$ and $\sigma<\frac{1}{4}$, we have $a^{\frac{1}{2}}\in C_c^2(\T^2)$. Applying the special symbolic calculus Lemma \ref{symbolicspecial} (b) with $\kappa=\partial_x(a^{\frac{1}{2}}), b_2=a^{\frac{1}{2}}$ and $\varphi=\psi'$, we have
\begin{align*} \frac{1}{2}\mathrm{Op}_{\hbar}(\psi'(\xi)\partial_xa)=&\mathrm{Op}_{\hbar}(\psi'(\xi)\partial_x(a^{\frac{1}{2}}) )a^{\frac{1}{2}}-\frac{1}{i}\hbar\mathrm{Op}_{\hbar}(\psi''(\xi)\partial_x(a^{\frac{1}{2}})\cdot \partial_x(a^{\frac{1}{2}}))+\mathcal{O}_{\mathcal{L}(L^2)}(\hbar^2).
\end{align*}
Applying Lemma \ref{symbolicspecial} (a) with $\kappa=b_1=\partial_x(a^{\frac{1}{2}}), \varphi=\psi''$, we have
\begin{align*}
-\frac{\hbar}{i}\mathrm{Op}_{\hbar}(\psi''(\xi)\partial_x(a^{\frac{1}{2}})\cdot \partial_x(a^{\frac{1}{2}}) )=&-\frac{\hbar}{i}\mathrm{Op}_{\hbar}(\psi''(\xi)\partial_x(a^{\frac{1}{2}}) )\partial_x(a^{\frac{1}{2}})+\mathcal{O}_{\mathcal{L}(L^2)
}(\hbar^2).
\end{align*}

Since $\psi$ is only a function of $\xi$ and $a^{-\frac{1}{2}}\partial_xa\in L^{\infty}$, we have
$$ \big\|\mathrm{Op}_{\hbar}\big(\psi'(\xi)\frac{\partial_x a}{a^{\frac{1}{2}}}\big)a^{\frac{1}{2}}u_h^1\big\|_{L^2(\T^2)}=\big\|(a^{-\frac{1}{2}}\partial_xa)\psi'(hD_x)(a^{\frac{1}{2}}u_h^1) \big\|_{L^2(\T^2)}
\leq C\|a^{\frac{1}{2}}u_h^1\|_{L^2}=o(\hbar).
$$ 
 
For the term $\mathrm{Op}_{\hbar}(\partial_{\xi}^2\psi\partial_x^2a)$, since $|\partial_x^2a|\lesssim a^{1-2\sigma}\lesssim a^{\frac{1}{2}}$, by the sharp G$\mathring{\mathrm{a}}$rding inequality, 
$$ \Re\big(\mathrm{Op}_{\hbar}(Ca-|\psi''(\xi)\partial_x^2a|^2)u_h^1,u_h^1\big)_{L^2}\geq -C\hbar\|u_h^1\|_{L^2}^2.
$$
Thus $\|\hbar^2\mathrm{Op}_{\hbar}(\psi''(\xi)\partial_x^2a)u_h^1\|_{L^2}=O(\hbar^{\frac{5}{2}})$.
Therefore, by Calder\'on-Vaillancourt (Theorem \ref{CalderonVaillancourt}), we can write
$$ ih[\psi(\hbar D_x),a]=h\hbar A_{\hbar}a^{\frac{1}{2}}+h\hbar^2B_{\hbar}\partial_x(a^{\frac{1}{2}})
+\mathcal{O}_{\mathcal{L}(L^2)}(h\hbar^3),
$$
with $A_{\hbar}, B_{\hbar}$ bounded operators on $L^2$, uniformly in $\hbar$. Since $|\partial_x(a^{\frac{1}{2}})|\lesssim a^{\frac{1}{2}-\sigma}$, by (a) of Lemma \ref{apriori} and the interpolation, we get
$$ \|ih[\psi(\hbar D_x),a]u_h^1\|_{L^2}=o(h\hbar^2)+O(h\hbar^3).
$$
This completes the proof of Lemma \ref{newerro}.
\end{proof}

\begin{rem}
Compared to \cite{AL14} where the damping only satisfies $a\in W^{k_0,\infty}(\T^2)$ and $|\nabla a|\lesssim a^{1-\sigma}$, our assumption $a\in\mathcal{D}^{m,2,\sigma}$ is slightly stronger, in order to ensure the commutator of the damping term is still a remainder. 
Indeed, here we chose $\hbar=h^{\frac{1}{2}}\delta_h^{\frac{1}{2}}\ll h^{\frac{1}{2}}$ while in \cite{AL14}, the authors chose $\hbar=h^{\alpha},\alpha<\frac{1}{3}$ (see the sentence after Proposition 7.2 of \cite{AL14}). The use of the sharp G$\mathring{\mathrm{a}}$rding as in \cite{AL14} would not get $o(h\hbar^2)$ for the remainders $r_{1,h},r_{2,h}$. We also remark that a direct application of the Calder\'on-Vaillancourt theorem in the symbolic calculus requires $a^{\frac{1}{2}}\in W^{3,\infty}$. Since we assume only $a\in\mathcal{D}^{m,2,\sigma}$, (thus $a^{\frac{1}{2}}\in W^{2,\infty}$), we need to exploit the special structure of the commutator $[\psi(\hbar D_x),a]$ and apply the special symbolic calculus Lemma \ref{symbolicspecial}. 
\end{rem}
Recall that $\omega_0=I_0\times\T_y$. 
\begin{lem}\label{yellow-green}
We have
$$ \|w_h\mathbf{1}_{\omega_0}\|_{L^2}+\|h\nabla w_h\mathbf{1}_{\omega_0}\|_{L^2}=O(h^{\frac{1}{2}}),\text{ as } h\rightarrow 0.
$$
\end{lem}
\begin{proof}
The proof follows from the classical propagation argument, using the geometric control condition. Take small intervals $I_0'\subset\T_x, I_1=(\sigma_1,\sigma_2)\subset\T_y$, such that $I_0\subset I_0'$ and  $\omega_1=I_0'\times I_1\subset \{a\geq \delta_0 \}$ for some $\delta_0>0$. For any $z_0=(x_0,y_0)\in \omega_0$, by the geometric control condition, there exist $T_1>0, \delta_1>0, \delta_2>0$, and the small neighborhood  $U=(x_0-\delta_1,x_0+\delta_1)\times (y_0-\delta_1,y_0+\delta_1)$ of $z_0$, such that for all  $|\zeta-\zeta_0|\leq \epsilon_0, z\in U$, we have
$$ z+s\zeta \in\omega_1,\quad s\in[T_1-\delta_2,T_1+\delta_2].
$$
In particular, $a(z+s\zeta)\geq \delta_0$.
Without loss of generality, we assume that $\pi>\sigma_2>\sigma_1>y_0+\delta_1>y_0-\delta_1>-\pi$.
Pick two cutoffs $\chi_1(x),\chi_2(y)\geq 0$, supported in $(x_0-\delta_1,x_0+\delta_1)$, $(y_0-\delta_1,y_0+\delta_1)$  and equal to $1$ on $(x_0-\delta_1/2,x_0+\delta_1/2),(y_0-\delta_1/2,y_0+\delta_1/2)$, respectively.
% Take $\epsilon_1<\min\{\epsilon_0,\frac{1}{2T_1} \}$ such that
%$$ \chi_1'(x-s\xi)=0,\; \text{ for all }|\xi|\leq \epsilon_1,\; |s|\leq T_1+\delta_2.
%$$
%whenever $|x-x_0|\leq \delta_1/2$.
 Let $\chi_0\in C_c^{\infty}(\R)$ be a cutoff near $|\zeta-\zeta_0|\leq \epsilon_0$. For any $s\geq 0$, define the symbol
$$ b_s(z,\zeta):=\chi_0(\zeta)\cdot(\chi_1\otimes\chi_2)\circ\varphi_{-s}(z,\zeta) =\chi_0(\zeta)\chi_1(x-s\xi)\chi_2(y-s\eta).
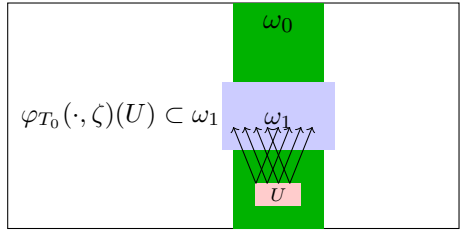
$$
\begin{center} 
	\begin{tikzpicture}[scale=1.5]
	\draw (0,0) rectangle (4,2); 
	\fill[green!70!black] (2,0)--(2.8,0)--(2.8,2)--(2,2);
	%	\fill[yellow!90!black] (2.4,1) circle[radius=0.65];
	\fill[blue!20]
	(1.9,1.3)--(2.9,1.3)--(2.9,0.7)--(1.9,0.7);
	\draw (2.4,1.65)
	node[above]{$\omega_0$};
	\draw (2.4,0.8) node[above]{\small{$\omega_1$}};
	\fill[red!20]
	(2.2,0.4)--(2.6,0.4)--(2.6,0.2)--(2.2,0.2);
	\draw (2.4,0.15)
	node[above]{\tiny{$U$}};
	\draw[->] (2.2,0.4)--(2.0,0.9);
	\draw[->]
	(2.3,0.4)--(2.1,0.9);
	\draw[->]
	(2.4,0.4)--(2.2,0.9);
	\draw[->]
	(2.5,0.4)--(2.3,0.9);
	\draw[->] (2.2,0.4)--(2.4,0.9);
	\draw[->]
	(2.3,0.4)--(2.5,0.9);
	\draw[->]
	(2.4,0.4)--(2.6,0.9);
	\draw[->]
	(2.5,0.4)--(2.7,0.9);
	\draw (1.0,0.8) node[above]{\small{$\varphi_{T_0}(\cdot,\zeta)(U)\subset \omega_1$}};
	%	\draw[->>] (0.5,0)--(0.5,0.5);
	%	\draw[->>] (0.5,0.5)--(0.5,1);
	%	\draw[->>] (0.5,1)--(0.5,1.5);
	%	\draw[->>] (0.5,1.5)--(0.5,2);
	%	\draw (1.2,1)
	%	node[above]{\small{$\zeta_0=(0,1)$}};
	\end{tikzpicture}
\end{center}
Direct computation yields
\begin{align*}
\frac{d}{ds}(\mathrm{Op}_h(b_s)w_h,w_h)_{L^2}=&(\mathrm{Op}_h(\partial_sb_s)w_h,w_h)_{L^2}=-(\mathrm{Op}_h(\zeta\cdot\nabla_zb_s )w_h,w_h)_{L^2}.
\end{align*}
Integrating this equality from $s=0$ to $s=T_0$,
\begin{align}\label{eq4.1}
(\mathrm{Op}_h(b_{T_1})w_h,w_h)_{L^2}-(\mathrm{Op}_h(b_0)w_h,w_h)_{L^2}=-\int_0^{T_1}
(\mathrm{Op}_h(\zeta\cdot\nabla_zb_s )w_h,w_h)_{L^2}ds.
\end{align}
Note that for fixed $s\in[0,T_1]$,
$$ \frac{i}{h}[P_h,\mathrm{Op}_h(b_s)]=2\mathrm{Op}_h(\zeta\cdot\nabla_zb_s)+\mathcal{O}_{\mathcal{L}(L^2)}(h),
$$
we have
\begin{align}\label{eq4.2}
(\mathrm{Op}_h(b_0)w_h,w_h)_{L^2}=&(\mathrm{Op}_h(b_{T_1})w_h,w_h)_{L^2}+\frac{i}{2h}\int_0^{T_1}([P_h,\mathrm{Op}_h(b_s)]w_h,w_h)_{L^2}+O(h).
\end{align}
Using the equation
$$ P_hw_h=r_{2,h}-ihaw_h,
$$
we have
\begin{align}\label{eq4.3}
 \frac{1}{h}([P_h,\mathrm{Op}_h(b_s)]w_h,w_h)_{L^2}=&\frac{2i}{h}\Im(\mathrm{Op}_h(b_s)w_h,r_{2,h}-ihaw_h)_{L^2}\notag\\=&o(h^{1+\delta})+O(1)\|a^{\frac{1}{2}}w_h\|_{L^2}\|a^{\frac{1}{2}}\mathrm{Op}_h(b_s)w_h\|_{L^2}\notag \\
 =&O(h),
\end{align}
where to the last step, we write
$$ a^{\frac{1}{2}}\mathrm{Op}_h(b_s)=\mathrm{Op}_h(b_s)a^{\frac{1}{2}}+[a^{\frac{1}{2}},\mathrm{Op}_h(b_s)]
$$
and use the last assertion of Lemma \ref{newerro}, as well as the symbolic calculus.  

Finally, from the support property of $b_0(z)=\chi_0(\zeta)\chi_1(x)\chi_2(y)$, we have
$$ a(z)\chi_0(\zeta)\geq \delta_0 b_{T_1}(z,\zeta).
$$
Thus by the sharp G$\mathring{\mathrm{a}}$rding inequality,
$$ |(\mathrm{Op}_h(b_{T_1})w_h,w_h)_{L^2}|\leq \delta_0^{-1}|(\mathrm{Op}_h(a(z)\chi_0(\zeta))w_h,w_h)_{L^2}|+O(h)\leq C\delta_0^{-1}\|a^{\frac{1}{2}}w_h\|_{L^2}^2+O(h).
$$
Combining this with \eqref{eq4.2},\eqref{eq4.3} and the last assertion of Lemma \ref{newerro}, we deduce that
$\|w_h\mathbf{1}_{U}\|_{L^2}=O(h^{\frac{1}{2}})$, $\|h\nabla w_h\mathbf{1}_{U}\|_{L^2}=O(h^{\frac{1}{2}})$. By the partition of unity of $\omega_0$,
we complete the proof of Lemma \ref{yellow-green}.
%When $\alpha<1,$ the operation $(1-\psi(h^{\alpha}D_x))$ does not affect the semiclassical measure $\mu_1$ of $u_h^1$. Since $\mu_1$ is invariant by the geodesic flow $\varphi_t$, by \eqref{GCCverticle}, we have
%$$ \frac{1}{T_0}\langle\mu_1,\mathbf{1}_{\omega_0}\rangle\leq \frac{1}{c_0T_0}\int_0^{T_0}\langle\mu_1, a\circ\varphi_t\rangle dt=\frac{1}{c_0T_0}\int_0^{T_0}\langle\mu_1,a\rangle dt=0,
%$$
%where we use the invariance $\varphi_t^*\mu_1=\mu_1$ and $\|au_h^1\|_{L^2}$ in the last step. This completes the proof of Lemma \ref{yellow-green}.
\end{proof}
Now we are ready to prove the main result in this section, that is the transversal high frequency part is of order  $o_{L^2}(1)$:
\begin{prop}\label{highDy} 
We have $\|w_h\|_{L^2}=O(\delta_h)=O(\hbar h^{-\frac{1}{2}})$, as $h\rightarrow 0$.
\end{prop}
\begin{proof}
We use the positive commutator method to detect the transversal propagation, similarly as in \cite{BuSun}. Recall that $\omega_0=I_0\times \T$ and $I_0=(-\sigma_0,\sigma_0)$, $\sigma_0<\frac{\pi}{100}$. Take $\phi=\phi(x)\in C^{\infty}(\T_x;[0,1])$ such that:
$$\mathrm{supp}(1-\phi)\subset I_0,\quad \phi\equiv 0 \text{ near} [-\frac{\sigma_0}{2},\frac{\sigma_0}{2}],\quad \mathrm{supp}(\phi')\subset I_0.
$$
Denote by 
$$ X(x):=(x+\pi)\mathbf{1}_{-\pi\leq x<-\frac{\sigma_0}{2}}+(x-\pi)\mathbf{1}_{\frac{\sigma_0}{2}\leq x<\pi},
$$
then $\phi(x)X\partial_x$ is well-defined smooth vector field on $\T^2$.
We now compute the inner product
$$ ([P_h,\phi(x)X\partial_x]w_h,w_h)_{L^2}
$$
in two ways. On the one hand, from the commutator relation
$$ [P_h,\phi(x)X\partial_x]=-2(\phi X)'h^2\partial_x^2-h^2(\phi X)''\partial_x,
$$
we have
\begin{align}\label{pfeq:1}
([P_h,\phi(x)X\partial_x]w_h,w_h)_{L^2}\geq &2(\phi(x)h\partial_xw_h,h\partial_xw_h)_{L^2}-C\|(\phi'(x))^{1/2}h\partial_xw_h\|_{L^2}^2\notag \\-&Ch\|h\partial_xw_h\|_{L^2}\|w_h\|_{L^2}.
\end{align}
On the other hand, using the equation \eqref{eq:w_h}, we have
\begin{align}\label{pfeq:2}
([P_h,\phi(x)X\partial_x]w_h,w_h)_{L^2}=& (\phi(x)X\partial_xw_h,r_{2,h}-ihaw_h)_{L^2}-(\phi(x)X\partial_x(r_{2,h}-ihaw_h),w_h)_{L^2}\notag \\
\leq & \frac{C}{h}(\|h\partial_xw_h\|_{L^2}+h\|w_h\|_{L^2})\|r_{2,h}\|_{L^2}+C\|a^{1/2}h\partial_xw_h\|_{L^2}\|a^{\frac{1}{2}}w_h\|_{L^2}+C\|a^{\frac{1}{2}}w_h\|_{L^2}^2.
\end{align} 
Combining \eqref{pfeq:1} and \eqref{pfeq:2} and Lemma \ref{yellow-green}, we get
\begin{align*}
\|h\partial_xw_h\|_{L^2}^2\leq & C\|h\partial_xw_h\mathbf{1}_{\omega_0}\|_{L^2}^2+ C\|a^{1/2}h\partial_xw_h\|_{L^2}^2+C\|a^{\frac{1}{2}}w_h\|_{L^2}^2+\frac{C}{h^2}\|r_{2,h}\|_{L^2}^2+Ch^2\|w_h\|_{L^2}^2\\
\leq &O(h)+o(h\delta_h).
\end{align*}
In particular, by the definition of $w_h$, we have
$$C(h\hbar^{-1})\|w_h\|_{L^2}\leq \|h\partial_xw_h\|_{L^2}\leq O(h^{\frac{1}{2}}),
$$
and this completes the proof of Proposition \ref{highDy}.
\end{proof}

%%%%%%%%%%%%%%%%%%%%%%%%%%%%%%%%%%%%%%%%%%%%%%%%%%%%%%%

\section{The averaging properties of functions}\label{function}

In order to treat the transversal low frequency part $v_h=\psi(\hbar D_x)u_h^{1}$ of \eqref{vhwh}, we will adapt the averaging argument of Sj\"ostrand (\cite{Sj}) and Hitrik (\cite{Hi1}) to average the operator $P_h+iha$ along the trapped direction (it worth also mentioning a series work of Hitrik-Sj\"ostrand (\cite{HS1}\cite{HS2}\cite{HS3}) that a related averaging procedure was used to describe the spectrum of the damped wave operator).
 This amounts to understand the regularity of averaged functions.   
The goal of this section is to establish several properties of the averaging operator which will be used in Section \ref{normalformreduction} for normal form reductions. More importantly, we will prove the key geometric proposition (Proposition \ref{convexaveraging}) for convex-shaped damping that is responsible for the improvement of the resolvent estimate.
 
Given a direction $v\in\mathbb{S}^1$, we say that $v$ is \emph{periodic}, if $v=(\xi,\eta)$ and $\xi,\eta$ are $\mathbb{Q}$-linearly dependent. Otherwise, we say that $v$ is \emph{ergodic}\footnote{These definitions stemmed from the fact that we are on the two-dimensional torus $\T^2$.}. We define the averaging operator along $v$:
$$ f\mapsto \mathcal{A}(f)_v(z):=\lim_{T\rightarrow\infty}\frac{1}{T}\int_0^Tf(z+tv)dt,
$$
where the limit exists, thanks to Weyl's equidistribution theorem.  
\begin{lem}\label{Jensen} 
	Assume that $v\in\mathbb{S}^1$. Then for any non-negative function $f\in \mathcal{D}^{m,k,\sigma}(\T^2)$, $m\geq 10$, the averaged function $\mathcal{A}(f)_v \in\mathcal{D}^{m,k,\sigma}(\T^2)$.
\end{lem}
\begin{proof}
	First we assume that
	$v=(\xi,\eta)$ is periodic. This implies that the orbit $z\mapsto z+tv$ is periodic, and $T_v$ is the period. Clearly, for any $f\in \mathcal{D}^{m,k,\sigma}$,
	$$ \mathcal{A} (f)_{v}(z)=\frac{1}{T_v}\int_0^{T_v}f(z+tv)dt.
	$$ 
	Since the function $s\mapsto |s|^{1-|\alpha|\sigma}$ is concave, by Jensen's inequality we have
	\begin{align}\label{Jensen1} \frac{1}{T_v}\int_0^{T_v}|f(z+tv)|^{1-|\alpha|\sigma}dt\leq \Big(\frac{1}{T_v}\int_0^{T_v}f(z+tv)dt\Big)^{1-|\alpha|\sigma}.
	\end{align}
	Indeed, if $\int_0^{T_v}f(z+tv)dt=0$, then $f(z+tv)\equiv 0$ for all $t\in[0,T_v]$, and the inequality \eqref{Jensen} is trivial. Assume now that $X_0=\frac{1}{T_v}\int_0^{T_v}f(z+tv)dt>0$, then for any $X\geq 0$, we have
	$$ X^{1-|\alpha|\sigma}\leq X_0^{1-|\alpha|\sigma}+(1-|\alpha|\sigma) X_0^{-|\alpha|\sigma}(X-X_0).
	$$
	Replacing the inequality above by $X=f(z+tv)$ and averaging over $t\in[0,T_v]$, we obtain \eqref{Jensen}. Since by definition, $|\partial^{\alpha}f|\lesssim_{\alpha,\sigma}|f|^{1-|\alpha|\sigma}$ for all $|\alpha|\leq k$,
	we get
	$$ |\partial^{\alpha}(\mathcal{A}(f)_v)(z)|\lesssim_{\alpha,\sigma}|\mathcal{A}(f)_v(z)|^{1-|\alpha|\sigma}.
	$$
	
	Next we assume that $v=(\xi,\eta)$ is ergodic. In this case, the orbit $z\mapsto z+tv$ is ergodic, then by Weyl's equidistribution theorem, we have
	\begin{align}\label{Weylequi}  \mathcal{A}_v(f)(z)=\widehat{f}(0)=\frac{1}{(2\pi)^2}\int_{\T^2}f(z')dz'.
	\end{align}
To prove this, 
by the assumption of $v$, we have  $k\cdot v\neq 0$ for all $k\in\Z^2\setminus\{0\}$. Moreover, since $f\in C^m$, we can write  
	$$ \frac{1}{T}\int_0^Tf(z+tv)dt=\frac{1}{T}\int_0^T\sum_{k\in\Z^2}\widehat{f}(k)\mathrm{e}^{ik\cdot (z+tv)}dt=\widehat{f}(0)+\sum_{k\neq 0}\widehat{f}(k)\mathrm{e}^{ik\cdot z}\cdot\frac{\mathrm{e}^{iTk\cdot v}-1}{iT(k\cdot v)}.
	$$
	As $|\mathrm{e}^{iTk\cdot v}-1|\leq 2|iTk\cdot v|$ for all $k\in\Z^2$ and 
	$$ \lim_{T\rightarrow\infty}\frac{\mathrm{e}^{iTk\cdot v}-1}{iTk\cdot v}=0,
	$$
	by the dominated convergence theorem, we get \eqref{Weylequi}, which means that $\mathcal{A}_v(f)$ is a constant function.
	Clearly, $\mathcal{A}(f)_v(z)\in \mathcal{D}^{m,k,\sigma}(\T^2)$. The proof of Lemma \ref{Jensen} is now complete.
\end{proof}
By the triangle inequality, the following Lemma is immediate:
\begin{lem}\label{comparison} 
	Let $v\in\mathbb{S}^1$. For any function $f$ on $\T^2$, there holds
	$$|\mathcal{A}(f)_v|\leq \mathcal{A}(|f|)_v.$$ Moreover, if $f_1,f_2$ are two non-negative functions such that $f_1\leq f_2$, we have
	$$ \mathcal{A}(f_1)_v\leq \mathcal{A}(f_2)_v.
	$$
\end{lem}
%\begin{proof}
%	First assume that $v$ is a periodic direction with period $T_v>0$. Then by the triangle inequality, $|\mathcal{A}(f)_v(z)|\leq \mathcal{A}(|f|)_v(z)$ for any $z\in\T^2$. If $0\leq f_1\leq f_2$, we have trivially $\mathcal{A}(f_1)\leq \mathcal{A}(f_2)$. When $v$ is an ergodic direction, $\mathcal{A}(f)_v(z)=\frac{1}{(2\pi)^2}\int_{\T^2}f(z')dz'$, the desired results follow from the triangle inequality. The proof of Lemma \ref{comparison} is complete.
%\end{proof}
\begin{lem}\label{dampingproperty}
	Assume that $f\in \mathcal{D}^{m,k,\sigma}(\T^2)$ and $f\geq 0$. Denote by
	$$ F(x,y):=\int_{-\pi}^y\big(f(x,y')-\mathcal{A}(f)_{\mathrm{e}_2}(x) \big)dy',\quad -\pi<y<\pi,
	$$
	where $\mathrm{e}_2=(0,1)$.
	Then for $0\leq j\leq k$, we have
	$$ |F(x,y)|\leq 4\pi\mathcal{A}(f)_{\mathrm{e}_2}(x),\quad |\partial_x^jF(x,y)|\leq 4\pi\mathcal{A}(f)_{\mathrm{e}_2}^{1-j\sigma}. $$
	Moreover, for all $j_1\geq 0, 1\leq j_2\leq k-j_1$, we have
	$$
	|\partial_x^{j_1}\partial_y^{j_2}F(x,y)|\leq \mathcal{A}(f)_{\mathrm{e}_2}^{1-(j_1+j_2-1)\sigma}(x)+ f(x,y)^{1-(j_1+j_2-1)\sigma}.$$ 
	
\end{lem}
\begin{proof}
	Since $f\geq 0$, the bound for $|F(x,y)|$ is trivial. Taking derivatives, we get
	$$ \partial_x^{j_1}\partial_y^{j_2}F(x,y)= \partial_y^{j_2}\int_{-\pi}^y(\partial_x^{j_1}f(x,y')-\partial_x^{j_1}\mathcal{A} (f)_{\mathrm{e}_2}(x) )dy'.
	$$
	When $j_2=0$, the absolute vaule of the above quantity can be bounded by
	$$ 2\pi\mathcal{A} (f)_{\mathrm{e}_2}^{1-j_1\sigma}(x)+2\pi\partial_x^{j_1}\mathcal{A} (f)_{\mathrm{e}_2}(x)\leq 4\pi(\mathcal{A} (f)_{\mathrm{e}_2})^{1-j_1\sigma}(x),
	$$
	thanks to Lemma \ref{Jensen1} and Lemma \ref{comparison} and Jensen's inequality \eqref{Jensen1}. When $j_2\geq 1$, by definition, we have
	$$ \partial_x^{j_1}\partial_y^{j_2}F(x,y)=\partial_y^{j_2-1}\partial_x^{j_1}f(x,y)-\partial_y^{j_2-1}\partial_x^{j_1}\mathcal{A} (f)_{\mathrm{e}_2}(x). 
	$$
	Taking the absolute value and the proof of Lemma \ref{dampingproperty} is complete.
\end{proof}

Finally, we prove the key geometric proposition, allowing us to improve the local H\"older regularity for averaged damping functions, provided that the original damped region is locally strictly convex with positive curvature:
\begin{prop}\label{convexaveraging} 
	Assume that $a\in\mathcal{D}^{m,k,\sigma}(\T^2)$ such that $a(z)\geq 0$ and the damping boundary $\Sigma_a=\partial\{z:a(z)>0 \}$ is a disjoint union of strictly convex curves with positive curvature. Assume moreover that there exists $R>0$ such that for every $z\in\{a>0\}$ near $\Sigma_a$, 
	\begin{align}\label{vanishingbehavior} 
	 R^{-1}\cdot \mathrm{dist}(z,\Sigma_a)^{\frac{1}{\sigma}}\leq a(z)\leq R\cdot\mathrm{dist}(z,\Sigma_a)^{\frac{1}{\sigma}}.
	\end{align}
	Then for any periodic direction $v\in\mathbb{S}^1$, 
	we have $\mathcal{A}(a)_v\in \mathcal{D}^{m,k,\frac{2\sigma}{\sigma+2}}$, as a one-dimensional periodic function. Furthermore, there exists $R_v>0$, such that for every $z\in\{\mathcal{A}(a)_v(z)>0 \}$near $\Sigma_{\mathcal{A}(a)_v}$, we have
	\begin{align}\label{distance} 
		R_v^{-1} \mathrm{dist}(z,\Sigma_{\mathcal{A}(a)_v})^{\frac{1}{\sigma}+\frac{1}{2}}\leq \mathcal{A}(a)_v(z)\leq R_v
		\mathrm{dist}(z,\Sigma_{\mathcal{A}(a)_v})^{\frac{1}{\sigma}+\frac{1}{2}}.
	\end{align}
\end{prop}

\begin{proof}
Since the vector $v$ is periodic, we can find $p_0,q_0\in\Z$, gcd$(p_0,q_0)=1$, such that $v=\frac{(p_0,q_0)}{\sqrt{p_0^2+q_0^2}}.$
Now we perform a change of coordinate as described in Section \ref{changingcoordinate}. Recall that we have a covering map $\pi_{v}:\mathbb{T}_v^2\mapsto \mathbb{T}^2$ of degree $p_0^2+q_0^2$ that lifts every $2\pi$ periodic function to a $2\pi\sqrt{p_0^2+q_0^2}$-periodic function. Moreover, the change of coordinate system $z\mapsto (X,Y)$ is given by $z=Xv^{\perp}+Yv$ locally.
 As $\pi_v$ is locally isometric, each component of the pre-image of the damped region $\pi_v^{-1}(\omega)$ is still strictly convex with positive curvature, so the inequality \eqref{vanishingbehavior} is preserved, near the boundary of each component. Denote by $\tau=2\pi\sqrt{p_0^2+q_0^2}$. We define the averaging operator on the new torus $\T_v^2$ as
$$ \widetilde{\mathcal{A}}(F)(X,Y):=\frac{1}{\tau}\int_{-\tau/2}^{\tau/2}F(X,Y+t)dt.
$$ 
then by definition 
$$ \pi^*(\mathcal{A}(a)_v)(X,Y)=\mathcal{A}(a)_v(\pi(Xv^{\perp}+Y))=\frac{1}{\tau}\int_{0}^{\tau}a\circ\pi(Xv^{\perp}+(Y+t)v)dt=\frac{1}{\tau}\int_{0}^{\tau}(\pi^*a)(X,Y+t)dt.
$$
Thus
$ 
\pi_v^{*}\big(\mathcal{A}(a)_v\big)=\widetilde{\mathcal{A}}(\pi_v^*a).
$
Therefore, if we are able prove \eqref{distance} for the lifted averaging function $\widetilde{\mathcal{A}}(\pi_v^{*}(a))$, we obtain automatically \eqref{distance} by projection.
	
In summary, from the argument above, without loss of generality, we assume that $v=\mathrm{e}_2$ and assume that $\Omega:=\{a>0\}$ has $l=l(v)$ connected components $\Omega_1,\cdots,\Omega_l$ such that the boundary $\Sigma_{a,j}$ of each $\Omega_j$ has positive curvature. We first consider the situation where $l=1$. By translation invariance, we may assume that $\Omega_1=\{a_1>0\}$ is contained in the fundamental domain $(-K\pi,K\pi)_x\times (-M\pi,M\pi)_y$.
	Then the function $\mathcal{A}(a)_{\mathrm{e}_2}$ can be identified as a function on $\R_x$, 
	
	Since $\Omega_1=\{a_1>0 \}$ is strictly convex, each line $\mathcal{P}_{x}$ of $\R^2$, passing through $(x,0)$ and parallel to $\mathrm{e}_2$ can intersect at most $2$ points of the curve $\Sigma_{a,1}$. Consider the function
	$ x\mapsto P(x):=\mathrm{mes}(\mathcal{P}_x\cap\Omega_1).
	$
	This function is continuous and is supported on a single interval $I=(\alpha,\beta)\subset(-K\pi,K\pi)$. Since $\mathcal{A}(a_1)(x)=0$ if $P(x)=0$, the vanishing behavior of $\mathcal{A}(a_1)$ is determined when $x$ is close to $\alpha$ and $\beta$. Below we only analyze $\mathcal{A}(a_1)(x)$ for $x\in[\alpha,\alpha+\epsilon)$, since the analysis is similar for $x$ near $\beta$. First we observe that $\mathcal{P}_{\alpha}$ must be tangent at a point $z_{0}:=(\alpha,y_{0})$ to the curve $\Sigma_{a,1}$. For sufficiently small $\epsilon>0$, we may parametrize the curve $\Sigma_{a,1}$ near $z_0$ by the function $x=\alpha+g(y)$ with $g(y_0)=g'(y_0)=0$ and $g''(y_0)=c_0>0$, thanks to the fact that the curvature is strictly positive. Therefore, there exists a $C^1$ diffeomorphism $Y=\Phi(y)$ from a neighborhood of $y_0$ to a neighborhood of $Y=0$ such that $\Phi(y_0)=0$ and $g(y)=Y^2$. For each $x\in(\alpha,\alpha+\epsilon)$, $\mathcal{P}_x\cap \Sigma_{a,1}=\{(x,l^-(x)),(x,l^+(x))\}$. We have
	$$ l^+(x)=\Phi^{-1}(\sqrt{x-\alpha}),\quad l^-(x)=\Phi^{-1}(-\sqrt{x-\alpha}).
	$$
	\begin{center} 
		\begin{tikzpicture}[scale=1.4]
			\fill[blue!20] (0,0) circle[radius=1];
			\draw (0,0)
			node[above]{$\Omega_1$};
			\draw[->]
			(-1.2,-1.2)--(1.8,-1.2);
			\draw (1.9,-1.4)
			node[above]{$x$};
			\draw[-, blue]
			(-1,-1.2)--(-1,1.2);
			\draw (-1.2,-0.2)
			node[above]{$z_0$};
			\draw (-1,-1.45)
			node[above]{\small{$\alpha$}};
			\draw (-0.3,-1.45)
			node[above]{\small{$\alpha+\epsilon$}};
			\draw[dashed, red!90!black]
			(-0.7,-1.2)--(-0.7,1.5);
			\fill[blue!90!black] (-1,0) circle (1pt);%%画点
			\fill[blue!90!black] (-1,0) circle (1pt);
			\fill[blue!90!black] (-0.7,-0.73) circle (1pt);
			\draw (-0.3,-0.8)
			node[above]{\tiny$(x,l^-(x))$};
			\fill[blue!90!black] (-0.7,0.72) circle (1pt);
			\draw (-0.3,0.55)
			node[above]{\tiny$(x,l^-(x))$};
			\draw (-0.5, 0.85)
			node[above]{\small {$\mathcal{P}_x$}};
			\draw (0,-1.8)
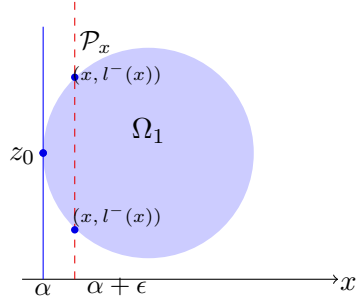
			node[above]{\small{Averaging improves the local H\"older regularity}};
			
			%	\draw (0,0)
			%	node[above]{$\omega_0$};
			%	\draw (2.4,0.8) node[above]{\small{supp$(a)$}};
			%	\draw[->>] (0.5,0)--(0.5,0.5);
			%	\draw[->>] (0.5,0.5)--(0.5,1);
			%	\draw[->>] (0.5,1)--(0.5,1.5);
			%	\draw[->>] (0.5,1.5)--(0.5,2);
			%	\draw (1.2,1)
			%	node[above]{\small{$\zeta_0=(0,1)$}};
		\end{tikzpicture}
	\end{center}
	Since near $z_0$, we have $a(z)\sim \mathrm{dist}(z,\Sigma_a)^{\frac{1}{\sigma}}=\mathrm{dist}(z,\Sigma_{a,1})^{\frac{1}{\sigma}}\sim (x-(\alpha+g(y)))_+^{\frac{1}{\sigma}}$. For $x\in(\alpha,\alpha+\epsilon)$, we have
	\begin{align*}
		\mathcal{A}(a_1)(x)=&\frac{1}{2K\pi}\int_{-K\pi}^{K\pi}a_1(x,y)dy=\frac{1}{2K\pi}\int_{l^-(x)}^{l^+(x)}a_1(x,y)dy\\=&\frac{1}{2K\pi}\int_{-\sqrt{x-\alpha}}^{\sqrt{x-\alpha}}|x-\alpha-Y^2|^{\frac{1}{\sigma}}|(\Phi^{-1})'(Y)|dY\sim_{\Phi,K,\sigma} |x-\alpha|^{\frac{1}{\sigma}+\frac{1}{2}}.
	\end{align*}
Here the implicit constant depends on $\Phi,K,\sigma$, hence it depends on the periodic direction $v$ and the damping $a(z)$.

	Finally, since $a_1\in\mathcal{D}^{m,k,\sigma}$, for $j\leq k$, $x\in(\alpha,\alpha+\epsilon)$,
	\begin{align*}
		|\partial_x^j\mathcal{A}(a_1)(x)|\leq &\frac{1}{2K\pi}\int_{-K\pi}^{K\pi}|\partial_x^{j_1}a_1(x,y)|dy
		\lesssim_K  \int_{l^-(x)}^{l^+(x)}a_1^{1-j\sigma}(x,y)dy\\
		\lesssim &|x-\alpha|^{\frac{1-j\sigma}{\sigma}+\frac{1}{2}}\sim (\mathcal{A}(a_1)(x))^{1-\frac{2\sigma j}{\sigma+2}}. 
	\end{align*}
	This implies that $\mathcal{A}(a_1)\in\mathcal{D}^{m,k,\frac{2\sigma}{\sigma+2}}$.

	To complete the proof of Proposition \ref{convexaveraging}, we need to deal with the situation where $l>1$. In this case, the supports of $\mathcal{A}(a_j)$ may overlap. By linearity and the inequality
	$$ |\partial_x^{j'}\mathcal{A}(a_1+\cdots a_l)|\leq \sum_{j=1}^l|\partial_x^{j'}\mathcal{A}(a_j)|\leq C(a_1,\cdots, a_l)\mathcal{A}(a_1+\cdots +a_l)^{1-\frac{2\sigma j'}{\sigma+2}},
	$$
	we deduce that $\mathcal{A}(a)\in \mathcal{D}^{m,k,\frac{2\sigma}{\sigma+2}}$. It remains to show \eqref{distance}. We define
	\begin{align*}
		& S_+:=\big\{j\in\{1,\cdots,l\}: \int_{x_0}^{x_0+\epsilon}\mathcal{A}(a_j)(x)dx>0,\;\forall \epsilon>0 \big\},\\
		&  S_-:=\big\{j\in\{1,\cdots,l\}: \int_{x_0-\epsilon}^{x_0}\mathcal{A}(a_j)(x)dx>0,\;\forall \epsilon>0 \big\}.
	\end{align*}
	Observe that $x_0\in\Sigma_{\mathcal{A}(a)}$ if and only if $\mathcal{A}(a_j)(x_0)=0$ for all $j\in\{1,\cdots,l \}$ and $S_+\cup S_-\neq \emptyset.$ Note that for $j\in S_{\pm}$, we have
	$$ \mathrm{dist}(x,\Sigma_{\mathcal{A}(a)})^{\frac{1}{2}+\frac{1}{\sigma}}=\mathrm{dist}(x,\Sigma_{\mathcal{A}(a_j)})^{\frac{1}{2}+\frac{1}{\sigma}}\sim \mathcal{A}(a_j)(x),\;\forall x\mp x_0>0 \text{ near }x_0, 
	$$
	and $\mathcal{A}(a_j)=0$, in a neighborhood of $x_0$, for all $j\notin S_+\cup S_-$. 
	Summing over $j\in S_+\cup S_-$, we obtain \eqref{distance}. The proof of Proposition \ref{convexaveraging}
	is now complete.
\end{proof}

\section{Normal form reductions}\label{normalformreduction} 

Now we treat the transversal low frequency part $v_h=\psi(\hbar D_x)u_h^1$, defined in \eqref{vhwh}. We want to average the operator $P_h+iha$ along the direction $\mathrm{e}_2=(0,1)$. Recall that $\mathcal{A}(a)$ is the averaging of $a$ along the vertical direction. 
Recall that $v_h$ satisfies the equation
$$ (P_h+iha)v_h=r_{1,h}=o_{L^2}(h^2\delta_h)=o_{L^2}(h\hbar^2).
$$
We will apply successively two normal form reductions. The first reduction replaces $iha$ by $ih\mathcal{A}(a)$, with a $\mathcal{O}(h^2)$ anti-selfadjoint remainder which cannot be absorbed directly as a remainder of size $\mathcal{O}(h^2\delta_h)$. We need to perform a second normal form reduction to average the anti-selfadjoint remainder.

\subsection{The first averaging}\label{1normalform}
Throughout this section, we denote by
$$ A(x,y):=\int_{-\pi}^{y}(a(x,y')-\mathcal{A}(a)(x))dy'.
$$
We will also fix a cutoff $\psi_1\in C_c^{\infty}(\R)$, supported on $|\eta\pm 1|\leq \frac{1}{2}$ and $\psi_1(\eta)=1$ if $|\eta\pm 1|\leq \frac{1}{4}$.
We need the following basic lemma for exponentials of bounded linear operators:
\begin{lem}\label{exponential}
Let $(G_h)_{0<h<1}$ be a family of $h$-dependent, uniformly bounded operators on a Hilbert space $\mathcal{H}$. Defining the exponential via
$$ e^{sG_h}:=\sum_{k=0}^{\infty}\frac{(sG_h)^k}{k!},\quad s\in\R. 
$$
Then the operator $e^{sG_h}$ is  invertible with inverse $e^{-sG_h}$. Moreover, 
$$ \|e^{sG_h}\|_{\mathcal{L}(\mathcal{H})}\leq e^{|s|\|G_h\|_{\mathcal{L}(\mathcal{H})}}<\infty.
$$
For any linear operator $B$,
\begin{align}\label{conjugateformula} 
 \frac{d}{ds}(e^{sG_h}Be^{-sG_h})=e^{sG_h}\mathrm{ad}_{G_h}(B)e^{-sG_h},
\end{align}
where $\mathrm{ad}_A(B):=[A,B]$. Consequently, for any bounded operator $B$ on $\mathcal{H}$,
\begin{align}\label{ecommutator}
\|[e^{sG_h},B]\|_{\mathcal{L}(\mathcal{H})}\leq |s|e^{3|s|\|G_h\|_{\mathcal{L}(\mathcal{H})}}\|\mathrm{ad}_{G_h}(B)\|_{\mathcal{L}(\mathcal{H})}.
\end{align}
Finally, we have the Taylor expansion
\begin{align}\label{Taylor} 
e^{sG_h}Be^{-sG_h}=\sum_{k=0}^{N-1}\frac{s^k}{k!}\mathrm{ad}_{G_h}^k(B)+\frac{1}{(N-1)!}\int_0^1(1-s)^{N-1}e^{sG_h}\mathrm{ad}_{G_h}(B)e^{-sG_h}ds,\quad \forall N\in\N,
\end{align}
where $\mathrm{ad}_{A}^k(B)=\mathrm{ad}_{A}(\mathrm{ad}_A^{k-1}(B))$, $\mathrm{ad}_A^0(B)=B$.
\end{lem}
\begin{proof}
Assume that $\|G_h\|_{\mathcal{L}(\mathcal{H})}\leq M$ for all $h\in(0,1)$. The series
$$  \sum_{k=0}^{\infty}\frac{\|(sG_h)^k\|_{\mathcal{L}(\mathcal{H})}}{k!}\leq \sum_{k=0}^{\infty}\frac{|s|^kM^k }{k!}=e^{|s|M}
$$
converges absolutely. Therefore, $$\|e^{sG_h}\|_{\mathcal{L}(\mathcal{H})}\leq e^{|s|\|G_h\|_{\mathcal{L}(\mathcal{H})}}.$$
To show that $e^{sG_h}$ is invertible, again, by the absolute convergence of the series $\sum_{k\geq 0}\frac{\|(sG_h)^k\|_{\mathcal{L}(\mathcal{H})}}{k!}$, we deduce that
\begin{align*}
e^{-sG_h}e^{sG_h}=\sum_{k=0}^{\infty}\frac{s^kG_h^k}{k!}\sum_{k_1+k_2=k}(-1)^{k_1}\cdot\frac{k!}{k_1!k_2!}=\mathrm{Id}.
\end{align*}
Therefore, $e^{sG_h}$ is invertible with inverse $e^{-sG_h}$. One easily verifies that $G_h$ commutes with $e^{sG_h}$ and
$\frac{d}{ds}e^{sG_h}=G_he^{sG_h}=e^{sG_h}G_h.
$
This implies \eqref{conjugateformula}. To prove \eqref{ecommutator}, we remark that
\begin{align}\label{conjugationformula2}
[e^{sG_h},B]=\Big(\int_0^se^{s'G_h}\mathrm{ad}_{G_h}(B)e^{-s'G_h}ds'\Big)e^{sG_h}.
\end{align}
Taking the operator norm we obtain \eqref{ecommutator}. The last identity \eqref{Taylor} follows directly from the Taylor expansion with integral remainders. The proof of Lemma \ref{exponential} is complete.
\end{proof}

Our goal is to find an exponential (elliptic) $e^{G_h}$ for some $G_h=\mathrm{Op}_h^w(g)$ with $g\in S^0(\T_y\times\R_\eta)$, depending smoothly in $x$ such that
$$ e^{G_h}(P_h+iha)e^{-G_h}=P_h+ih\mathcal{A}(a)+\text{lower orders}.
$$
To find the operator $G_h$, we consider the conjugate operator
$$ F_h(s):=e^{sG_h}(P_h+iha)e^{-sG_h}
$$
Using the Taylor expansion \eqref{Taylor} up to order $N=2$ in Lemma \ref{exponential} and the symbolic calculus, we get 
$$ F_h(1)=P_h+iha-ih\cdot \frac{i}{h}[G_h,P_h+iha]+\frac{1}{2}[G_h,[G_h,P_h+iha]]+o_{\mathcal{L}(L^2)}(h\hbar^2).
$$
To average the leading order of the anti-selfadjoint part $iha$, we expect the principal symbol of $a-\frac{i}{h}[G_h,h^2D_y^2]$ to be $\mathcal{A}(a)$. To this end, we need to solve the cohomological equation
$$ a+H_{|\eta|^2}(g)=\mathcal{A}(a).
$$
Set
$$ g(x,y,\eta)=-\frac{\psi_1(\eta)}{2\eta}\int_{-\pi}^y(a(x,y')-\mathcal{A}(a)(x))dy'=-\frac{\psi_1(\eta)}{2\eta}\int_{-\pi}^y(a(x,y')-\mathcal{A}(a)(x))dy'A(x,y),
$$
where $\psi_1, A(x,y)$ are defined at the beginning of Subsection \ref{1normalform}. We can indeed define explicitly the quantization of $g$ as $$G_h=\mathrm{Op}_h^w(g)=-\frac{\psi_1(hD_y)}{4hD_y}A(x,y)-A(x,y)\frac{\psi_1(hD_y)}{4hD_y}.
$$ Then $G_h$ is self-adjoint and is $\mathcal{O}(1)$  $h$-semiclassical of order $0$ and $\mathcal{O}(h^{-1})$ classical of order $-1$, smoothly depending on $x\in\T$. In particular, $G_h$ is uniformly bounded on $\mathcal{L}(L^2(\T^2))$. Therefore, by Lemma \ref{exponential}, the operators
$$ e^{sG_h}:=\sum_{n=0}^{\infty}\frac{s^nG_h^n}{n!},\quad e^{-sG_h}=\sum_{n=0}^{\infty}\frac{(-1)^ns^nG_h^n}{n!}.
$$ 
are well-defined and invertible on $L^2(\T^2)$, and $e^{sG_h}e^{-sG_h}=e^{-sG_h}e^{sG_h}=\mathrm{Id},$ for all $s\in\R$.

We now state and prove the main result in this subsection:
\begin{prop}\label{1averaging}
Given $g(x,y,\eta)=-\frac{\psi_1(\eta)}{2\eta}A(x,y)$ and $G_h=\mathrm{Op}_h^w(g)$. Let $v_h^{(1)}:=e^{G_h}v_h$. Then 
$$ P_hv_h^{(1)}+ih\mathcal{A}(a)v_h^{(1)}-[h^2D_x^2,G_h]v_h^{(1)}=\widetilde{r}_h=o_{L^2}(h\hbar^2).
$$  
Moreover, $v_h^{(1)}$ satisfies
\begin{itemize}
	\item[(a)] $\|a^{\frac{1}{2}}v_h^{(1)}\|_{L^2}+\|\mathcal{A}(a)^{\frac{1}{2}}v_h^{(1)}\|_{L^2}=o(\hbar)$.
	\item[(b)] $\mathrm{WF}^{m}_h(v_h^{(1)})\subset\mathrm{WF}_h(v_h)\subset\{(z,\zeta):\zeta=\zeta_0=(0,1)\}$.
	\item[(c)] For any $\widetilde{\psi}\in C_c^{\infty}(\R;[0,1])$ such that $\widetilde{\psi}(\xi)=1$ on the support of $\psi$ which defines $v_h$ in \eqref{vhwh}, we have
	$$ (1-\widetilde{\psi}(\hbar D_x))v_h^{(1)}=O_{L^2}(h^{\frac{m}{2} }).
	$$
\end{itemize}
\end{prop}
Note that the definition of semiclassical wavefront sets $\mathrm{WF}^m(\cdot)$ and $\mathrm{WF}_h(\cdot)$ are recalled at the end of Appendix B. 

From Proposition \ref{1averaging}, one may deduce from Lemma \ref{dampingproperty} that the anti-selfadjoint remainder satisfies $$[G_h,h^2D_x^2]v_h^{(1)}=o_{L^2}(h^2).$$ Though we can not absorb it directly as an error of order $o(h\hbar^2)$, its principal part is non self-adjoint and can be viewed as an lower order perturbation of the averaged damping  $ih\mathcal{A}(a)$. In the next subsection, we will perform a second normal form to average the operator $[G_h,h^2D_x^2]$ so that it becomes \emph{independent} of the variable $y$.

\begin{proof} 
To simplify the notation, we will use $R_h$ to denote operators of size at most $o_{\mathcal{L}(L^2)}(h\hbar^2)$ and $r_h$ to denote errors of size $o_{L^2}(h\hbar^2)$. Both of them may change from line to line.

For $F_h(s)=e^{sG_h}(P_h+iha)e^{-sG_h}$, using the equation for $v_h$, we have
$$ F_h(s)v_h^{(1)}=e^{sG_h}(P_h+iha)v_h=r_h.
$$
As motivated before, we write down the full conjugate operator:
\begin{align}\label{normalform}F_h(1)=e^{G_h}(P_h+iha)e^{-G_h}=&P_h+ih\mathcal{A}(a)+ih(a-\mathcal{A}(a))-[h^2D_y^2,G_h]\notag\\-&[h^2D_x^2,G_h]+ih[G_h,a]+\frac{1}{2}[G_h,[G_h,P_h]]+R_{h}.
\end{align}
Here the lower order operators $$ih(a-\mathcal{A}(a))-[h^2D_y^2,G_h],\; ih[G_h,a],\; [G_h,[G_h,P_h]]$$ are not in priorly negligible, since they are merely of order $\mathcal{O}_{\mathcal{L}(L^2)}(h^2)$. Nevertheless, it turns out that they are negligible when acting on the function $v_h^{(1)}:=e^{G_h}v_h$. We will prove this fact through several lemmas. First we show that the normal form transform $e^{G_h}v_h$ does not alter $\mathrm{WF}_h(v_h)$. In particular, we prove (b), (c) of Proposition \ref{normalform}:
\begin{lem}\label{wavefront} 
We have 
$$ \mathrm{WF}^{m}_h(v_h^{(1)})\subset \mathrm{WF}_h(v_h).
$$
Moreover, for any $\widetilde{\psi}\in C_c^{\infty}(\R;[0,1])$ such that $\widetilde{\psi}(\xi)=1$ on the support of $\psi$ which defined $v_h$ in \eqref{vhwh}, we have
$$ \|(1-\widetilde{\psi}(\hbar D_x))v_h^{(1)}\|_{L^2}=O(\hbar^m). 
$$
\end{lem}
\begin{proof}
Let $l(z,\zeta)$ be a symbol supported on a compact set of $T^*\T^2\setminus \mathrm{WF}_h(v_h)$ and $L_h=\mathrm{Op}_h(l)$, it suffices to show that
$$ L_he^{G_h}v_h=O_{L^2}(h^{N}),
$$
for any $N\leq m$. Recall that by definition of $\mathrm{WF}^m_h(v_h)$ (see the end of Appendix B), for any $h$-pseudodifferential operator $Q_h$ with the principal symbol supported away from $\mathrm{WF}_h(v_h)$,
$$ Q_hv_h=O_{L^2}(h^m).
$$
Using again \eqref{Taylor}, we write
\begin{align*}
 L_he^{G_h}v_h=&e^{G_h}(e^{-G_h}L_he^{G_h})v_h\\=&e^{G_h}\Big(\sum_{n=0}^{N-1}\frac{(-s)^{n}}{n!}\mathrm{ad}_{G_h}^n(L_h)v_h\Big)-\frac{1}{(N-1)!}e^{G_h}\int_0^1(1-s)^{N-1}e^{-sG_h}\mathrm{ad}_{G_h}^{N}(L_h)e^{sG_h}v_hds.
\end{align*}
By the symbolic calculus, the last term is $O_{L^2}(h^N)$, and for each $n\leq N-1$, $\mathrm{ad}_{G_h}^n(L_h)$ is a $h$-pseudodifferential operator with symbol supported away from $\mathrm{WF}_h(v_h)$, thus $\mathrm{ad}_{G_h}(L_h)v_h=O_{L^2}(h^N)$. This shows that $\mathrm{WF}_h^m(v_h^{(1)})\subset\mathrm{WF}_h(v_h)$. 

For the second assertion, we first note that, since $v_h^{(1)}=\psi(\hbar D_x)u_h^1$,  $(1-\widetilde{\psi}(\hbar D_x))v_h=O_{L^2}(h^m)$. We observe also that since $G_h$ is of the form
$$ G_h=b(hD_y)A(x,y)+A(x,y)b(hD_y)
$$
and the Fourier multiplier $b(hD_y)$ commutes with $\widetilde{\psi}(\hbar D_x)$, we have
$$ [G_h,\widetilde{\psi}(\hbar D_x)]v_h^{(1)}=b(hD_y)[A,\widetilde{\psi}(\hbar D_x)]v_h^{(1)}+[A,\widetilde{\psi}(\hbar D_x)]b(hD_y)v_h^{(1)}=O_{L^2}(\hbar^m)=o_{L^2}(h^{\frac{m}{2}}),
$$ 
since $[A,\widetilde{\psi}(\hbar D_x)]$ is a $\hbar$-pseudodifferential operator with symbol supported away from the support of $\psi(\xi)$, thanks to the fact that $\widetilde{\psi}\equiv 1$ on supp$(\psi)$. More generally, for any other cutoff $\chi_1$ such that $\chi_1=1$ on supp$(\psi)$, we always have
$$ (1-\chi_1(\hbar D_x))G_hv_h^{(1)}=O_{L^2}(\hbar^m).
$$
Since
$$ \mathrm{ad}_{G_h}^{n+1}(\widetilde{\psi}(\hbar D_x))v_h^{(1)}=G_h\mathrm{ad}_{G_h}^n(\widetilde{\psi}(\hbar D_x) )v_h^{(1)}-\mathrm{ad}_{G_h}^n(\widetilde{\psi}(\hbar D_x) )G_hv_h^{(1)}.
$$
By writing $G_hv_h^{(1)}=(1-\chi_n(\hbar D_x))G_hv_h^{(1)}+\chi_n(\hbar D_x)G_hv_h^{(1)}$ for some cutoffs $\chi_n$, $1\leq n\leq m$, such that $\widetilde{\psi}=1$ on supp$(\chi_n)$ and $\chi_{n+1}=1$ on supp$(\chi_{n})$. Therefore, by induction, we deduce that for every $1\leq n\leq m$,
$$ \mathrm{ad}_{G_h}^n(\widetilde{\psi}(\hbar D_x))v_h^{(1)}=O_{L^2}(\hbar^m).
$$
By Taylor expansion, this shows that $e^{-G_h}(1-\widetilde{\psi}(\hbar D_x))e^{G_h}v_h=O_{L^2}(\hbar^m)$.
 The proof of Lemma \ref{wavefront} is now complete.
\end{proof}
Next we show that $ih(a-\mathcal{A}(a))v_h^{(1)}-[h^2D_y^2,G_h]v_h^{(1)}$ and $ih[G_h,a]v_h^{(1)}$ are indeed remainders.
\begin{lem}\label{erro1} 
We have
$$ \|a^{\frac{1}{2}}v_h^{(1)}\|_{L^2}=o(\hbar),\quad ih\|[G_h,a]v_h^{(1)}\|_{L^2}=O(h^2\hbar)
$$
and
$$ \|ih(a-\mathcal{A}(a))v_h^{(1)}-[h^2D_y^2,G_h]v_h^{(1)}\|_{L^2}=o(h^2\hbar).
$$
\end{lem} 
\begin{proof}
Since $a^{\frac{1}{2}}\in W^{1,\infty}$, by Lemma \ref{commutator}, $e^{-G_h}a^{\frac{1}{2}}e^{G_h}=a^{\frac{1}{2}}+O_{\mathcal{L}(L^2)}(h)$ and $a^{\frac{1}{2}}v_h=o_{L^2}(\hbar)$, we have
$$a^{\frac{1}{2}}v_h^{(1)}=e^{G_h}(a^{\frac{1}{2}}v_h+O_{L^2}(h))=o_{L^2}(\hbar).$$ 
Note that $$[G_h,a]=a^{\frac{1}{2}}[G_h,a^{\frac{1}{2}}]+[G_h,a^{\frac{1}{2}}]a^{\frac{1}{2}}=2[G_h,a^{\frac{1}{2}}]a^{\frac{1}{2}}+[a^{\frac{1}{2}},[G_h,a^{\frac{1}{2}}]],$$
Since $a^{\frac{1}{2}}\in W^{1,\infty}$, from Corollary \ref{commutator}, we have
$$ ih\|[G_h,a]v_h^{(1)}\|_{L^2}\leq Ch^2\|a^{\frac{1}{2}}v_h^{(1)}\|_{L^2}+Ch^3\|v_h^{(1)}\|_{L^2}=o(h^{2}\hbar).
$$
For the last assertion, denote by $b(\eta)=-\frac{\psi_1(\eta)}{4\eta}$, then $G_h=b(hD_y)A+Ab(hD_y)$, where $A=A(x,y)$. Direct computation yields
\begin{align*}
[h^2D_y^2,G_h]=&i\frac{h}{2} \big(\psi_1(hD_y)\partial_yA+\partial_yA \psi_1(hD_y)\big)-h^2\big(\partial_y^2A b(hD_y)-b(hD_y)\partial_y^2A \big).
\end{align*}
Since 
$$ \partial_yA=(a-\mathcal{A}(a)),\quad \partial_y^2A=\partial_ya.
$$
Using the symbolic calculus, we are able to write
\begin{align*}
 ih(a-\mathcal{A}(a))v_h^{(1)}-[h^2D_y^2,G_h]v_h^{(1)}=&ih(1-\psi_1(hD_y))(a-\mathcal{A}(a))v_h^{(1)}+\frac{ih}{2}[\psi_1(hD_y),a]v_h^{(1)}\\+&h^2B_h\partial_yav_h^{(1)}+O_{L^2}(h^3),
\end{align*}
where $B_h=\mathcal{O}_{L^2}(1)$, uniformly in $0<h<1$. Note that supp$(1-\psi_1(\eta))\cap \mathrm{WF}_h^m(v_h)=\emptyset$, hence by the first assertion of Lemma \ref{wavefront}, supp$(1-\psi_1(\eta))\cap \mathrm{WF}_h^m(v_h^{(1)})=\emptyset$, the first term on the right side is $O_{L^2}(h^3)$, say. Next, writing $[\psi_1(hD_y),a]$ as $2[\psi_1(hD_y),a^{\frac{1}{2}}]a^{\frac{1}{2}}+[a^{\frac{1}{2}},[\psi_1(hD_y),a^{\frac{1}{2}}]]$, using the first inequality of Lemma \ref{erro1} and Corollary \ref{commutator}, we get
$$h[\psi_1(hD_y),a]v_h^{(1)}=o_{L^2}(h^2\hbar).$$ Finally, since $|\partial_y a|\lesssim a^{\frac{1}{2}}$, we deduce that $h^2B_h\partial_yav_h^{(1)}=o_{L^2}(h^2\hbar)$.
The proof of Lemma \ref{erro1} is now complete.
\end{proof}
Our next goal is to show that $[G_h,[G_h,P_h]]v_h^{(1)}$ is a remainder. The argument is slightly more tricky. To this end, we need to exploit an extra smallness from the operator $G_h$. This in turns requires to show that $\|\mathcal{A}(a)^{\frac{1}{2}}v_h^{(1)}\|_{L^2}$ has the same order of $\|a^{\frac{1}{2}}v_h^{(1)}\|_{L^2}$. The key observation is that, modulo $O_{\mathcal{L}(L^2)}(h^3)$, the operator $[G_h,[G_h,P_h]]$ is self-adjoint, thus we can perform the same energy estimate for the anti-selfadjoint part only, as in the proof of (a) in Lemma \ref{apriori}.
\begin{lem}\label{averagingdamping} 
If $h\delta_h^{-\frac{1}{3}}=o(1)$, we have
$$ \|\mathcal{A}(a)^{\frac{1}{2}}v_h^{(1)}\|_{L^2}=o(\hbar).
$$
\end{lem}
\begin{proof}
Recall the notation that $r_h$ represents the error terms of size $o_{L^2}(h\hbar^2)$, from Lemma \ref{erro1} and \eqref{normalform}, we have the equation
\begin{align}\label{normalform1} 
\big(P_h+ihQ_h+\frac{1}{2}[G_h,[G_h,P_h]]\big)v_h^{(1)}=r_h=o_{L^2}(h\hbar^2), 
\end{align}
where
$$ Q_h=\mathcal{A}(a)+\frac{i}{h}[h^2D_x^2,G_h].
$$
Multiplying by $\ov{v}_h^1$, integrating and taking the imaginary part, we have
$$ h(Q_hv_h^{(1)},v_h^{(1)})_{L^2}=\Im(v_h^{(1)},r_h)_{L^2}-\frac{1}{2}\Im([G_h,[G_h,P_h]]v_h^{(1)},v_h^{(1)})_{L^2}.
$$
Since $([G_h,[G_h,P_h]])^*-[G_h,[G_h,P_h]]=O_{L^2}(h^3)$, we have
\begin{align}\label{interpolation0} 
 |(Q_hv_h^{(1)},v_h^{(1)})_{L^2}|\leq o(\hbar^2).
\end{align}
To conclude, we need to estimate
$|(\frac{i}{h}[h^2D_x^2,G_h]v_h^{(1)},v_h^{(1)})_{L^2}|$.
Note that,
\begin{align}\label{expansion} \frac{i}{h}[h^2D_x^2,G_h]=2(\partial_xG_h)hD_x-ih(\partial_x^2G_h).
\end{align}
Recall that with $b(\eta)=-\frac{ \chi(\eta)}{4\eta}$, we have
$ \partial_x^jG_h=b(hD_y)(\partial_x^jA)+(\partial_x^jA)b(hD_y)
$ for $j=1,2$.
By Lemma \ref{dampingproperty}, 
$ |\partial_x^jA|\leq 4\pi\mathcal{A}(a)^{1-j\sigma}.
$
Thus for $j=1,2$, we can write\footnote{Note that $\mathcal{A}(a)$ commutes with $\partial_x^jG_h$.}
$$ \partial_x^jG_h=\frac{\partial_x^jG_h}{\mathcal{A}(a)^{1-j\sigma}}\mathcal{A}(a)^{1-j\sigma}.
$$
For $j=1,2$, the operator $\frac{\partial_x^jG_h}{\mathcal{A}(a)^{1-j\sigma}}$ is bounded on $L^2(\T^2)$, since $\sigma<\frac{1}{4}$, we have
\begin{align}\label{interpolation1}  |h(\partial_x^2G_hv_h^{(1)},v_h^{(1)})_{L^2}|\lesssim h\|\mathcal{A}(a)^{1-2\sigma}v_h^{(1)}\|_{L^2}\|v_h^{(1)}\|_{L^2}\lesssim h\|\mathcal{A}(a)^{\frac{1}{2}}v_h^{(1)}\|_{L^2}\|v_h^{(1)}\|_{L^2}.
\end{align}
Denote by $\widetilde{G}_h=\frac{\partial_xG_h}{\mathcal{A}(a)^{1-\sigma}}$, which is self-adjoint, uniformly bounded on $L^2$ and commutes with $\mathcal{A}(a)$, we have
\begin{align}\label{interpolation2} 
|(\partial_xG_hhD_xv_h^{(1)},v_h^{(1)})_{L^2}|=|(\widetilde{G}_h\mathcal{A}(a)^{\frac{1}{2}}\mathcal{A}(a)^{\frac{1}{2}-\sigma}hD_xv_h^{(1)},v_h^{(1)})_{L^2}|\lesssim &\|\mathcal{A}(a)^{\frac{1}{2}}hD_xv_h^{(1)}\|_{L^2}\|\mathcal{A}(a)^{\frac{1}{2}-\sigma}v_h^{(1)}\|_{L^2}.
\end{align}
From Lemma \ref{wavefront}, modulo an acceptable error $O_{L^2}(h^{3}),$ say, we may replace $hD_xv_h^{(1)}$ by $h\hbar^{-1}b_1(\hbar D_x)v_h^{(1)}$, with $b_1(\xi)=\xi \widetilde{\psi}(\xi)$. By the Corollary \ref{commutator} and the fact that $\mathcal{A}(a)^{\frac{1}{2}}\in W^{1,\infty}$ (since $\mathcal{A}(a)\in\mathcal{D}^{m,2,\sigma}$), we have
\begin{align*}
 \|\mathcal{A}(a)^{\frac{1}{2}}hD_xv_h^{(1)}\|_{L^2}\leq &h\hbar^{-1}\|[\mathcal{A}(a)^{\frac{1}{2}},b_1(\hbar D_x)]v_h^{(1)}\|_{L^2}+h\hbar^{-1}\|b_1(\hbar D_x)\mathcal{A}(a)^{\frac{1}{2}}v_h^{(1)}\|_{L^2}+O(h^3)\\
 \lesssim &h\|v_h^{(1)}\|_{L^2}+h\hbar^{-1}\|\mathcal{A}(a)^{\frac{1}{2}}v_h^{(1)}\|_{L^2}+O(h^3).
\end{align*}
Combining this with the interpolation:
$$ \|\mathcal{A}(a)^{\frac{1}{2}-\sigma}v_h^{(1)}\|_{L^2}\leq \|v_h^{(1)}\|_{L^2}^{2\sigma}\|\mathcal{A}(a)^{\frac{1}{2}}v_h^{(1)}\|_{L^2}^{1-2\sigma},
$$
we get from \eqref{interpolation2} that
$$ |(\partial_xG_hhD_xv_h^{(1)},v_h^{(1)})_{L^2}|\lesssim h\|v_h^{(1)}\|_{L^2}^{1+2\sigma}\|\mathcal{A}(a)^{\frac{1}{2}}v_h^{(1)}\|_{L^2}^{1-2\sigma}+h\hbar^{-1}\|v_h^{(1)}\|_{L^2}^{2\sigma}\|\mathcal{A}(a)^{\frac{1}{2}}v_h^{(1)}\|_{L^2}^{2-2\sigma}+O(h^3).
$$
Together with \eqref{interpolation0} and \eqref{interpolation1}, we finally get
\begin{align*}
\|\mathcal{A}(a)^{\frac{1}{2}}v_h^{(1)}\|_{L^2}^2\lesssim &h\|\mathcal{A}(a)^{\frac{1}{2}}v_h^{(1)}\|_{L^2}\|v_h^{(1)}\|_{L^2}+ h\|v_h^{(1)}\|_{L^2}^{1+2\sigma}\|\mathcal{A}(a)^{\frac{1}{2}}v_h^{(1)}\|_{L^2}^{1-2\sigma}\\+&h\hbar^{-1}\|v_h^{(1)}\|_{L^2}^{2\sigma}\|\mathcal{A}(a)^{\frac{1}{2}}v_h^{(1)}\|_{L^2}^{2-2\sigma}+o(\hbar^2).
\end{align*}
By Young's inequality $$K_1K_2\leq \epsilon K_1^{p}+C_{\epsilon}K_2^{p'},\quad \frac{1}{p}+\frac{1}{p'}=1, \;\epsilon>0,$$
we deduce that
$$ \|\mathcal{A}(a)^{\frac{1}{2}}v_h^{(1)}\|_{L^2}^2\leq \frac{1}{2}\|\mathcal{A}(a)^{\frac{1}{2}}v_h^{(1)}\|_{L^2}^2+C\max\{h^{2},(h\hbar^{-1})^{\frac{1}{\sigma}},h^{\frac{2}{1+2\sigma}}\}\|v_h^{(1)}\|_{L^2}^2+o(\hbar^2)
$$
Note that $\sigma< \frac{1}{4}$ since $\beta\geq 4$, and $h^3\delta_h^{-1}=o(1)$, the second term on the right hand side is $o(\hbar^2)$. This completes the proof of Lemma \ref{averagingdamping}.
\end{proof}

Finally, we note that the principal symbol of $-\frac{1}{h^2}[G_h,[G_h,P_h]]$ is
$$ q_0(x,y,\xi,\eta)=-\partial_{\eta}g\partial_y(2\xi\partial_xg+2\eta\partial_yg)+\partial_xg\partial_{\xi}(2\xi\partial_xg+2\eta\partial_yg)+\partial_yg\partial_{\eta}(2\xi\partial_xg+2\eta\partial_yg).
$$ 
Since $g(x,y,\eta)=2b(\eta)A(x,y)=2b(\eta)\int_{-\pi}^y(a(x,y')-\mathcal{A}(a)(x))dy'$, we deduce from Lemma \ref{dampingproperty}, the fact that $a,\mathcal{A}(a)\in \mathcal{D}^{m,2,\sigma},\sigma<\frac{1}{4}$, that
$$ |q_0(x,y,\xi,\eta)|^2\leq Ca(x,y)+C\mathcal{A}(a)(x)+q_1(x,y,\xi,\eta), 
$$
where supp$(q_1)\cap\mathrm{WF}_h^m(v_h^{(1)})=\emptyset$. Therefore, by the sharp G$\mathring{\mathrm{a}}$rding inequality,
$$ \|\frac{1}{h^2}[G_h,[G_h,P_h]]v_h^{(1)}\|_{L^2}\leq C\|a^{\frac{1}{2}}v_h^{(1)}\|_{L^2}+C\|\mathcal{A}(a)^{\frac{1}{2}}v_h^{(1)}\|_{L^2}+O(h^{\frac{1}{2}})=o(\delta_h).
$$
Hence
$[G_h,[G_h,P_h]]v_h^{(1)}=o_{L^2}(h^{2}\delta_h)=o_{L^2}(h\hbar^2).
$
The proof of Proposition \ref{1averaging} is now complete.
\end{proof}
\subsection{The second averaging}
In this subsection, we prove the following proposition of the second normal form reduction. Unlike the first normal form which is less perturbative (the operator $e^{G_h}$ is not close to the identity), we are able to make the normal form transform close to the identity, in the spirit of \cite{BZ4} (see also \cite{BuSun},\cite{LeS} for related applications to the Bouendi-Grushin operators):
\begin{prop}\label{2averaging} 
There exist a real-valued symbol $g_1(x,y,\eta)$ in $S^0$ and the associated pseudo-differential operator $G_{1,h}=\mathrm{Op}_h(g_1)$, a function $\kappa(x)\in W^{1,\infty}(\T_x)$, the Fourier multiplier $b(hD_y)$, such that $v_h^{(2)}:=(\mathrm{Id}-G_{1,h}hD_x)^{-1}v_h^{(1)}$ satisfies the equation
$$ (P_h+ih\mathcal{A}(a))v_h^{(2)}+ih\kappa(x)\mathcal{A}(a)^{\frac{1}{2}}b(hD_y)hD_xv_h^{(2)}=r_{4,h}=o_{L^2}(h\hbar^2).
$$
Moreover,
$$ \|v_h^{(2)}-v_h^{(1)}\|_{L^2}=O(h\hbar^{-1}),\quad \|\mathcal{A}(a)^{\frac{1}{2}}v_h^{(2)}\|_{L^2}=o(\hbar).
$$
\end{prop}
The importance of the above proposition is that it makes possible to take the Fourier transform in $y$ variable and reduce the equation of $v_h^{(2)}$ mode-by-mode to one-dimensional ordinary differential equations.

To prove Proposition \ref{2averaging}, we want to average the non self-adjoint part $[h^2D_x^2,G_h]$ in the equation
$$ (P_h+i\mathcal{A}(a))v_h^{(1)}-[h^2D_x^2,G_h]v_h^{(1)}=\widetilde{r}_h=o_{L^2}(h\hbar^2).
$$
We first identify the principal part of this lower order non-selfadjoint part:
\begin{lem}\label{secondaveraginglemma1} 
We have
$$ [h^2D_x^2,G_h]v_h^{(1)}=-4ih^2\hbar^{-1}(\partial_xA)b(hD_y)\hbar D_xv_h^{(1)}+o_{L^2}(h\hbar^2).
$$
Moreover,
$$ [h^2D_x^2,G_h]v_h^{(1)}=o_{L^2}(h^2).
$$
\end{lem}
\begin{proof}
Recall that $G_h=b(hD_y)A+Ab(hD_y)$, hence $$[h^2D_x^2,G_h]=[[h^2D_x^2,A],b(hD_y)]+2b(hD_y)[h^2D_x^2,A].$$
The principal symbol of $-\frac{1}{h^2}[[h^2D_x^2,A],b(hD_y)]$ is $$q_2(x,y,\xi,\eta)=-2\xi b'(\eta)(\partial_xa-\mathcal{A}'(a)(x)).$$
Thus from Lemma \ref{dampingproperty}, $|\partial_x a|+|\mathcal{A}'(a)|\leq Ca^{\frac{1}{2}}+C\mathcal{A}(a)^{\frac{1}{2}}$, thus
$$ |q_2(x,y,\xi,\eta)|^2\leq Ca(x,y)+C\mathcal{A}(a)(x)+q_3(x,y,\xi,\eta),
$$
where supp$(q_3)\cap\mathrm{WF}_h^{m}(v_h^1)=\emptyset$. By the sharp G$\mathring{\mathrm{a}}$rding inequality (Theorem \ref{Garding}),
$$ \|[[h^2D_x^2,A],b(hD_y)]v_h^{(1)}\|_{L^2}\leq O(h^{\frac{5}{2}})=o(h\hbar^2).
$$

It remains to treat $b(hD_y)[h^2D_x^2,A]$.
Note that
$ [h^2D_x^2,A]=h^2(D_x^2A)+2h(D_xA)hD_x.
$ From Lemma \ref{wavefront} we may replace $v_h^{(1)}$ by $\widetilde{\psi}(\hbar D_x)v_h^{(1)}$. 
Therefore,
$$ [h^2D_x^2,A]v_h^{(1)}=h^2(D_x^2A)v_h^{(1)}+2h^2\hbar^{-1}b(hD_y)(D_xA)\hbar D_x\widetilde{\psi}(\hbar D_x)v_h^{(1)}+O_{L^2}(h^3).
$$
From Lemma \ref{dampingproperty}, $|D_x^jA|\lesssim \mathcal{A}(a)^{1-j\sigma}\leq \mathcal{A}(a)^{\frac{1}{2}}$ for $j=1,2$, we have $$h^2b(hD_y)(D_x^2A)v_h^{(1)}=o_{L^2}(h^2\hbar).$$ Next, we write
$$ hb(hD_y)(D_xA)hD_x=h^2\hbar^{-1}[b(hD_y),D_xA]\hbar D_x+h^2\hbar^{-1}(D_xA)b(hD_y)\hbar D_x.
$$
By the symbolic calculus, $h^2\hbar[b(hD_y),D_xA]\hbar D_x\widetilde{\psi}(\hbar D_x)v_h^{(1)}=O_{L^2}(h^{3}\hbar^{-1}).$ This completes the proof of Lemma \ref{secondaveraginglemma1}. 
\end{proof}
\begin{proof}[Proof of Proposition \ref{secondaveraginglemma1}]
Thanks to Lemma \ref{secondaveraginglemma1}, we can write
$$ (P_h+ih\mathcal{A}(a))v_h^{(1)}+ihQ_{1,h}hD_xv_h^{(1)}=r_{2,h}=o_{L^2}(h\hbar^2),
$$
where
$$ Q_{1,h}=-4(\partial_x A)b(hD_y),\quad b(\eta)=-\frac{\chi(\eta)}{4\eta}.
$$
Note that the principal symbol of $Q_{1,h}$ is independent of $\xi$ variable. Now we perform a second normal form transform. Recall that $$A(x,y)=\int_{-\pi}^y(a(x,y')-\mathcal{A}(a)(x))dy'.$$
Consider the ansatz $v_h^{(1)}=(1-G_{1,h}hD_x)v_h^{(2)}$, where $G_{1,h}$ will be chosen such that  $G_{1,h}h D_x=\mathcal{O}_{\mathcal{L}(L^2)}(h\hbar^{-1})$. The new quasi-modes $v_h^{(2)}$ satisfy the equation
\begin{align*}
(P_h+ih\mathcal{A}(a)+ihQ_{1,h}hD_x)v_h^{(2)}-[h^2D_y^2,G_{1,h}hD_x]v_h^{(2)}=r_{3,h}
\end{align*}
where
\begin{align}\label{rest} 
r_{3,h}=&(1-G_{1,h}hD_x)^{-1}r_{2,h}+(1-G_{1,h}hD_x)^{-1}[h^2D_x^2+ih\mathcal{A}(a)+ihQ_{1,h}hD_x, G_{1,h}hD_x]v_h^{(2)}\notag \\
+&\big((1-G_{1,h}hD_x)^{-1}-\mathrm{Id}\big)[h^2D_y^2,G_{1,h}hD_x]v_h^{(2)}.
\end{align}
Note that if  $G_{1,h}hD_x=\mathcal{O}_{\mathcal{L}(L^2)}(h\hbar^{-1})$, the operator $(1-G_{1,h}hD_x)$ is invertible for sufficiently small $h$ (thus $\hbar$). In particular,
$$ v_h^{(2)}=(1-G_{1,h}hD_x)^{-1}v_h^{(1)}=\sum_{n=0}^{\infty}(G_{1,h}hD_x)^nv_h^{(1)}=\sum_{n=0}^{9}(G_{1,h}hD_x)^nv_h^{(1)}+O_{L^2}(h^{10}\hbar^{-10}),
$$
where the last error term is $o_{L^2}(h\hbar^2)$. 
 Since from Lemma \ref{wavefront}, we may replace $v_h^{(1)}$ by $\widetilde{\psi}(\hbar D_x)v_h^{(1)}$, modulo an error of $\mathcal{O}_{L^2}(h^\frac{N}{2})$ for any $N\leq m$, we may also replace $v_h^{(2)}$ by $\widetilde{\psi}(\hbar D_x)v_h^{(2)}$ implicitly in the argument below.
Therefore, with $G_{1,h}=\mathrm{Op}_h(g_1)\widetilde{\psi}(\hbar D_x)$, we have
$$ (P_h+ih\mathcal{A}(a))v_h^{(2)}+ihQ_{1,h}hD_xv_h^{(2)}+ih\mathrm{Op}_h(\{\eta^2,g_{1}\})hD_xv_h^{(2)}
+h^2C_hhD_xv_h^{(2)}=r_{3,h},
$$
where $C_h$ is uniformly bounded on $L^2(\T^2)$. Note that the last term on the right hand size is of size $O_{L^2}(h^3\hbar^{-1})$. Now we set\footnote{Recall that the support of $b(\eta)$ is away from $\eta=0$.}
$$ g_1(x,y,\eta)=\frac{2b(\eta)}{\eta}\int_{-\pi}^y\big(\partial_xA(x,y')-\mathcal{A}(\partial_xA)(x)\big)dy'.
$$
Then we have $2\eta\partial_yg_1-4(\partial_xA)b(\eta)=-4\mathcal{A}(\partial_xA)(x)b(\eta)$. By the symbolic calculus, modulo an error of size $O_{L^2}(h^3\hbar^{-1})$, we can replace $ih(Q_{1,h}+\mathrm{Op}_h(\{\eta^2,g_1\}))hD_xv_h^{(2)}$ by $-4ih\mathcal{A}(\partial_x A)b(hD_y)\widetilde{\psi}(\hbar D_x)hD_xv_h^{(2)}.$ 
Therefore, the equation of  $v_h^{(2)}$ becomes
$$ (P_h+ih\mathcal{A}(a)-4ih\mathcal{A}(\partial_xA)b(hD_y)hD_x)v_h^{(2)}=r_{4,h},
$$
where
$$ r_{4,h}=r_{3,h}+O_{L^2}(h^3\hbar^{-1})=r_{3,h}+o_{L^2}(h\hbar^2).
$$
It is clear that
$ \|v_h^{(2)}-v_h^{(1)}\|_{L^2}=O(h\hbar^{-1})=o(1),$ and from the relation $$v_h^{(2)}=v_h^{(1)}+G_{1,h}hD_xv_h^{(1)}+O_{L^2}(h^2\hbar^{-2}),$$
we deduce that $\mathcal{A}(a)^{\frac{1}{2}}v_h^{(2)}=o_{L^2}(\hbar)$.

Set 
$$ \kappa:=-\frac{4\mathcal{A}(\partial_x A)}{\mathcal{A}(a)^{\frac{1}{2}}},
$$
to complete the proof, we need to verify that
\begin{itemize}
	\item[(i)] $\kappa\in W^{1,\infty}(\T_x)$;
	\item[(ii)] $r_{3,h}=o_{L^2}(h\hbar^2)$.
\end{itemize}
To verify (i), observe that $\mathcal{A}'(f)(x)=\mathcal{A}(\partial_xf)(x)$. By Lemma \ref{comparison}, Lemma \ref{dampingproperty} and the fact that $\sigma<\frac{1}{4}$, we have
$$ |\mathcal{A}(\partial_x^jA)|\leq \mathcal{A}(|\partial_x^j A|)\leq C\mathcal{A}(a)^{1-j\sigma}\leq C\mathcal{A}(a)^{\frac{1}{2}},\;\forall j=1,2.
$$
This shows that $\kappa,\kappa'$ are bounded. 

It remains to prove (ii). Recall \eqref{rest} and the fact that $(1-G_{1,h}hD_x)^{-1}-\mathrm{Id}=\mathcal{O}_{\mathcal{L}(L^2)}(h\hbar^{-1})$, 
it suffices to show that 
\begin{align}\label{(ii)term1}
 [h^2D_x^2+ih\mathcal{A}(a)+ihQ_{1,h}hD_x,G_{1,h}hD_x]v_h^{(2)}=o_{L^2}(h\hbar^2)
\end{align}
and
\begin{align}\label{(ii)term2}
 [h^2D_y^2,G_{1,h}hD_x]v_h^{(2)}=o_{L^2}(\hbar^3).
\end{align}
Denote by
$$ A_2(x,y)=\int_{-\pi}^y\big(\partial_xA(x,y')-\mathcal{A}(\partial_xA)(x)\big)dy'.
$$
Pointwise, we have
$$ |A_2|\lesssim \mathcal{A}(|\partial_x A|)\lesssim \mathcal{A}(a)^{1-\sigma},\quad |\partial_x A_2|\lesssim \mathcal{A}(|\partial_x^2A|)\lesssim \mathcal{A}(a)^{1-2\sigma} 
$$
and
$$ |\partial_yA_2|\lesssim |\partial_xA|+\mathcal{A}(|\partial_xA|)\lesssim \mathcal{A}(a)^{1-\sigma},
$$
thanks to Lemma \ref{dampingproperty}. Therefore,
\begin{align}\label{DampingA2} 
 |\nabla^jA_2|v_h^{(2)}=o_{L^2}(\hbar),\quad j=0,1.
\end{align}
Note that by the symbolic calculus,
$$ ih[Q_{1,h}hD_x, G_{1,h}hD_x]=ih(h\hbar^{-1})^2[Q_{1,h}\hbar D_x, G_{1,h}\hbar D_x] =\mathcal{O}_{\mathcal{L}(L^2)}(h^3\hbar^{-1}),
$$
which is $o_{\mathcal{L}(L^2)}(h\hbar^2)$ since $h^2=o(\hbar^3)$.
For the other terms, if we only apply the symbolic calculus, we will gain only $\mathcal{O}(h^2)+\mathcal{O}(h^3\hbar^{-2})$ for \eqref{(ii)term1} and $\mathcal{O}(h^2\hbar^{-1})$ for \eqref{(ii)term2}, which are not enough to conclude. We need to open the definition of $G_{1,h}$. Since $h^2=o(\hbar^3)$, it suffices to take into account the principal part of $G_{1,h}$. Therefore, without loss of generality, we assume that $G_{1,h}=A_2(x,y)b_1(hD_y)\widetilde{\psi}(\hbar D_x)$, with $b_1(\eta)=\frac{2b(\eta)}{\eta}$. Note that any commutator will generate at least one more $\hbar$, the main contribution of $[h^2D_x^2,G_{1,h}hD_x]$ is $-2ihb_1(hD_y)\widetilde{\psi}(\hbar D_x)h^2D_x^2(\partial_xA_2)v_h^{(2)}$ whose $L^2$ norm is $o(h^3\hbar^{-1})=o(h\hbar^2)$, thanks to \eqref{DampingA2}. Similarly, modulo acceptable errors from the commutators, the main contribution of $ih[\mathcal{A}(a),G_{1,h}hD_x]v_h^{(2)}$ is $$ih^2\hbar^{-1}b_1(hD_y)[\mathcal{A}(a),\widetilde{\psi}(\hbar D_x)\hbar D_x]A_2v_h^{(2)},$$
whose $L^2$ norm is, thanks to \eqref{DampingA2}, bounded by $O(h^2)\|A_2v_h^{(2)}\|_{L^2}=o(h\hbar^2)$. This verifies \eqref{(ii)term1}. By the same argument, to verify \eqref{(ii)term2}, we note that, modulo acceptable errors from commutators, the main contribution of $[h^2D_y^2,G_{1,h}hD_x]v_h^{(2)}$ is
$$ -ih\hbar^{-1}2h\cdot hD_yb_1(hD_y)\widetilde{\psi}(\hbar D_x)\hbar D_x (\partial_yA_2)v_h^{(2)},
$$
which is of size $o_{L^2}(h^2)=o_{L^2}(\hbar^3)$, by \eqref{DampingA2}. This verifies \eqref{(ii)term2} and the proof of Proposition \ref{2averaging} is now complete.
\end{proof}
%%%%%%%%%%%%%%%%%%%%%%%%%%%%%%%%%%%%%%%%%%%
\section{One-dimensional resolvent estimate}\label{sec:1D}

From Proposition \ref{2averaging}, $v_h^{(2)}$ satisfies the equation
\begin{align}\label{finalequation}  (P_h+ih\mathcal{A}(a))v_h^{(2)}+ih\kappa(x)\mathcal{A}(a)^{\frac{1}{2}}b(hD_y)hD_xv_h^{(2)}=r_{4,h}=o_{L^2}(h\hbar^2),
\end{align}
and $\|v_h^{(2)}\|_{L^2}=O(1)$, $\|\mathcal{A}(a)^{\frac{1}{2}}v_h^{(2)}\|_{L^2}=o(\hbar)$. In this section, we are going to show that $\|v_h^{(2)}\|_{L^2}=o(1)$. Since the left hand side of \eqref{finalequation} commutes with $D_y$, by taking the Fourier transform in $y$, are can reduce the analysis to a sequence of one-dimensional problems. 
\subsection{1D resolvent estimate for the H\"older damping}

In order to finish the proof of Theorem \ref{thm:resolvent1}, it remains to prove a one-dimensional resolvent estimate. Below we establish a slightly more general version. By abusing a bit the notation, we denote by $v_{h,E}\in H^2(\T_x)$, solutions of equations
\begin{align}\label{eq:1D}
-h^2\partial_x^2v_{h,E}-Ev_{h,E}+ihW(x)v_{h,E}+h^2\kappa_{h,E}(x)W(x)^{\frac{1}{2}}\partial_xv_{h,E}=r_{h,E}.
\end{align}
We assume that $(\kappa_{h,E})_{h>0,E\in\R}$ is a uniform bounded family in $W^{1,\infty}(\T;\R)$. 
\begin{prop}\label{1DresolventHolder} 
Assume that $W\in \mathcal{D}^{m,2,\theta}(\T_x)$, $\theta\leq \frac{1}{4}$ be a non-negative function such that the set $\{W(x)>0\}$ is a disjoint unions of finitely many intervals $I_j=(\alpha_j,\beta_j)\subset \T_x$, $j=1,\cdots,l$ and 
\begin{align}\label{Holder1} 
C^{-1}(x-\alpha_j)_{+}^{\frac{1}{\theta}}\leq W(x)\leq C(x-\alpha_j)_+^{\frac{1}{\theta}}\; \text{ in $I_j$ near } \alpha_j
\end{align}
and
\begin{align}\label{Holder2} 
 C^{-1}(\beta_j-x)_+^{\frac{1}{\theta}}\leq W(x)\leq C(\beta_j-x)_+^{\frac{1}{\theta}}\;\text{ in $I_j$ near }\beta_j,
\end{align}
for all $j\in\{1,\cdots,l\}$. Then there exists $h_0\in(0,1)$ and $C_0>0$, such that for all $h\in(0,h_0)$ and all $E\in\R$, the solutions $v_{h,E}$ of \eqref{eq:1D} satisfy the uniform estimate
\begin{align}\label{1Du niform}
\|v_{h,E}\|_{L^2}\leq C_0h^{-2-\frac{\theta}{2\theta+1}}\|r_{h,E}\|_{L^2}+C_0h^{-\frac{3\theta+1}{2(2\theta+1)}}\|W^{\frac{1}{2}}v_{h,E}\|_{L^2}.
\end{align}
\end{prop}
\begin{rem}
 We will reduce the proof, in the low-energy hyperbolic regime to a known one-dimensional resolvent estimate (Proposition \ref{1DHolderrough}), which is the main result of \cite{DKl}. However, in the paper of \cite{DKl}, the final gluing argument is not clear to the author. For this reason as well as self-containedness, we will reprove Proposition \ref{1DHolderrough} (thus Theorem \ref{ProceedingAMS}) in Appendix A. 
\end{rem}
We postpone the proof of Proposition \ref{1DresolventHolder} for the moment and proceed on proving Theorem \ref{thm:resolvent1}. Let $\theta=\frac{2}{2\beta+1}$ and $k=2m$, $\delta=\frac{\theta}{2\theta+1}$, and $\hbar=h^{\frac{1+\delta}{2}}=h^{\frac{1}{2}}\delta_h^{\frac{1}{2}}=h^{\frac{3\theta+1}{2(2\theta+1)}}$. Let $W(x)=\mathcal{A}(a)(x)$. By Proposition \ref{convexaveraging}, $W\in \mathcal{D}^{m,2,\theta}(\T_x)$ and $W$ satisfies \eqref{Holder1}, \eqref{Holder2} near the vanishing points inside the damped region. Take the Fourier transform in $y$ for \eqref{finalequation} and denote by $v_{h,n}^{(2)}(x)=\mathcal{F}_y(v_h^{(2)})(x,n)$, we have
$$ (-h^2\partial_x^2+h^2n^2-1+ihW(x)+h^2\kappa(x)W(x)^{\frac{1}{2}}b(hn)\partial_x )v_{h,n}^{(2)}=\mathcal{F}_yr_{4,h}.
$$
Recall that $\|r_{4,h}\|_{L^2(\T^2)}=o(h\hbar^2)=o(h^{2+\delta})$ and $\|W^{\frac{1}{2}}v_h^{(2)}\|_{L^2(\T^2)}=o(\hbar)=o(h^{\frac{1+\delta}{2}})$.
Let $E=1-h^2n^2$ and $\kappa_{h,E}(x)=\kappa(x)b(hn)$ which is uniformly bounded in $W^{1,\infty}(\T)$ with respect to $h$ and $n$. Applying Proposition \ref{1DresolventHolder} for each fixed $n\in\Z$ and then taking the $l_n^2$ norm, by Plancherel we get
$$ \|v_h^{(2)}\|_{L^2}\lesssim h^{-2-\delta}\|r_{4,h}\|_{L^2(\T^2)}+h^{-\frac{1+\delta}{2}}\|W^{\frac{1}{2}}v_h^{(2)}\|_{L^2(\T^2)}=o(1).
$$
Recall that from Proposition \ref{1averaging} and Proposition \ref{2averaging},
$$ v_{h}^{(2)}=(\mathrm{Id}-G_{1,h}hD_x)^{-1}v_h^{(1)}=(\mathrm{Id}-G_{1,h}hD_x)^{-1}\circ e^{G_h}v_h.
$$
Thus $\|v_h\|_{L^2}=o(1)$, and 
this contradicts to \eqref{contradiciton}. The proof of Proposition \ref{resolvent1} (as well as Theorem \ref{thm:resolvent1}) is now complete.

\hspace{0.3cm}

Now we prove Proposition \ref{1DresolventHolder}. In what follows, we note that $\delta=\frac{\theta}{2\theta+1}\leq \frac{1}{6}$.
We argue by contradiction. If \eqref{1Du niform} is untrue, by normalization, we may assume that there exist sequences $h_n\rightarrow 0, (E_n)_{n\in\N}\subset\R$ and $(v_{h_n,E_n})_{n\in\N}\subset L^2$, $(r_{h_n,E_n})_{n\in\N}\subset L^2$, such that 
\begin{align}\label{eq1D}
 (-h_n^2\partial_x^2-E_n+ih_nW(x)+h_n^2\kappa_{h_n,E_n}(x)W(x)^{\frac{1}{2}}\partial_x)v_{h_n,E_n}=r_{h_n,E_n}
\end{align}
and
\begin{align}\label{vhE}
 \|v_{h_n,E_n}\|_{L^2}=1,\; \|r_{h_n,E_n}\|_{L^2}=o(h_n^{2+\delta}),\; \|W^{\frac{1}{2}}v_{h_n,E_n}\|_{L^2}=o(h_n^{\frac{1+\delta}{2}}).
\end{align}
In what follows, when we use the asymptotic notations as small $o$ and $O$, we mean a limit (or bound) \emph{independent} of the sequences $h_n\rightarrow 0$ and $E_n$, as $n\rightarrow\infty$. 
To simplify the notation, we will sometimes omit the subindex $n$ in the sequel.  For a function $f$, sometimes we denote by $f'=\partial_xf$. Also, when we write $\lesssim, \gtrsim$, the implicit bounds are \emph{independent} of $h$ and $E$.

We record an elementary weighted energy identity which allows us to deal with the elliptic regime where $E\ll h^2$.
\begin{lem}[Weighted energy identity]\label{weighted}
Let $w\in C^2(\T;\R)$, then 
\begin{align}\label{eq:weighted}
\int_{\T}w(x)|h\partial_xv_{h,E}|^2dx+\int_{\T}(-\frac{1}{2}h^2\partial_x^2w-Ew)|v_{h,E}|^2dx-\frac{h^2}{2}\int_{\T}(w\kappa_{h,E}W^{\frac{1}{2}})'|v_{h,E}|^2dx=\Re\int_{\T}wr_{h,E}\ov{v}_{h,E}dx.
\end{align}
\end{lem}
\begin{proof}
Multiplying \eqref{eq1D} by $w\ov{v}_{h,E}$ and integrating over $\T$, taking the real part and using the relation $(|v_{h,E}|^2)'=2\Re(v_{h,E}'\ov{v}_{h,E})$,  we get
$$ \Re\int_{\T}h^2(w\ov{v}_h)'v_{h,E}'dx-\int_{\T}Ew|v_{h,E}|^2dx-\frac{h^2}{2}\int_{\T}(w\kappa_{h,E}W^{\frac{1}{2}})'|v_{h,E}|^2dx=\Re\int_{\T}wr_{h,E}\ov{v}_{h,E}dx.
$$
To finish the proof, we just write $\Re(w'\ov{v}_{h,E}v_{h,E}')=\frac{1}{2}w'(|v_{h,E}|^2)'$ and integrate by part.
\end{proof}
By choosing $w=1$ and using the fact that  $\kappa_{h,E}'$ is uniformly bounded in $L^{\infty}(\T)$ and $(W^{\frac{1}{2}})'\lesssim W^{\frac{1}{4}}$, we have
$$ 
h^2|((\kappa_{h,E}W^{\frac{1}{2}})',|v_{h,E}|^2)_{L^2}|\lesssim h^2\|W^{\frac{1}{8}}v_{h,E}\|_{L^2}^2\leq h^2\|v_{h,E}\|_{L^2}^{\frac{3}{2}}\|W^{\frac{1}{2}}v_{h,E}\|_{L^2}^{\frac{1}{2}}=o(h^{\frac{9+\delta}{4}}).
$$
Since $\delta\leq\frac{1}{6}$, we have:
\begin{cor}[Energy identity]\label{finalcor1} 
	There holds
	$$ \|h\partial_xv_{h,E}\|_{L^2}^2-E\|v_{h,E}\|_{L^2}^2=o(h^{2+\delta}).
	$$	
\end{cor}
The proof of Proposition \ref{1DresolventHolder} will be divided into several steps, according to the range of $E$.\\
\noi
$\bullet${\bf (A) Elliptic regime $E\ll h^2$: } 
Recall that $W$ is supported on disjoint intervals $I_j=(\alpha_j,\beta_j)\subset (-\pi,\pi)$, $j=1,\cdots,l$. Therefore, we are able to construct a weight $w\in C^2(\T;\R)$ such that
$$ w\geq c_0>0,\quad w''<0,\quad  \text{ in a neighborhood of } \T\setminus\cup_{j=1}^lI_j.
$$
Therefore, there exists $c_1>0$, sufficiently small, such that
$$ -\frac{1}{2}w''(x)-c_1w>0,\quad \text{ in a neighborhood of } \T\setminus\cup_{j=1}^lI_j.
$$ 
\begin{lem}\label{lem:elliptic}
If $E\leq c_1h^2$, the solution $v_{h_n,E_n}$ satisfies
$$ \|v_{h,E}\|_{L^2}\lesssim h^{-2}\|r_{h,E}\|_{L^2}+\|W^{\frac{1}{2}}v_{h,E}\|_{L^2}.
$$
  
\end{lem}
\begin{proof} 
Since $E\leq c_1h^2$, we have $\frac{1}{2}h^2w''+Ew<0$  in a neighborhood of  $\T\setminus\cup_{j=1}^lI_j.$ Thus there exists a compact set $K\subset \cup_{j=1}^lI_j$ such that
$$ \int_{\T}(\frac{1}{2}h^2\partial_x^2w+Ew)|v_{h,E}|^2dx\leq \int_{K}(\frac{1}{2}h^2\partial_x^2w+Ew)|v_{h,E}|^2dx\lesssim h^2\int_{\T}W(x)|v_{h,E}|^2dx. 
$$
Then applying Lemma \ref{weighted}, we have
\begin{align*}
c_0\|h\partial_xv_{h,E}\|_{L^2}^2\leq &\int_{\T}w(x)|h\partial_xv_{h,E}|^2dx\\
\leq &\Re\int_{\T}wr_{h,E}\ov{v}_{h,E}dx+h^2\int_{\T}W(x)|v_{h,E}|^2dx+C\frac{h^2}{2}\int_{\T}(w\kappa_{h,E}W^{\frac{1}{2}})'|v_{h,E}|^2dx.
\end{align*}
Note that $|(w\kappa_{h,E}W^{\frac{1}{2}})'|\lesssim W^{\frac{1}{4}}(x)$, by interpolation
$$ \|W^{\frac{1}{8}}v_{h,E}\|_{L^2}^2\leq \|W^{\frac{1}{2}}v_{h,E}\|_{L^2}^{\frac{1}{2}}\|v_{h,E}\|_{L^2}^{\frac{3}{2}}
$$
and Young's inequality,
we deduce that
$$ \|v_{h,E}'\|_{L^2}\leq Ch^{-1}\|v_{h,E}\|_{L^2}^{\frac{1}{2}}\|r_{h,E}\|_{L^2}^{\frac{1}{2}}+C_{\epsilon}\|W^{\frac{1}{2}}v_{h,E}\|_{L^2}+\epsilon\|v_{h,E}\|_{L^2},\;\forall \epsilon>0.
$$
By the Poincar\'e-Wirtinger inequality,
$$ \big\|v_{h,E}-\widehat{v}_{h,E}(0)\big\|_{L^2(\T)}\leq C\|v_{h,E}'\|_{L^2},
$$ 
where $\widehat{v}_{h,E}(0)=\frac{1}{2\pi}\int_{\T}v_{h,E}$. Combining with the fact that $\int_{T}W>0$ and the elementary inequality
$$ \Big(\int_{\T}Wdx\Big)|\widehat{v}_{h,E}(0)|^2\leq C\int_{\T}W(x)|v_{h,E}(x)|^2dx+C\int_{\T}W(x)|v_{h,E}(x)-\widehat{v}_{h,E}(0)|^2dx,
$$
we deduce that
$$ \|v_{h,E}\|_{L^2}+\|v_{h,E}'\|_{L^2}\lesssim h^{-1}\|wv_{h,E}\|_{L^2}^{\frac{1}{2}}\|r_{h,E}\|_{L^2}^{\frac{1}{2}}+\|W^{\frac{1}{2}}v_{h,E}\|_{L^2}.
$$
Using Young's inequality again to absorb $\|v_{h,E}\|_{L^2}$ to the left, we complete the proof of Lemma \ref{lem:elliptic}. 
\end{proof}

\noi
$\bullet${\bf (B) High energy hyperbolic regime $E>h^{1+\delta}$: } In this regime, we put the damping terms to the right as remainders and use the estimate from the geometric control as a black box. Let us recall:
\begin{lem}\label{geometriccontrol} 
Let $I\subset \T$ be a non-empty open set. Then there exists $C=C_I>0$, such that for any $v\in L^2(\T), f_1\in L^2(\T), f_2\in H^{-1}(\T)$,  $\lambda\geq 1$, if
$$ (-\partial_x^2-\lambda^2)v=f_1+f_2,
$$
we have
$$ \|v\|_{L^2(\T)}\leq C\lambda^{-1}\|f_1\|_{L^2(\T)}+C\|f_2\|_{H^{-1}(\T)}+C\|v\|_{L^2(I)}.
$$
\end{lem}
The proof is standard and can be found, for example in \cite{Bu19} (Proposition 4.2). In the one-dimensional setting, a straightforward proof using the multiplier method is also available.
Consequently, we have:
\begin{cor}\label{cor:hyperbolicB}
If $E>h^{1+\delta}$, then
$$ \|v_{h,E}\|_{L^2(\T)}\lesssim h^{-\frac{3+\delta}{2}}\|r_{h,E}\|_{L^2(\T)}+h^{-\frac{1+\delta}{2}}\|Wv_{h,E}\|_{L^2(\T)}+\|W^{\frac{1}{2}}v_{h,E}\|_{L^2}.
$$
\end{cor}
\begin{proof}
Let $\lambda=h^{-1}E^{\frac{1}{2}}(\geq h^{-\frac{1-\delta}{2}})$, then
$$ (-\partial_x^2-\lambda^2)v_{h,E}=h^{-2}r_{h,E}-ih^{-1}Wv_{h,E}-\kappa_{h,E}W^{\frac{1}{2}}\partial_xv_{h,E}.
$$
Applying Lemma \ref{geometriccontrol} to $v=v_{h,E}, f_1=h^{-2}r_{h,E}-ih^{-1}Wv_{h,E}$, $f_2=-\kappa_{h,E}W^{\frac{1}{2}}v_{h,E}'$ with $I=(\alpha_1+\epsilon_0,\beta_1-\epsilon_0)$ for some $\epsilon_0<\frac{\beta_1-\alpha_1}{2}$, we get
\begin{align*}
\|v_{h,E}\|_{L^2(\T)}\lesssim_{\epsilon_0}& \lambda^{-1}\|h^{-2}r_{h,E}-ih^{-1}Wv_{h,E}\|_{L^2(\T)}+\|\kappa_{h,E}W^{\frac{1}{2}}v_{h,E}'\|_{H^{-1}(\T)} +\|v_{h,E}\|_{L^2(I)}\\
\lesssim &h^{-\frac{3+\delta}{2}}\|r_{h,E}\|_{L^2}+h^{-\frac{1+\delta}{2}}\|Wv_{h,E}\|_{L^2}+\|(\kappa_{h,E}W^{\frac{1}{2}}v_{h,E})'-(\kappa_{h,E}W^{\frac{1}{2}})'v_{h,E}\|_{H^{-1}(\T)}.
\end{align*}
Since $(\kappa_{h,E}W^{\frac{1}{2}})'\lesssim W^{\frac{1}{4}}$,
the last term on the right hand side is bounded by
$$ C\|W^{\frac{1}{4}}v_{h,E}\|_{L^2}\leq C\|W^{\frac{1}{2}}v_{h,E}\|_{L^2}^{\frac{1}{2}}\|v_{h,E}\|_{L^2}^{\frac{1}{2}}. 
$$
By Young's inequality, we obtain the desired estimate.
\end{proof}

\noi
$\bullet${\bf (C) Low energy hyperbolic regime: $c_1h^2<E\leq h^{1+\delta}$}

Again we denote by $\lambda=h^{-1}E^{\frac{1}{2}}$, then $\lambda\leq h^{-\frac{1-\delta}{2}}$. In this situation, the non self-adjoint term $h^2\kappa_{h,E}W^{\frac{1}{2}}v_{h,E}'$ can be absorb to the right as a remainder:

\begin{lem}\label{newerror}
Assume that $E\leq h^{1+\delta}$, then
$$ \|W^{\frac{1}{2}}v_{h,E}'\|_{L^2(\T)}\lesssim h^{-2}\|r_{h,E}\|_{L^2}+h^{-\frac{1-\delta}{2}}\|W^{\frac{1}{2}}v_{h,E}\|_{L^2}+\|W^{\frac{1}{2}}v_{h,E}\|_{L^2}^{\frac{1}{2}}\|v_{h,E}\|_{L^2}^{\frac{1}{2}}.
$$
\end{lem}
\begin{proof}
With the notation $\lambda=h^{-1}E^{\frac{1}{2}}$, $v_{h,E}$ solves the equation
\begin{align}\label{eq:vlambda}
-v_{h,E}''-\lambda^2v_{h,E}+ih^{-1}Wv_{h,E}=h^{-2}r_{h,E}-\kappa_{h,E}W^{\frac{1}{2}}v_{h,E}'.
\end{align}
Doing the integration by part and inserting the equation \eqref{eq:vlambda}, we have
\begin{align*}
\Re\int_{\T}Wv_{h,E}'\ov{v}_{h,E}'dx=-&\Re\int_{\T}W'v_{h,E}'\ov{v}_{h,E}dx-\Re\int_{\T}Wv_{h,E}''\ov{v}_{h,E}dx\\
=-&\Re\int_{\T}W'v_{h,E}'\ov{v}_{h,E}dx\\-&\Re\int_{\T}W\ov{v}_{h,E}(-\lambda^2v_{h,E}+ih^{-1}Wv_{h,E}-h^{-2}r_{h,E}+\kappa_{h,E}W^{\frac{1}{2}}v_{h,E}' )dx.
\end{align*}
By writing $\Re(v_{h,E}'\ov{v}_{h,E})=\frac{1}{2}(|v_{h,E}|^2)'$, we have
\begin{align*}
-\Re\int_{\T}W'v_{h,E}'\ov{v}_{h,E}dx=&\frac{1}{2}\int_{\T}W''|v_{h,E}|^2dx\lesssim \|W^{\frac{1}{4}}v_{h,E}\|_{L^2}^2\leq \|W^{\frac{1}{2}}v_{h,E}\|_{L^2}\|v_{h,E}\|_{L^2},
\end{align*}
where we used $|W''|\lesssim W^{\frac{1}{2}}$.
Writing 
$$\Re\int_{\T}W\ov{v}_{h,E}\kappa_{h,E}W^{\frac{1}{2}}v_{h,E}'dx=\frac{1}{2}\int_{\T}W^{\frac{3}{2}}\kappa_{h,E}(|v_{h,E}|^2)'dx=-\frac{1}{2}\int_{\T}(\kappa_{h,E}W^{\frac{3}{2}})'|v_{h,E}|^2dx,
$$
one verifies that
$$ \Big|\Re\int_{\T}W\ov{v}_{h,E}\cdot\kappa_{h,E}W^{\frac{1}{2}}v_{h,E}'dx \Big|\lesssim \|W^{\frac{1}{2}}v_{h,E}\|_{L^2}^2.
$$
Since $\lambda^2\leq h^{-(1-\delta)}$, we have
$$ \Big|\int_{\T}\kappa_{h,E}W^{\frac{3}{2}}\lambda^2|v_{h,E}|^2dx\Big|\lesssim h^{-(1-\delta)}\|W^{\frac{1}{2}}v_{h,E}\|_{L^2}^2.
$$
The last term $$\Re\int_{\T}W\ov{v}_{h,E}(-ih^{-1}Wv_{h,E}-h^{-2}r_{h,E})dx=h^{-2}\Re\int_{\T}W\ov{v}_{h,E}r_{h,E}dx,$$ and it can be easily controlled by
$ h^{-2}\|r_{h,E}\|_{L^2}\|W^{\frac{1}{2}}v_{h,E}\|_{L^2}.
$
Putting the bounds together, we complete the proof of Lemma \ref{newerror}.
\end{proof}

The importance of Lemma \ref{newerror} is that, in the low energy hyperbolic regime, the term $h^2\kappa_{h,E}W^{\frac{1}{2}}v_{h,E}'$ has the same size $o(h^{2+\delta})$ in $L^2$, and thus can be absorbed as a remainder.

At this state, we are able to apply the following 1D resolvent estimate for the H\"older-like damping in \cite{DKl}:
\begin{prop}[\cite{DKl}]\label{1DHolderrough}
Let $\gamma\geq 0$. Assume that $W=W(x)\geq 0$ and $\{W>0\}$ is disjoint unions of intervals $I_j=(\alpha_j,\beta_j), j=1,2,\cdots, l$ and that for each $j\in\{1,\cdots,l\}$,
$$ C_1V_j(x)\leq W(x)\leq C_2V_j(x) \text{ on } (\alpha_j,\beta_j),
$$
where $V_j(x)>0$ are continuous functions on $(\alpha_j,\beta_j)$ such that
\begin{equation}\label{Vj}
V_j(x)=\left\{
\begin{aligned}
&(x-\alpha_j)^{\gamma},\; \alpha_j<x<\frac{3\alpha_j+\beta_j}{4} \\
&(\beta_j-x)^{\gamma},\; \frac{\alpha_j+3\beta_j}{4}<x<\beta_j.
\end{aligned}
\right.
\end{equation}
Then there exist $h_0>0$, $c_1>0$, $C>0$, such that for all $0<h<h_0$, $\sqrt{c_1}\leq \lambda\leq h^{-\frac{1-\delta}{2}}$ and all solutions $v_{h,\lambda}$ of the equation 
$$ -v_{h,\lambda}''-\lambda^2v_{h,\lambda}+ih^{-1}W(x)v_{h,\lambda}=r_{h,\lambda},
$$
we have
\begin{align}\label{1Dresolvent} 
 \|v_{h,\lambda}\|_{L^2}\leq Ch^{-\frac{1}{\gamma+2}}\|r_{h,\lambda}\|_{L^2}.
\end{align}
%{\color{red}{Furthermore, when $W\in C^2(\T)$ and the pointwise bound holds
%$$ |W^{(k)}|\lesssim W^{1-\frac{k}{\gamma}} \text{ for } k=1,2,
%$$
%we have the same estimate \eqref{1Dresolvent} if $W$ satisfies $W(x)\leq C_2V_j(x)$ on each $(\alpha_j,\beta_j)$ only.}} 
\end{prop}
The proof of Proposition \ref{1DHolderrough} will be given in Appendix \ref{gap}.

 Now applying Proposition \ref{1DHolderrough} for $\gamma=\frac{1}{\theta}$ (then $\frac{1}{\gamma+2}=\delta=\frac{\theta}{2\theta+1}$), $v_{h,\lambda}=v_{h,E}$ and $r_{h,\lambda}=h^{-2}r_{h,E}-\kappa_{h,E}W^{\frac{1}{2}}v_{h,E}'$ in our previous setting, combining with Lemma \ref{newerro}, we deduce that
when $c_1h^2\leq E<h^{1+\delta}$,
\begin{align}\label{regime:final}
\|v_{h,E}\|_{L^2}\leq Ch^{-2-\delta}\|r_{h,E}\|_{L^2}+h^{-\frac{1+\delta}{2}}\|W^{\frac{1}{2}}v_{h,E}\|_{L^2}+h^{-\delta}\|W^{\frac{1}{2}}v_{h,E}\|_{L^2}^{\frac{1}{2}}\|v_{h,E}\|_{L^2}^{\frac{1}{2}}.
\end{align}
Finally, to get a contradiction, we denote three index sets for three regimes:
$$ \mathcal{E}_A:=\{n: E_n\leq c_1h_n^2\},\;\mathcal{E}_B:=\{n: E_n>h_n^{1+\delta}\},\; \mathcal{E}_C:=\{n: c_1h_n^2<E_n\leq h_n^{1+\delta}\}.
$$ 
Clearly, $\N=\mathcal{E}_A\cup\mathcal{E}_B\cup\mathcal{E}_C$. From Lemma \ref{lem:elliptic}, Corollary \ref{cor:hyperbolicB} and \eqref{regime:final}, we deduce that
$$ \lim_{\substack{n\rightarrow\infty\\
n\in S }}\|v_{h_n,E_n}\|_{L^2}=0,\quad \forall S\in\{A,B,C\},
$$
and this implies that $\|v_{h_n,E_n}\|_{L^2}=o(1)$, a contradiction to \eqref{vhE}.
This finishes the proof of Proposition \ref{1DresolventHolder}. In summary, the proof of Theorem \ref{thm:resolvent1} is now complete.
%%%%%%%%%%%%%%%%%%%%%%%%%%%%%%%%%%%%%%%%%%%%%%

\section{Optimality}\label{lowerbound} 
Let $r_0\in(0,\pi)$ and $\beta\geq 10$. Consider the function
$$ a_0(x,y)=c_{\beta}\frac{(r_0^2-x^2-y^2)_+^{\beta}}{(r_0+|x|)^{\beta+\frac{1}{2}}},\quad c_{\beta}=2\pi\Big(\int_{-1}^1(1-\rho^2)^{\beta}\Big)^{-1}.
$$
Obviously, the damped region of $a_0$ is the disc $\omega_0:=\{|(x,y)|<r_0\}$, which is strictly convex with positive curvature. Moreover, close to $|(x,y)|=r_0$, $a_0(x,y)$ vanishes exactly of order $(r_0-|(x,y)|)_+^{\beta}$. Therefore, $a_0$ verifies all the hypothesis of the damping function in Theorem \ref{thm:resolvent1}.

The averaged damping of $a_0$ along $\mathrm{e}_2$ is given by
$$ b_0(x):=\frac{1}{2\pi}\int_{-\pi}^{\pi}a_0(x,y)dy,
$$
and direct computation yields
$$ b_0(x)=(r_0-|x|)_+^{\beta+\frac{1}{2}},\quad \text{ as }|x| \text{ close to }r_0.
$$
The goal of this section is to show that 
\begin{align}\label{lastsection}  
\|(-h^2\Delta-1+iha_0)^{-1}\|_{\mathcal{L}(L^2)}\gtrsim h^{-2-\frac{1}{\gamma+2}}, \;\text{ with } \gamma=\beta+\frac{1}{2}.
\end{align}
 The idea is to use the quasimodes constructed in \cite{Kl} that saturates the lower bound of $$\|(-h^2\Delta-1+ihb_0(x))^{-1}\|_{\mathcal{L}(L^2)}$$ and use the averaging argument to transform back to obtain the desired quasimodes for the operator $-h^2\Delta-1+iha_0(z)$. In order to control remainders appearing in the normal form analysis, rather than using the construction in \cite{Kl} as a blackbox, we should keep track of the regularity (anisotropic) of the quasimodes. For this reason, we will briefly review the construction of \cite{Kl} and prove some extra estimates in the next subsection.
\subsection{Estimates for the $\T^2$ quasimodes}
First we review the construction of quasimodes of $-h^2\Delta-1+ihb_0(x)$ in \cite{Kl}. The original idea for the construction dates back to the appendix in \cite{AL14} by Nonnenmacher.

Recall that $b_0(x)=(r_0-|x|)_+^{\gamma}$, we consider the ansatz
$$ u_{k}(x,y):=\mathrm{e}^{iky}v_{h_k}(x),
$$ 
with $0<h_k\lesssim \frac{1}{|k|}$ to be specified later. We fix $\delta=\frac{1}{\gamma+2}$ throughout this section. As before, we will drop the dependence in $k$ and write simply $h=h_k, v_h=v_{h_k}$. Plugging into the equation $$(-h^2\Delta-1+ihb_0(x))u_h=O_{L^2}(h^{2+\delta}),$$ we would like $v_h$ to satisfy
\begin{align}\label{quasimodev_h}
	-h^2\partial_x^2v_h+ihb_0(x)v_h=(1-h^2k^2)v_h+O_{L^2}(h^{2+\delta}),\quad \|v_h\|_{L^2}\sim 1.
\end{align}
This amounts to solve the eigenvalue problem:
\begin{align}\label{EigenPb} 
	-h^2\partial_x^2v_h+ihb_0(x)v_h=\lambda_h^2v_h
\end{align}
with $\lambda_h=Ch+O(h^{1+\delta})$. We consider the even eigenfunction $v_h(x)=v_h(|x|)$ with
\begin{align*}
	v_h(x)=\begin{cases} 
		& \!\!\!\!\!\!v_{h,l}(x),\; 0\leq x<r_0,\\
		& \!\!\!\!\!\!v_{h,r}(x),\; r_0\leq x<\pi
	\end{cases}
\end{align*}
where 
$ v_{h,r}(x)=\cos\big(\frac{\lambda_h}{h}x\big).
$ The left function $v_{h,l}$ should satisfy the equation
\begin{align}\label{eq:vhr} 
	-h^2\partial_x^2v_{h,l}+ih(r_0-x)_+^{\gamma}v_{h,l}=\lambda_h^2v_{h,l},\; 0\leq x<r_0.
\end{align}
In order to ensure $v_h\in H^2$, the function $v_{h,l}$ should satisfy the compatibility condition
\begin{align}\label{compatibility} 
	v_{h,l}( r_0)=\cos\big(\frac{\lambda_hr_0}{h}\big),\quad v'_{h,l}( r_0)=-\frac{\lambda_h}{h}\sin\big(\frac{\lambda_hr_0}{h}\big).
\end{align} 
Since the mass in the damped region is very small compared with the total mass, and the amplitude of the transmitted mass should be the same size as the reflected mass, so we expect that $|v_{h,l}'(r_0)|\gg |v_{h,l}(r_0)|$. As $\frac{\lambda_h}{h}$ is of size $O(1)$, the principal part of the argument $\frac{\lambda_h r_0}{h}$ must belong to $\pi\big(l+\frac{1}{2}\big), l\in\Z$. Therefore, we take the following ansatz for $\lambda_h$:
\begin{align}\label{lambdahansatz} 
	\lambda_h=\frac{\pi\big(l+\frac{1}{2}\big)h}{y_0}+O(h^{1+\delta}).
\end{align}

Next, denote by $F(x;\theta)$ the $H^1(\R_+)$ solution of the Neumann problem (with a parameter $\theta\in\C$)
\begin{align}\label{variational}
	\begin{cases}
		&\!\!\!\!-F''(x)+ix^{\gamma}F-\theta F=0,\; x>0\\
		&\!\!\!\! F'(0)=1,
	\end{cases} 
\end{align} 
then $v_{h,l}(x)$ takes the form $h^{\delta}\alpha_hF\big(\frac{r_0-x}{h^{\delta}};\theta_h\big)$, with $\theta_h=h^{-\frac{2(\gamma+1)}{\gamma+2}}\lambda_h^2$, and a constant $O(1)=\alpha_h\in\C$ to be determined in order to match the compatibility condition \eqref{compatibility}.

Denote by $\mu_0$ be the lowest Neumann eigenvalue of the operator $-\partial_y^2+x^{\gamma}$ on $L^2(\R_+)$ and $F_0$ the Dirichelet trace $F(x;0)|_{x=0}$. It was shown  that $\mu_0>0$ (Lemma 4.1 of \cite{Kl}). Moreover, by using the implicit function theory (Lemma 4.2 of \cite{Kl}), there exists a uniform constant $C_0>0$, such that for all $|\theta|\leq \frac{\mu_0}{2}$, the (unique) solution $F(y;\theta)$ of the Neumann problem \eqref{variational} satisfies
\begin{align}\label{nonvanishing} 
	\frac{1}{C_0}\leq |F(0;\theta)|\leq C_0.
\end{align}
In particular, $F_0=F(0;0)\neq 0$.  

Now we are ready to solve \eqref{eq:vhr} with the compatibility condition \eqref{compatibility}. In view of \eqref{lambdahansatz}, we precise the ansatz of $\lambda_h$ as
$$ \lambda_h=\frac{\pi\big(l+\frac{1}{2}\big)h}{r_0}+\gamma_h h^{1+\delta},\quad O(1)=\gamma_h\in\C. 
$$
Now $\alpha_h,\gamma_h$ are the parameters to be determined.
Plugging into \eqref{compatibility} with $$v_{h,l}(x)=h^{\delta}\alpha_hF\big(\frac{r_0-x}{h^{\delta}};\theta_h\big),$$
we obtain a system
\begin{align}\label{compatibility2}& \alpha_hF(0;\theta_h)=(-1)^{l+1}h^{-\delta}\sin(\gamma_hr_0h^{\delta}),\notag \\ 
	& \alpha_h=(-1)^{l}\big[\frac{\pi\big(l+\frac{1}{2}\big)}{r_0}+\gamma_hh^{\delta}\big]\cos(\gamma_hr_0h^{\delta}).  
\end{align}
Since $|\theta_h|\sim h^{\frac{2}{\gamma+2}}$, by Taylor expansion of sin and cos, the leading term of $\gamma_h$ should be $\gamma_0:=\frac{\pi\big(l+\frac{1}{2}\big)F_0}{r_0^2}$ and the leading term for $\alpha_h$ should be $\alpha_0=(-1)^{l}\frac{\pi\big(l+\frac{1}{2}\big)}{y_0}$. Fix the number $l$, by using the implicit function theorem, the solution $(\alpha_h,\gamma_h)$ to \eqref{compatibility2} exists for $0\leq h\ll 1$ and
$$ |(\alpha_h,\gamma_h)-(\alpha_0,\gamma_0)|\lesssim h^{\delta}.
$$
For the detailed argument, we refer Lemma 4.3 and Lemma 4.4 of \cite{Kl}.

Finally, we take $h=h_k=\frac{r_0}{\sqrt{k^2r_0^2+\pi^2(l+\frac{1}{2})^2 }}$, then $1-h_k^2k^2=\lambda_{h_k}^2+O(h_k^{2+\delta})$, hence 
\begin{align}\label{T2quasimodes} 
	u_h(x,y)=e^{iky}v_{h}(x)=e^{iky}\big[\cos\big(\frac{\lambda_h |x|}{h}\big)\mathbf{1}_{r_0<|x|\leq \pi}+h^{\delta}\alpha_hF\big(\frac{r_0-|x|}{h^{\delta}};\eta_h\big) \mathbf{1}_{|x|\leq r_0}\big]
\end{align}
are the desired $\T^2$ quasimodes, satisfying
$$ -h^2\Delta u_h-u_h+ihb_0(x)u_h=O_{L^2}(h^{2+\delta}).
$$
We are going to prove more 
estimates on $u_h$ as well as its transformations:
\begin{prop}\label{estimatesT2quasimodes} 
	Let $u_h$ is given by \eqref{T2quasimodes} and $\psi\in C_c^{\infty}(\R)$ be a bump function supported near $0$. Then  $\widetilde{u}_h:=\psi(h^{\frac{1+\delta}{2}} D_x)u_h$ satisfies
	$$ (-h^2\Delta-1+ihb_0(x))\widetilde{u}_h=O_{L^2}(h^{2+\delta}).
	$$
	Moreover, there exist uniform constants $h_0>0$, $C_1>0$, such that for all $0<h<h_0$:
	\begin{itemize}
		\item[$\mathrm{(a)}$] $\frac{1}{C _1}\leq \|\widetilde{u}_h\|_{L^2}\leq C_1,\quad $ $\|b_0(x)^{\frac{1}{2}}\widetilde{u}_h\|_{L^2}\leq C_1h^{\frac{1+\delta}{2}}$.
		\medskip
		
		\item[$\mathrm{(b)}$]
		$\|h^j\partial_x^j\widetilde{u}_h\|_{L^2}\leq C_1h^{j(1-\delta)+\frac{\delta}{2}}$ and $\|h^j\partial_y^j\widetilde{u}_h\|_{L^2}\leq C_1$, for $j=1,2$.
		\medskip
		
		\item[$\mathrm{(c)}$] $ \|b_0'(x)\partial_x\widetilde{u}_h\|_{L^2}\leq C_1h^{\delta},\quad \Big\|\Big(\int_{-\pi}^y(\partial_xa_0(x,y')-b_0'(x))dy'\Big)\partial_x\widetilde{u}_h\Big\|_{L^2}\leq C_1h^{\delta}$.
%		\medskip
%		\item[$\mathrm{(d)}$]
%		Let $w_h=\big(1-\chi(h^{\frac{1+\delta}{2}}D_x)\big)u_h$, where $\chi\in C_c^{\infty}(\R)$ such that $\chi\equiv 1$ near $0$, then  $\|w_h\|_{L^2}\leq C_1h^2$.
	\end{itemize}
Moreover, for any $S^0$ symbol $g(x,y,\eta)$ in $(y,\eta)$, smoothly depending on $x$, if $\widetilde{v}_h=e^{-\mathrm{Op}_h^w(g)}\widetilde{u}_h$, the above estimates $\mathrm{(a)},\mathrm{(b)},\mathrm{(c)}$ still hold, up to some different uniform constant $C_2>0$.
\end{prop}
We need a Lemma:
\begin{lem}\label{AprioriF} 
	Let $\mu_0$ be the least eigenvalue of the operator $\mathcal{A}_{\gamma}=-\partial_x^2+x^{\gamma}$ on $L^2(\R_+)$ associated to the Neumann boundary condition. Then there exists a uniform constant $C>0$, such that for all $|\theta|\leq \frac{\mu_0}{4}$, the solution $F(x;\theta)$ of \eqref{variational} satisfies
	$$ \|F\|_{H^2(\R_+)}\leq C.
	$$
\end{lem}
\begin{proof}
	Take $H\in C_c^{\infty}([0,\infty))$ such that $H(0)=0$ and $H'(0)=1$. Consider $\widetilde{F}:=F-H$, then
	$$ -\widetilde{F}''+ix^{\gamma}\widetilde{F}-\theta \widetilde{F}=W,
	$$
	with $W=H''-ix^{\gamma}H+\theta H$. Multiplying by $\ov{\widetilde{F}}$ and doing the integration by part, we get
	\begin{align*}
		\int_0^{\infty}|\widetilde{F}'(x)|^2dx+i\int_0^{\infty}x^{\gamma}|\widetilde{F}(x)|^2dx-\theta\int_0^{\infty}|\widetilde{F}(x)|^2dx=\int_0^{\infty}W(x)\ov{\widetilde{F}}(x)dx.
	\end{align*}
	Taking the real part and imaginary part, we get
	$$ \|\widetilde{F}'\|_{L^2(\R_+)}^2\leq |\Re\theta|\|\widetilde{F}\|_{L^2(\R_+)}^2+\|W\|_{L^2(\R_+)}\|\widetilde{F}\|_{L^2(\R_+)},
	$$
	and
	$$ \|x^{\frac{\gamma}{2}}\widetilde{F}\|_{L^2(\R_+)}^2\leq |\Im \theta|\|\widetilde{F}\|_{L^2(\R_+)}^2+\|W\|_{L^2(\R_+)}\|\widetilde{F}\|_{L^2(\R_+)}.
	$$
	Adding two inequalities above, using $\mathcal{A}_{\gamma}-\mu_0\geq 0$ and the fact that $|\theta|\leq \frac{\mu_0}{4}$, we get
	$$ \mu_0\|\widetilde{F}\|_{L^2(\R_+)}^2\leq \frac{\mu_0}{2}\|\widetilde{F}\|_{L^2(\R_+)}^2+2\|W\|_{L^2(\R_+)}\|\widetilde{F}\|_{L^2(\R_+)}.
	$$
	This proves the boundedness of $\|\widetilde{F}\|_{L^2(\R_+)}+\|\partial_x\widetilde{F}\|_{L^2(\R_+)}+\|x^{\frac{\gamma}{2}}\widetilde{F}\|_{L^2(\R_+)}$ in terms of $\|W\|_{L^2(\R_+)}$. Replacing $\widetilde{F}$ by $\widetilde{F}'$, the same argument yields the boundedness of $\|\widetilde{F}''\|_{L^2(\R_+)}$ in terms of $\|W\|_{H^1(\R_+)}$. Since the $H^1$ bound for $W$ is uniform with respect to $|\theta|\leq \frac{\mu_0}{4}$, the proof of Lemma \ref{AprioriF} is complete.

\end{proof}

\begin{proof}[Proof of Proposition \ref{estimatesT2quasimodes}]
Denote by $\hbar=h^{\frac{1+\delta}{2}}$ as before. The fact that $\widetilde{u}_h=\psi(\hbar D_x)u_h$ satisfies the equation of quasimodes is clear from Lemma \ref{newerro}, up to changing $o_{L^2}(h\hbar^2)$ there to  $O_{L^2}(h\hbar^2)$. We need to prove estimates only.

At the first step, we prove the same estimates (a),(b),(c) for $u_h$. The inequality (a) is clear by the construction and Lemma \ref{apriori}.
For (b), since $|k|\sim \frac{1}{h}$, we have $\|h^j\partial_y^ju_h\|_{L^2}\lesssim 1$ for all $j\in\N$. 
The derivatives of $u_h$ in $x$ satisfy
	$$ |\partial_x^ju_h|\lesssim 1+\mathbf{1}_{|x|\leq r_0}h^{-(j-1)\delta}|(\partial_x^jF)\big(\frac{r_0-|x|}{h^{\delta}};\theta_h\big)|,
	$$
	by Lemma \ref{AprioriF}, we deduce that $\|h^j\partial_x^ju_h\|_{L^2}\lesssim h^{j-(j-\frac{1}{2})\delta}$, for $j=1,2$.
	Finally, since
	$ \partial_xu_h=F'\big(\frac{r_0-|x|}{h^{\delta}};\theta_h\big),
	$ on supp$(b_0')$, we have
	\begin{align}\label{b'dx} \|b_0'(x)\partial_xu_h\|_{L^2}\lesssim \big\|(r_0-|x|)_+^{\gamma-1}F'\big(\frac{r_0-|x|}{h^{\delta}};\theta_h\big)\big\|_{L^2}\lesssim h^{\frac{\delta}{2}}\cdot h^{\delta(\gamma-1)}=h^{\frac{\gamma-1/2}{\gamma+2}}=h^{1-\frac{5\delta}{2}}\leq h^{\delta},
	\end{align}
	thanks to $\delta=\frac{1}{\gamma+2}$ and $\gamma>2$. Next,
	\begin{align*}
	\Big\|\Big(\int_{-\pi}^y(\partial_xa_0(x,y')-b_0'(x))dy'\Big)\partial_xu_h\Big\|_{L^2}\lesssim &\Big\|\Big(\int_{-\pi}^{\pi}|\partial_xa_0(x,y')|dy'\Big)\partial_xu_h
	\Big\|_{L^2(|x|\leq r_0)}+\||b_0'(x)|\partial_xu_h\|_{L^2(|x|\leq r_0)}.
	\end{align*}
    Since for $|x|\leq r_0$,
    $$ \int_{-\pi}^{\pi}|\partial_xa_0(x,y')|dy'=\int_{-\sqrt{r_0^2-x^2}}^{\sqrt{r_0^2-x^2}}|\partial_xa_0(x,y')|dy'\lesssim (r_0-|x|)_+^{\gamma-1},
    $$
   the same computation as \eqref{b'dx} yields 
   $$	\Big\|\Big(\int_{-\pi}^y\big(\partial_ya_0(x,y')-b_0'(x)\big)dy'\Big)\partial_xu_h\Big\|_{L^2}\lesssim h^{1-\frac{5\delta}{2}}\leq h^{\delta}.
   $$

As the second step, we deal with $\widetilde{u}_h$. As the Fourier multiplier $\psi(\hbar D_x)$ commutes with derivatives, the upper bounds for $\|\widetilde{u}_h\|_{L^2}$, $\|h^j\partial_x^j\widetilde{u}_h\|_{L^2}$ and $\|h^j\partial_y^j\widetilde{u}_h\|_{L^2}$ follow directly from the estimates for $u_h$, hence (b) holds for $\widetilde{u}_h$. To prove the lower bound of $\|\widetilde{u}_h\|_{L^2}$, we take a smooth function $\chi\in C^{\infty}(\R)$ such that $\chi\equiv 1$ on the support of $1-\psi$. Then we write $$u_h-\widetilde{u}_h=\big(\partial_x^{-1}\chi(\hbar D_x)\big)\cdot \partial_x(1-\psi(\hbar D_x))u_h.
	$$ 	 
	Since $\|\partial_xu_h\|_{L^2}\lesssim h^{-\delta/2}$ and $\|\partial_x^{-1}\chi(\hbar D_x)\|_{\mathcal{L}(L^2)}\lesssim \hbar$, we obtain that $\|u_h-\widetilde{u}_h\|_{L^2}\lesssim \hbar h^{-\delta}=o(1)$. Therefore, the assertions (a) follows for $\widetilde{u}_h$.  By the commutator estimate
	$$ \|[f,\psi(\hbar D_x)]\|_{\mathcal{L}(L^2)}\lesssim \hbar,
	$$ 
	for $f\in W^{1,\infty}$ and the fact that $\|\partial_xu_h\|_{L^2}\lesssim h^{-\frac{\delta}{2}}$, we deduce that
	$$ \Big\|\Big(\int_{-\pi}^y(\partial_xa_0(x,y')-b_0'(x))dy'\Big)\partial_x\widetilde{u}_h
	\Big\|_{L^2}+\|b_0'(x)\partial_x\widetilde{u}_h\|_{L^2}\lesssim \hbar \cdot h^{-\frac{\delta}{2}}+h^{\delta}\lesssim h^{\delta}.
	$$
	
	The last step is to prove estimates (a),(b),(c) for $\widetilde{v}_h=e^{-G_h}\widetilde{u}_h$ where $G_h=\mathrm{Op}_h^w(g)$. Since $e^{-G_h}$ is invertible and is uniformly bounded, we get $\|\widetilde{v}_h\|_{L^2}\sim 1$. Moreover, since $G_h$ commutes with the multiplication by functions depending only in $x$, we have $\|b_0(x)^{\frac{1}{2}}\widetilde{v}_h\|_{L^2}\lesssim h^{\frac{1+\delta}{2}}$. So (a) holds for $\widetilde{v}_h$.
	
	 To prove (b),(c) for $\widetilde{v}_h$, from the same estimates for $\widetilde{u}_h$,
	 it suffices to control the commutator terms involving the operator $e^{-G_h}$. 
	 Applying the formula
	$$ [B^2,e^{-G_h}]=2[B,e^{-G_h}]B+[B,[B,e^{-G_h}]],
	$$
	to $B=h\partial_x, h\partial_y$ and using \eqref{ecommutator} and \eqref{conjugationformula2},
	 we deduce that for $j=1,2$,
	$$ \|[h^j\partial_x^j,e^{-G_h}]\widetilde{u}_h\|_{L^2}\lesssim h^{j},\quad \|[h^j\partial_y^j,e^{-G_h}]\widetilde{u}_h\|_{L^2}\lesssim h^j,
	$$
hence (b) follows for $\widetilde{v}_h$.
	
	To estimate $\|[b_0'(x)\partial_x,e^{-G_h}]\widetilde{u}_h\|_{L^2}$, since $G_h, e^{sG_h}$ both commute with $b_0'(x)$, by formula \eqref{conjugationformula2}, we have
	 $$ b_0'(x)[\partial_x,e^{-G_h}]=[\partial_x,e^{-G_h}]b_0'(x).
	 $$ 
	From \eqref{ecommutator},  $\|[\partial_x,e^{-G_h}]\|_{\mathcal{L}(L^2)}\lesssim 1$ and $|b_0'|\lesssim b_0^{1/2}$, we deduce that
	   $$\|[b_0'(x)\partial_x,e^{-G_h}]\widetilde{u}_h\|_{L^2}\lesssim h^{\frac{1+\delta}{2}}.$$
	 Finally, to estimate $\|[(\partial_xA)\partial_x,e^{-G_h}]\widetilde{u}_h\|_{L^2}$, where
	 $$ A(x,y):=\int_{-\pi}^y(a_0(x,y')-b_0(x))dy',
	 $$
	 we write
	 $$[(\partial_xA)\partial_x,e^{-G_h}]\widetilde{u}_h=[(\partial_xA),e^{-G_h}]\partial_x\widetilde{u}_h+(\partial_xA)[\partial_x,e^{-G_h}]\widetilde{u}_h.
	 $$
	 Since $\|[(\partial_xA),e^{-G_h}]\|_{\mathcal{L}(L^2)}\lesssim h$, together with the estimate (b) for $\widetilde{u}_h$, we have
	 $$\|[(\partial_xA),e^{-G_h}]\partial_x\widetilde{u}_h\|_{L^2}\lesssim h^{1-\frac{\delta}{2}}\lesssim h^{\delta}.$$
	 Further commuting $\partial_xA$ and $[\partial_x,e^{-G_h}]$, by \eqref{conjugationformula2}, we can write
	 $$ (\partial_xA)[\partial_x,e^{-G_h}]=[\partial_x,e^{-G_h}](\partial_xA)+\mathcal{O}_{\mathcal{L}(L^2)}(h).
	 $$
	Therefore, 
	$$ \|(\partial_xA)[\partial_x,e^{-G_h}]\widetilde{u}_h\|_{L^2}\lesssim \|(\partial_xA)\widetilde{u}_h\|_{L^2}+O(h).
	$$
	By Lemma \ref{dampingproperty}, together with the fact that $|\nabla a_0|\lesssim a_0^{\frac{1}{2}}, |b_0'|\lesssim b_0^{1/2}$ and Jensen's inequality, we have
	$ |(\partial_xA)(x,y)|\lesssim b_0(x)^{\frac{1}{2}}.
	$ Therefore, we obtain that $\|(\partial_xA)\widetilde{u}_h\|_{L^2}\lesssim h^{\frac{1+\delta}{2}}\leq h^{\delta}$, hence (c) follows. 
	
	The proof of Proposition \ref{estimatesT2quasimodes} is now complete.
\end{proof}

\subsection{Proof of the optimality}
Take quasimodes $\widetilde{u}_h$ in Proposition \ref{estimatesT2quasimodes}. Define $\widetilde{v}_h=e^{-G_h}\widetilde{u}_h$, we are going to show that $\widetilde{v}_h$ are the desired quasimodes, satisfying
\begin{align}\label{final} 
(-h^2\Delta-1+iha_0)\widetilde{v}_h=O_{L^2}(h^{2+\delta}),\quad \|\widetilde{v}_h\|_{L^2}\sim 1. 
\end{align}
Then \eqref{final} implies \eqref{lastsection}. 
 
Recall the notations $b_0(x)=\mathcal{A}(a_0)(x)$, $G_h=\mathrm{Op}_h^w(g(x,y,\eta))$, where
$$ g(x,y,\eta)=-\frac{\psi_1(\eta)}{2\eta}\int_{-\pi}^y(a_0(x,y')-b_0(x))dy'.
$$
First, we prove:
\begin{prop}\label{averaginginverse} 
The quasimodes $\widetilde{v}_h=e^{-G_h}\widetilde{u}_h$ satisfy the equation
	$$ (-h^2\Delta-1+iha_0)\widetilde{v}_h+[h^2D_x^2,G_h]\widetilde{v}_h=O_{L^2}(h\hbar^2),
	$$
where $\hbar=h^{\frac{1+\delta}{2}}$.  	Moreover, $\widetilde{v}_h$ verifies properties $\mathrm{(a)},\mathrm{(b)},\mathrm{(c)}$ in Proposition \ref{estimatesT2quasimodes} and
 $$\|a_0^{\frac{1}{2}}\widetilde{v}_h\|_{L^2}+\|b_0^{\frac{1}{2}}\widetilde{v}_h\|_{L^2}=O(\hbar).$$

\end{prop}

\begin{proof}
 The proof is very similar to the proof of Proposition \ref{1averaging}, but much simpler, since we have better estimates for $\widetilde{v}_h=e^{-G_h}\widetilde{u}_h$, thanks to Proposition \ref{estimatesT2quasimodes}.
 Consider the conjugate operator
$$ \widetilde{F}_h(s):=e^{-sG_h}(P_h+ihb_0(x))e^{sG_h},
$$
where $P_h=-h^2\Delta-1$.
Similar to \eqref{normalform}, we obtain a simpler formula
\begin{align}\label{normalform'}
	\widetilde{F}_h(1)=e^{-G_h}(P_h+ihb_0(x))e^{G_h}=&P_h+ih\big(b_0-\frac{i}{h}[h^2D_y^2,G_h]\big)\notag\\+&[h^2D_x^2,G_h]+\frac{1}{2}[G_h,[G_h,P_h]]+\mathcal{O}_{\mathcal{L}(L^2)}(h^{2+\delta}).
\end{align}
Here we use the fact that $[G_h,b_0(x)]=0$. As the principal symbol of $\frac{i}{h}[h^2D_y^2,G_h]$ is $\psi_1(\eta)(b_0-a_0)$, we have
\begin{align}\label{equationv_h}
(P_h+iha_0+[h^2D_x^2,G_h])\widetilde{v}_h=f_h+O_{L^2}(h^{2+\delta}),
\end{align}
where
\begin{align}\label{f_h}
 f_h=ih\big(b_0-a_0-\frac{i}{h}[h^2D_y^2,G_h]\big)\widetilde{v}_h-\frac{1}{2}[G_h,[G_h,P_h]]\widetilde{v}_h.
\end{align}
It remains to show that $f_h=O_{L^2}(h^{2+\delta})$. 

First of all, the analogue of Lemma \ref{wavefront} holds with the same proof:
\begin{lem}\label{analogue} 
Let $\widetilde{\psi}$ be any cutoff such that $\widetilde{\psi}\equiv 1$ on the support of $\psi$ that defined $\widetilde{u}_h$ in Proposition \ref{estimatesT2quasimodes}. Then
$$ \mathrm{WF}_h^{10}(\widetilde{v}_h)\subset \mathrm{WF}_h(\widetilde{u}_h)
$$ 
and
$$ \|(1-\widetilde{\psi}(\hbar D_x))\widetilde{v}_h\|_{L^2}=O(\hbar^4)=O(h^{2+2\delta}).
$$
\end{lem}
Similarly, the analogue of Lemma \ref{erro1} holds. Therefore, we have
$$ (P_h+iha_0+[h^2D_x^2,G_h])\widetilde{v}_h+\frac{1}{2}[G_h,[G_h,P_h]]\widetilde{v}_h=O_{L^2}(h^{2+\delta}).
$$
Next, we claim that
\begin{align}\label{claima0}  \|a_0^{\frac{1}{2}}\widetilde{v}_h\|_{L^2}=O(\hbar)=O(h^{\frac{1+\delta}{2}}).
\end{align}
This follows by mimicking the proof of Lemma \ref{averagingdamping} with the new operator
$$ Q_h=a_0-\frac{i}{h}[h^2D_x^2,G_h].
$$
Indeed, by multiplying by $\ov{\widetilde{v}_h}$ to the equation and taking the imaginary part, we have
$$ (Q_h\widetilde{v}_h,\widetilde{v}_h)_{L^2}=O(\hbar^2).
$$
Hence
$$ \|a_0^{\frac{1}{2}}\widetilde{v}_h\|_{L^2}^2+\frac{i}{h}([h^2D_x^2,G_h]\widetilde{v}_h,\widetilde{v}_h)=O(\hbar^2).
$$
By \eqref{expansion}, \eqref{interpolation1} and \eqref{interpolation2}, we have
\begin{align*}
\big|\frac{i}{h}([h^2D_x^2,G_h]\widetilde{v}_h,\widetilde{v}_h)\big|\lesssim h\|b_0(x)^{\frac{1}{2}}\widetilde{v}_h\|_{L^2}\|\widetilde{v}_h\|_{L^2}+h\|b_0(x)^{\frac{1}{2}}hD_x\widetilde{v}_h\|_{L^2}\|b_0(x)^{\frac{1}{2}-\sigma}\widetilde{v}_h\|_{L^2},
\end{align*}
where $\sigma<\frac{1}{4}$. Note that by Lemma \ref{analogue}, we can write $\widetilde{v}_h=\psi(\hbar D_x)\widetilde{v}_h+O_{L^2}(h^{2+2\delta})$, hence
\begin{align*}
 \|b_0(x)^{\frac{1}{2}}hD_x\widetilde{v}_h\|_{L^2}\leq &h\hbar^{-1}\|b_0(x)^{\frac{1}{2}}\hbar D_x\psi(\hbar D_x)\widetilde{v}_h\|_{L^2}+O(h^{2+2\delta})\leq h\hbar^{-1}\|[b_0^{\frac{1}{2}},\hbar D_x\psi(\hbar D_x)]\widetilde{v}_h\|_{L^2}+O(h)\\
 \leq &O(h),
\end{align*}
thanks to the symbolic calculus. Therefore,
$$ \big|\frac{i}{h}([h^2D_x^2,G_h]\widetilde{v}_h,\widetilde{v}_h)\big|\leq O(h^2).
$$
In particular, we obtain \eqref{claima0}. Finally, the estimate
$$ \|[G_h,[G_h,P_h]]\widetilde{v}_h\|_{L^2}=O(h^{2+\delta})
$$
follows from the same argument in the last paragraph of the proof of Proposition \ref{1averaging}.
The proof of Proposition \ref{averaginginverse} is now complete. 
\end{proof}
  
In view of Proposition \ref{estimatesT2quasimodes} and Proposition \ref{averaginginverse}, to complete the proof of \eqref{final}, it remains to show that
\begin{align}\label{finalcommutator} 
\|[h^2D_x^2,G_h]\widetilde{v}_h\|_{L^2}=O(h\hbar^2).	
\end{align}
By the same argument as in the proof of Lemma \ref{secondaveraginglemma1},
	we have
	$$ [h^2D_x^2,G_h]\widetilde{v}_h=-4ih^2\hbar^{-1}(\partial_xA)b(hD_y)\hbar D_x\widetilde{v}_h+O_{L^2}(h\hbar^2),
	$$
where
$$ A(x,y)=\int_{-\pi}^y(a_0(x,y')-b_0(x))dy',
\quad b(hD_y)=-\frac{ \chi(hD_y)}{4hD_y}.
$$
 Since $\|[(\partial_xA),b(hD_y)]\|_{\mathcal{L}(L^2)}=O(h)$, from (c) of Proposition \ref{estimatesT2quasimodes}, we deduce that
$$ \|4ih^2\hbar^{-1}(\partial_xA)b(hD_y)\hbar D_x\widetilde{v}_h\|_{L^2}\leq O(h^3\hbar^{-1})+O(h^{2+\delta})=O(h^{2+\delta}),
$$ 
thanks to the fact that $\delta<\frac{1}{3}$. This proves \eqref{finalcommutator}.

%%%%%%%%%%%%%%%%%%%%%%%%%%%%%%%%%%%%%%%%%%%%%%%%%%%%%%%%%%%%%%%%%%%%%
\appendix

\section{Proof of Proposition \ref{1DHolderrough} }\label{gap}

In this appendix, we prove Proposition \ref{1DHolderrough}. Note that the proof works also for $\gamma=0$, thus covering the main result in \cite{St} for the piecewise constant rectangular damping. Without loss of generality, we assume that $\cup_{j=1}^l\ov{I}_j\neq \T$, otherwise, we can apply Theorem 1.7 of \cite{LLe} and the corresponding resolvent estimate $h^{\frac{2}{\gamma+2}}$ is much better than \eqref{1Dresolvent}.

Recall the numerology:
$ \delta=\frac{1}{\gamma+2}.
$
To simplify the notation in the exposition, we argue by contradiction. We assume that there exists a sequence $h_n\rightarrow 0$ and $\lambda_n\subset\R$ such that $c_1^{\frac{1}{2}}\leq \lambda_n\leq h_n^{-\frac{1-\delta}{2}}$, such that
\begin{align}\label{AC:eq1}
	-v_{h_n,\lambda_n}''-\lambda_n^2v_{h_n,\lambda_n}+ih^{-1}W(x)v_{h_n,\lambda_n}=r_{h_n,\lambda_n}
\end{align}
and
$$ \|v_{h_n,\lambda_n}\|_{L^2}=1,\quad \|r_{h_n,\lambda_n}\|_{L^2}=o(h_n^{\delta}).
$$
For simplicity, we will ignore the subindex $n$ for $h_n,\lambda_n$ and write simply $v=v_{h,\lambda}, r=r_{h,\lambda}$ sometimes without displaying their dependence in $h$ and $\lambda$.
First we record the a priori estimate, for which the proof is a direct consequence of integration by part (see the proof of Lemma \ref{apriori})
\begin{lem}\label{aprioriAC} 
	\begin{align*}
		&(a)\quad \|v'\|_{L^2}^2-\lambda^2\|v\|_{L^2}^2=o(h^{\delta}).\\
		&(b)\quad \|W^{\frac{1}{2}}v\|_{L^2}=o(h^{\frac{1+\delta}{2}}).
	\end{align*}
\end{lem}
Pick $\chi\in C^{\infty}(\R;[0,1])$ such that $\chi(s)\equiv 1$ when $s\leq 1$ and $\chi(s)\equiv 0$ when $s>2$. Denote by $I_j=(\alpha_j,\beta_j)$ and without loss of generality, we assume that $$-\pi<\alpha_1<\beta_1\leq\alpha_2<\beta_2\leq\cdots \leq\alpha_l<\beta_l\leq\pi.$$
Define the function
$$ V_0(x):=\sum_{j=1}^l\max\{0,(x-\alpha_j)(\beta_j-x)\}.
$$
Note that $V_0(x)=(x-\alpha_j)(\beta_j-x)$ when $x\in I_j$ for some $j\in\{1,\cdots,l\}$ and $V_0(x)=0$ whenever $W(x)=0$. Define $\chi_h(x):=\chi\Big(\frac{V_0(x)^3}{h^{3\delta}}\Big)$.  Note that $\chi_h\in C^2$.
Denote by $I_{j,h}=(\alpha_j,\alpha_j+2\pi h^{\delta})\cup(\beta_j-2\pi h^{\delta},\beta_j)$ for $j\in\{1,\cdots,n\}$ and $\widetilde{I}_{j,h}=(\alpha_j+\sigma h^{\delta},\alpha_j+2\pi h^{\delta})\cup (\beta_j-2\pi h^{\delta},\beta_j-\sigma h^{\delta})$. Here the constant $\sigma<2\pi$ is chosen so that $\chi_h|_{I_{j}}$ is constant on $I\setminus \widetilde{I}_{j,h}$. 
Hence supp$(\chi_h')$, supp$(\chi_h'')$ are all contained in $\cup_{j=1}^l\widetilde{I}_{j,h}$, and 
\begin{align}\label{eq3:AC}
	|\chi_h'|\lesssim h^{-\delta}\sum_{j=1}^l\mathbf{1}_{\widetilde{I}_{j,h}},\quad |\chi_h''|\lesssim h^{-2\delta}\sum_{j=1}\mathbf{1}_{\widetilde{I}_{j,h}}.
\end{align}
We decompose
$$ v=v_1+v_2,\quad v_1=\chi_hv,\quad v_2=(1-\chi_h)v.
$$ 
When $\gamma>0$, $v_2$ is supported in the damped region $W\gtrsim h^{\alpha}$ while $v_1$ is supported on $W\lesssim h^{\alpha}$, where
$$ \alpha=\frac{\gamma}{\gamma+2}.
$$
\begin{lem}\label{lem:v2}
	For $v_2$, we have
	$$ \|v_2\|_{L^2}=o(h^{\frac{3}{2(\gamma+2)}}).
	$$	
\end{lem}
\begin{proof}
	First we assume that $\gamma>0$, then from (b) of Lemma \ref{aprioriAC},
	$$ \|v_2\|_{L^2}\leq h^{-\frac{\alpha}{2}}\|W^{\frac{1}{2}}v\|_{L^2}\leq o(h^{\frac{1+\delta-\alpha}{2}})=o(h^{\frac{3}{2(\gamma+2)}}).
	$$	
	When $\gamma=0$, $\delta=\frac{1}{2}$, by Lemma \ref{averagingdamping}, $\|v_2\|_{L^2}\leq \|W^{\frac{1}{2}}v\|_{L^2}\leq o(h^{\frac{3}{4}})$. 
\end{proof}
It remains to estimate $v_1$. We see that $v_1$ solves the equation
\begin{align}\label{eq:v1}
	-v_1''-\lambda^2v_1+ih^{-1}Wv_1=\widetilde{r}:=\chi_hr-2(\chi_h'v)'+\chi_h''v.
\end{align}
\begin{rem}
	If $W$ is $C^2$, the choice of cutoff $\chi_h$ is the same as $\chi(h^{-\alpha}W)$. Note that the parameter $\alpha$ is chosen so that the size of $ih^{-1}Wv$ and $\chi_h''v$ are the same in $L^2$. This choice of cutoff is more accurate than the choice $\chi(h^{-1}W)$ 
	in \cite{BuHi}, in order to balance the size of $ih^{-1}Wv$ and $\chi_h''v$ coming from the commutator on the right hand side of \eqref{eq:v1}.  
\end{rem}

If $\lambda$ is relatively large,  we are still able to apply the estimate from the geometric control:
\begin{lem}\label{lem:3AC}
	If $\lambda\geq h^{-\frac{\delta}{2}}$, we have
	$$ \|v_1\|_{L^2}\lesssim h^{\frac{\delta}{2}}\|r\|_{L^2}+h^{-\frac{1+\delta}{2}}\|W^{\frac{1}{2}}v\|_{L^2}.
	$$	
\end{lem}
\begin{proof}
	By Lemma \ref{geometriccontrol}, we have
	\begin{align*}
		\|v_1\|_{L^2}\lesssim &\lambda^{-1}\|\chi_hr-ih^{-1}Wv_1+\chi_h''v\|_{L^2}+\|(\chi_h'v)'\|_{H^{-1}}+\|Wv_1\|_{L^2}\\
		\lesssim &h^{\frac{\delta}{2}}\|r\|_{L^2}+h^{-1+\frac{\alpha+\delta}{2}}\|W^{\frac{1}{2}}v\|_{L^2}+h^{-\frac{3\delta}{2}}\sum_{j=1}^l\|v\|_{L^2(\widetilde{I}_{j,h})}+h^{\frac{\alpha}{2}}\|W^{\frac{1}{2}}v\|_{L^2},
	\end{align*}
	where we use the fact that $W\lesssim h^{\alpha}$ on supp$(\chi_h)$. Since $W\sim h^{\alpha}$ on $\widetilde{I}_{j,h}$ and $\alpha+2\delta=1$, we have $$\|v\|_{L^2(\widetilde{I}_{j,h})}\lesssim h^{-\frac{\alpha}{2}}\|W^{\frac{1}{2}}v\|_{L^2}\lesssim h^{-\frac{1}{2}+\delta}\|W^{\frac{1}{2}}v\|_{L^2}.$$
	This completes the proof of Lemma \ref{lem:3AC}.
\end{proof}

It remains to deal with the regime where $c_1^{\frac{1}{2}}\leq\lambda<h^{-\frac{\delta}{2}}$, the key point is to exclude the possible concentration of the energy density
$$ e_0(x):=|v_1'(x)|^2+\lambda^2|v_1(x)|^2
$$
in the damped shell of size $h^{\delta}$ near the interface where $W=0$. The tool used in \cite{DKl} is a Morawetz type inequality introduced:
\begin{lem}[Morawetz type inequality]\label{Morawetz}Let $\Phi\in C^0(\T)$ be a piece-wise $C^1$ function on $\T$, then there exists a uniform constant $C>0$, such that
	\begin{align}\label{eq:Morawetz} 
		\int_{\T}\Phi'(x)e_0(x)dx \leq &Ch^{-1}\Big|\int_{\T}\Phi(x)W(x)v_1\ov{v}_1'dx\Big|\notag\\+&C\Big|\Re\int_{\T}\Phi(x)\ov{v}_{1}'\cdot (\chi_hr-2(\chi_h'v)'+\chi_h''v ) dx\Big|.
	\end{align}
\end{lem}
\begin{proof}
	Direct computation yields
	$$ (\Phi e_0)'=\Phi'e_0+2\Re(\Phi v_1''\ov{v}_{1}' +\lambda^2\Phi(|v_1|^2)').
	$$
	Using the equation \eqref{eq:v1}, we have
	\begin{align*}
		(\Phi e_0)'=&\Phi'e_0+\lambda^2\Phi\partial_x(|v_1|^2)-2\lambda^2\Re(\Phi\ov{v}_{1}' v_1)+2h^{-1}\Re (i\Phi\ov{v}_{1}'\cdot Wv_1 ) \\
		-&2\Re[\Phi\ov{v}_{1}' (\chi_hr-2(\chi_h'v)'+\chi_h''v) ].
	\end{align*}
	Since $2\lambda^2\Re(\Phi\ov{v}_{1}'v_1)=\lambda^2\Phi\cdot\partial_x(|v_1|^2)$,
	integrating the above identity, we
	obtain the desired estimate  \eqref{eq:Morawetz}.
\end{proof}
Let $\epsilon_j<\frac{\beta_j-\alpha_j}{2}, j\in\{1,\cdots, n\}$. Define 
\begin{align*}
	\Psi_h(x):=\left\{\begin{aligned}
		& h^{-\delta},\; x\in\cup_{j=1}^lI_{j,h};\\
		&1,\; x\in\cup_{j=1}^l(\alpha_h+2\pi h^{\delta},\alpha_j+\epsilon_j)\cup (\beta_j-\epsilon_j,\beta_j-2\pi h^{\delta});\\
		-& M,\; x\in\cup_{j=1}^lI_j\setminus( (\alpha_j,\alpha_j+\epsilon_j)\cup (\beta_j-\epsilon_j,\beta_j) );\\
		&1,\; x\in\T\setminus \cup_{j=1}^lI_j
	\end{aligned}
	\right. 
\end{align*}
where the constant $M>0$ (independent of $h$) is chosen such that $\int_{\T}\Psi_h(x)dx=0$. Then the primitive function $\Phi_h\in C^0(\T)$ is well-defined, piecewise smooth and $\Phi_h'(x)=\Psi_h(x)$. Since $\Phi_h$ is unique up to a constant, we choose $\Phi_h$ such that  $\Phi_h(0)=0$, then 
$\|\Phi_h\|_{L^{\infty}(\T)}\leq C_M$. Define
$$ \Theta(x)=\Phi_h'(x)\mathbf{1}_{\Phi_h'>0},
$$
since supp$(e_0)\subset \mathrm{supp}(\chi_h)$, we have  $\Phi_h'(x)e_0(x)=\Theta(x)e_0(x)$.
From Lemma \ref{Morawetz} and $v_1'=\chi_h'v+\chi_hv'$, we have
\begin{align}\label{consequenceMorawetz} 
	\int_{\T}\Theta(x)e_0(x)dx\leq & Ch^{-1}\Big|\int_{\T}W(x)(\chi_h^2v\ov{v}'+2\chi_h\chi_h'|v|^2)dx\Big|\notag\\
	+&C\int_{\T}|\chi_hr (\chi_h'v+\chi_hv')|dx+C\int_{\T}|\chi_h''\ov{v}(\chi_h'v+\chi_hv')|dx\notag \\+&C\Big|\Re\int_{\T}\chi_h'\ov{v}\Phi_h(x)\chi_hv'' dx\Big|+C\int_{\T}|\chi_h'(\Phi_h\chi_h)'\ov{v}v'|dx+C\int_{\T}|\chi_h'\ov{v}(\Phi_h(x)\chi_h'v)'|dx,
\end{align}
where the terms in the third line of the right side is obtained by integration by part of $$\Big|\Re\int_{\T}\Phi_h(x)\ov{v}_1'(\chi_h'v)'dx\Big|.$$ Here we keep the real part for this term in order to perform some cancellation when replacing $v''$ by $-\lambda^2v+ih^{-1}Wv-r$ later on, after using the equation \eqref{AC:eq1} of $v$.  Denote by $\widetilde{I}_h=\cup_{j=1}^l\widetilde{I}_{j,h}$ and $I_h=\cup_{j=1}^lI_{j,h}$. Using \eqref{eq3:AC} and the facts that $\lambda^2\leq h^{-\delta}$, $W\sim h^{\alpha}$ on $\widetilde{I}_h$, we can control terms on the right hand side of \eqref{consequenceMorawetz} by the sum of following types:
\begin{align*}
	& \text{Type I}:\; \mathrm{I}=h^{-\delta}\|v\mathbf{1}_{\widetilde{I}_h}\|_{L^2}\|r\|_{L^2}+\|r\|_{L^2}\|v'\chi_h\|_{L^2};\\
	&\text{Type II}:\; \mathrm{II}=h^{-1}\int_{\T}W|vv'|\mathbf{1}_{\widetilde{I}_h}dx+h^{-1}\int_{\T}W\chi_h^2|vv'|dx;\\
	&\text{Type III}:\; \mathrm{III}=h^{-(1+\delta)}\|W^{\frac{1}{2}}v\mathbf{1}_{\widetilde{I}_h}\|_{L^2}^2.
\end{align*}
We analyze the type II term. By Cauchy-Schwarz,
\begin{align}
	\mathrm{II}\lesssim h^{-1}\|W^{\frac{1}{2}}v\mathbf{1}_{I_h}\|_{L^2}\|W^{\frac{1}{2}}\mathbf{1}_{I_h}v'\|_{L^2}.
\end{align}
Note that on $I_h$, $\Theta(x)=h^{-\delta}$, we have
\begin{align*}
	\mathrm{II}\lesssim h^{-1+\frac{\alpha}{2}}\|W^{\frac{1}{2}}v\|_{L^2}\|\mathbf{1}_{I_h}v'\|_{L^2}\leq h^{-1+\frac{\alpha}{2}+\frac{\delta}{2}}\|W^{\frac{1}{2}}v\|_{L^2}\Big(\int_{\T}\Theta(x)e_0(x)dx\Big)^{\frac{1}{2}}.
\end{align*}
Since $\alpha+2\delta=1$, using Young's inequality, we have
\begin{align*}
	\mathrm{II}\leq \epsilon\int_{\T}\Theta(x)e_0(x)dx+C_{\epsilon}h^{-(1+\delta)}\|W^{\frac{1}{2}}v\|_{L^2}^2,\quad \forall \epsilon>0.
\end{align*}
Plugging into \eqref{consequenceMorawetz}, we get 
\begin{align*}
	\|v_1\|_{L^2}^2\leq \int_{\T}\Theta(x)e_0(x)dx\lesssim h^{-(1+\delta)}\|W^{\frac{1}{2}}v\|_{L^2}^2+h^{-\delta}\|v\|_{L^2}\|r\|_{L^2}+\|r\|_{L^2}\|v'\|_{L^2}.
\end{align*}
Since $\lambda\leq h^{-\frac{\delta}{2}}$, by (a) of Lemma \ref{aprioriAC},  
$$ \|v'\|_{L^2}^2\leq o(h^{\delta})+\lambda^2\|v\|_{L^2}^2\leq o(h^{-\delta}).
$$
Recall that $\|r\|_{L^2}=o(h^{\delta})$ and $\|W^{\frac{1}{2}}v\|_{L^2}=o(h^{\frac{1+\delta}{2}})$, we get
$$ \|v_1\|_{L^2}^2\lesssim \int_{\T}\Theta(x)e_0(x)dx\lesssim o(1).
$$
Since we have already shown that $\|v_2\|_{L^2}=o(1)$, this is a contradiction to our assumption that $\|v\|_{L^2}=1$. The proof of Proposition \ref{1DHolderrough} is now complete.
%%%%%%%%%%%%%%%%%%%%%%%%%%%%%%%%%%%%%%%%%%%%

%%%%%%%%%%%%%%%%%%%%%%%%%%%%%%%%%%%%%%%%%%%%
\section{Semiclassical pseudo-differential calculus}

For $m\in\R$, the symbol class $S^m(T^*\R^d)$ consists of smooth functions $c(z,\zeta)$ such that
$$ |\partial_z^{\alpha}\partial_{\zeta}^{\beta}c(z,\zeta)|\leq C_{\alpha,\beta}\langle\zeta\rangle^{m-|\alpha|}.
$$
Given a symbol $c(z,\zeta)$, we associate it with the Weyl quantization $\mathrm{Op}_h^w(c)$:
$$ f\in\mathcal{S}(\R^d)\mapsto \mathrm{Op}_h^w(f)(z):=\frac{1}{(2\pi h)^d}\iint_{\R^{2d}}c\big(\frac{z+z'}{2},\zeta\big)e^{\frac{i(z-z')\cdot\zeta}{h}}f(z')dz'd\zeta.
$$
Most of the time we will use the standard quantization $\mathrm{Op}_h(c)$:
$$ \mathrm{Op}_h(c)(f)(z):=\frac{1}{(2\pi h)^d}\iint_{\R^{2d}}c(z,\zeta)e^{\frac{i(z-z')\cdot\zeta}{h}}f(z')dz'd\zeta. 
$$
An important mapping property is the following theorem due to Calder\'on-Vaillancourt:
\begin{thm}\label{CalderonVaillancourt} 
There exists a constant $C>0$ such that for any function $c$ on $T^*\R^d$ with uniformly bounded derivatives up to order $d$, we have
$$ 
\|\mathrm{Op}_h(c)\|_{\mathcal{L}(L^2(\R^d))}+
\|\mathrm{Op}^w_h(c)\|_{\mathcal{L}(L^2(\R^d))}\leq C\sum_{|\alpha|,|\beta|\leq d}h^{|\beta|}\|\partial_z^{\alpha}\partial_{\zeta}^{\beta}c\|_{L^{\infty}(T^*\R^d)}.
$$
\end{thm}
We use frequently the sharp G$\mathring{\mathrm{a}}$rding inequality:
\begin{thm}[Sharp G$\mathring{\mathrm{a}}$rding's inequality]\label{Garding} 
Assume that $c\in S^0(T^*\T^d)$ and $c(z,\zeta)\geq 0$ for all $(z,\zeta)\in T^*\T^d$. Then there exist $C>0$ and $h_0>0$ such that for all $0<h\leq h_0$ and $f\in L^2(\T^d)$,
$$ \Re(\mathrm{Op}_h(c)f,f)_{L^2}\geq -Ch\|f\|_{L^2(\T^d)}^2.
$$
\end{thm}

Since we deal with symbols (damping) of limited regularity in this article, we need additional estimates. First we recall the following boundedness property on $L^2$:
\begin{lem}\label{Boundedness} 
Assume that $b(x,y,\xi)\in L^{\infty}(\R^{3d})$ and there exists $\mu_0>0$, such that for all $|\alpha|\leq d+1$,
$$ |\partial^{\alpha}_{\xi}b(x,y,\xi) |\leq C_{\alpha}\langle\xi\rangle^{-(|\alpha|+\mu_0)}.
$$
Then the operator $T_h$ associated with the Schwartz kernel
$$ K_h(x,y)=\frac{1}{(2\pi h)^d}\int_{\R^d}b(x,y,\xi)e^{\frac{i(x-y)\cdot\xi}{h}}d\xi
$$
is bounded on $L^2(\R^d)$, uniformly in $h\in(0,1]$.
\end{lem}

\begin{proof}
	The proof with $\mu_0=1$ can be found in Lemma A.1 of \cite{BuSun2}. The same argument works for all $\mu_0>0$.
\end{proof}
A direct consequence is the following commutator estimates for Lipschitz functions:
\begin{cor}\label{commutator} 
	Assume that $\kappa\in W^{1,\infty}(\R^d)$ and  $b\in S^0(T^*\R^d)$, then there exists $C>0$ such that
	$$ \|[\kappa,\mathrm{Op}_h(b)]\|_{\mathcal{L}(L^2(\R^d))}+\|[\kappa,\mathrm{Op}_h^w(b)]\|_{\mathcal{L}(L^2(\R^d))}\leq Ch. 
	$$
Moreover generally, for some $m\geq 1$, we have
$$ \|\mathrm{ad}_{\kappa}^m(\mathrm{Op}_h(b))\|_{\mathcal{L}(L^2(\R^d))}+\|\mathrm{ad}_{\kappa}^m(\mathrm{Op}_h^w(b))\|_{\mathcal{L}(L^2(\R^d))}\leq Ch^m.
$$
\end{cor}
\begin{proof}
We do the proof for the standard quantization. The same proof applied to the Weyl quantization.	Denote by $T_j=\mathrm{ad}_{\kappa}^j(\mathrm{Op}_h(b))$ and $K_j(x,y)$ the Schwartz kernel of $T_j$. Then
	$$ K_j(x,y)=(\kappa(x)-\kappa(y))^jK_0(x,y),
	$$ 
	where
	$$ K_0(x,y)=\frac{1}{(2\pi h)^d}\int_{\R^d}b(x,\xi)e^{\frac{i(x-y)\cdot\xi}{h}}d\xi
	$$
	is the Schwartz kernel of $\mathrm{Op}_h(b)$. Since $\kappa\in W^{1,\infty}(\R^d)$, there exists $\Psi\in L^{\infty}(\R^{2d};\R^d)$, such that $$\kappa(x)-\kappa(y)=(x-y)\cdot\Psi(x,y)=\sum_{l=1}^d(x_l-y_l)\Psi^{(l)}(x,y),$$
	where $\Psi^{(l)}$ is the $l$-th component of $\Psi$.
	 Thus by integration by part,
	\begin{align*}
	K_m(x,y)=\sum_{j_1,\cdots,j_m}\frac{h^m}{(2\pi h)^d}\int_{\R^d}e^{\frac{i(x-y)\cdot\xi}{h}}(D_{\xi_1}\cdots D_{\xi_m}b)(x,\xi)\prod_{k=1}^m\Psi^{(j_k)}(x,y)d\xi.
	\end{align*}
Applying Lemma \ref{Boundedness}, we deduce that $\|T_m\|_{\mathcal{L}(L^2(\R^d))}\leq Ch^m$.
\end{proof}

%\begin{cor}\label{1DCV}
%Let $\kappa=\kappa(x,y)\in C_c(\R^2)$ be a bounded continuous function on $\R^2$ and $\varphi=\varphi(\xi)\in C_c^{\infty}(\R)$. Then the operator $\mathrm{Op}_h(\varphi(\xi)\kappa(x,y))$ is bounded on $L^2(\R^2)$, uniformly in $h\in(0,1]$.	
%\end{cor}
%\begin{proof}
%Take $f\in\mathcal{S}(\R^2)$, we have
%\begin{align*}
% \mathrm{Op}_h(\varphi\kappa)f(x,y)=&\frac{1}{(2\pi h)^2}\int\varphi(\xi)\kappa\big(x,y\big)e^{\frac{i(x-x')\xi}{h}+\frac{i(y-y')\eta}{h}}f(x',y')dx'dy'd\xi d\eta\\
% =&\frac{1}{2\pi h}\int\varphi(\xi)\kappa\big(x,y\big)e^{\frac{i(x-x')\cdot\xi}{h}}f(x',y)dx'd\xi\\=&(\mathrm{Op}_h(\varphi\kappa(\cdot,y))f(\cdot,y))(x). 
%\end{align*}
%By hypothesis,  $\|\kappa(\cdot,y)\|_{L^{\infty}(\R)}$ is bounded and independent of $y$, from Lemma \ref{Boundedness}, we have
%$$ \|\mathrm{Op}_h(\varphi\kappa)f(\cdot,y)\|_{L^2(\R_x)}\leq C\|f(\cdot,y)\|_{L^2(\R_x)}.
%$$
%Finally, taking $L^2$ norm in $y$, the proof is complete.
%\end{proof}
We also need a special version of symbolic calculus, used in the proof of Lemma \ref{newerro}:
\begin{lem}\label{symbolicspecial} 
Let $\kappa\in C_c^1(\R^2)$, $\varphi=\varphi(\xi)\in C_c^{\infty}(\R)$ and $b_2\in C_c^1(\R^2)$, $b_2\in C_c^2(\R^2)$. Denote by $c(x,y,\xi)=\kappa(x,y)\varphi(\xi)$. Then 
\begin{align*}
&(a)\hspace{0.3cm}
\|\mathrm{Op}_h(c)b_1-\mathrm{Op}_h(cb_1)\|_{\mathcal{L}(L^2(\R^2))}\leq C_1h;
 \\
&(b)\hspace{0.3cm} \big\|\mathrm{Op}_{h}(c)b_2-\mathrm{Op}_{h}(cb_2)-\frac{h}{i}\mathrm{Op}_h(\partial_{\xi}c\cdot\partial_xb_2)\big\|_{\mathcal{L}(L^2(\R^2))}\leq C_2h^2,
\end{align*}	
where the constants $C_1$ depend only on $\|\kappa\|_{W^{1,\infty}},\|b_1\|_{W^{1,\infty}}$ and $C_2$ depend only on $\|\kappa\|_{W^{1,\infty}}$ and $\|b_2\|_{W^{2,\infty}}$.
\end{lem}
\begin{proof}
Since $c$ does not depend on $\eta$, by viewing $y$ as a parameter as in the proof of Lemma \ref{Boundedness}, it suffices to view $\mathrm{Op}_h(c), b$ as operators acting on $L^2(\R_x)$ and prove the one-dimensional estimate. Hence we will not display the dependence in $y$ in the analysis below.

For $j=1,2$, note that the symbol of operators $\mathrm{Op}_h(c)b_j$ are given by
$$ c_j(x,\xi)=\frac{1}{2\pi h}\iint_{\R^2} \mathrm{e}^{-\frac{ix_1\cdot\xi_1}{h}}c(x,\xi+\xi_1)b_j(x+x_1)dx_1d\xi_1=\frac{1}{2\pi }\int_{\R}e^{ix\xi'}c(x,\xi+h\xi')\widehat{b_j}(\xi')d\xi'.
$$ 
Note that the Fourier transform makes sense since the function $b_j$ has compact support. 
%Recall the formal identities
%$$ \frac{1}{(2\pi)^d}\iint_{\R^{2d}}e^{-iX\cdot\Xi}dXd\Xi=1,\quad\frac{1}{(2\pi)^d}\iint_{\R^{2d}}e^{-iX\cdot\Xi}X^{\alpha}\Xi^{\beta}dXd\Xi=(-i)^{|\alpha|}\alpha!\delta_{\alpha\beta}, 
%$$
%understood in the sense of oscillatory integrals.
By the Taylor expansions up to order 1 and 2:
\begin{align*}
	&c(x,\xi+h\xi')=c(x,\xi)+h\xi'\int_0^1\partial_{\xi}c(x,\xi+th\xi')dt,
	\\
	& c(x,\xi+h\xi')=c(x,\xi)+h\xi'\partial_{\xi}c(x,\xi)+h^2\xi'^2\int_0^1(1-t)\partial_{\xi}^2c(x,\xi+th\xi')dt,
\end{align*}
 we have
 \begin{align*}
 c_1(x,\xi)=c(x,\xi)b_1(x)+h\int_0^1\Phi_{1,t}(x,\xi)dt
 \end{align*}
 and
\begin{align*}
c_2(x,\xi)=c(x,\xi)b_2(x)+\frac{h}{i}\partial_{\xi}c(x,\xi)\partial_xb_2(x)+h^2\int_0^1(1-t)\Phi_{2,t}(x,\xi)
dt,
\end{align*}
where
$$ \Phi_{1,t}(x,\xi)=\frac{1}{2\pi i}\int_{\R}e^{ix\xi'}\partial_{\xi}c(x,\xi+th\xi')\widehat{\partial_xb_1}(\xi')d\xi'
$$
and
$$ \Phi_{2,t}(x,\xi)=-\frac{1}{2\pi}\int_{\R}e^{ix\xi'}\partial_{\xi}^2c(x,\xi+th\xi')\widehat{\partial_x^2b_2}(\xi')d\xi'.
$$
Therefore, the symbol of the operator $\mathrm{Op}_h(c)b_1-\mathrm{Op}_h(cb_1)$ is $h\int_0^1\Phi_{1,t}(x,\xi)dt$
and the symbol of the operator
 $\mathrm{Op}_h(c)b_2-\mathrm{Op}_h(cb_2)-\frac{h}{i}\mathrm{Op}_h(\partial_{\xi}c\partial_xb_2)$ is $h^2\int_0^1(1-t)\Phi_t(x,\xi)dt$. It suffices to show that operators $T_{j,t}$ with  Schwartz kernels
$$ \mathcal{K}_{j,t}(x,x'):=\frac{1}{2\pi h}\int_{\R}e^{\frac{i(x-x')\xi}{h}}\Phi_{j,t}(x,\xi)d\xi,\quad j=1,2 
$$ 
are uniformly bounded on $L^2(\R)$ with respect to $t\in(0,1)$ and $h\in(0,1]$. Since $c(x,\xi)=\kappa(x)\varphi(\xi)$, explicit computation yields
$$\mathcal{K}_{j,t}(x,x')=\frac{1}{i^jh}\widehat{\varphi^{(j)}}\big(\frac{x'-x}{h}\big)\kappa(x)\cdot(\partial_x^jb_j)((1-t)x+tx'),\quad j=1,2.
$$ 
Since $|\mathcal{K}_{j,t}(x,x')|\leq C\frac{1}{h}\big|\widehat{\varphi''}\big(\frac{x'-x}{h}\big)\big|$. Finally, by Young's convolution inequality, we obtain that $$\|T_{j,t}f\|_{L^2(\R)}\leq C\|f\|_{L^2(\R)}$$ for $j=1,2$. This completes the proof of Lemma \ref{symbolicspecial}.
\end{proof}

The above estimates can be generalized on to symbols and functions on compact manifolds. In the special situation where the manifold is $\T^d$, we still have explicit formulas. Indeed, following \cite{EZB} (Chapter 5), for a symbol $c(z,\zeta)$ on $T^*\T^d$, by periodicity $c(z+2\pi k,\xi)=c(z,\zeta), k\in\Z^d$, the quantization is explicitly given by
$$ \mathrm{Op}_h(c)f(z)=\sum_{k\in\Z^d}C_kf(z),\; C_kf(z):=\frac{1}{(2\pi h)^d}\int_{\R^d}\int_{\T^d}c(z,\zeta)e^{\frac{i(z-z'+2\pi k)\cdot\zeta}{h}}f(z')dz'd\zeta.
$$
Then $C_k=\mathbf{1}_{\T^d}\tau_{-2\pi k}\mathrm{Op}_h(a)\mathbf{1}_{\T^d}$ where $\tau_{z_0}f(z):=f(z-z_0)$.
By the stationary phase analysis, for symbols $c\in S^0(T^*\T^d)$ and for $|k|>2$, 
$$ \|C_k\|_{\mathcal{L}(L^2(\T^d))}=O(h^{N}\langle k\rangle^{-N}),\;\forall N\in\N.
$$
These facts imply that Lemma \ref{Boundedness}, Lemma \ref{commutator} and Lemma \ref{symbolicspecial} still hold by replacing $\R^d,\R^2,\R$ to $\T^d,\T^2,\T$,  respectively\footnote{One can also argue in local coordinate and use the partition of unity.}.

Finally, we recall the definition of the semiclassical wavefront set, following Chapter 8 of \cite{EZB}. The \emph{semiclassical wavefront set} $\mathrm{WF}_h(u)$ associated with a $h$-tempered family $u=(u_h)_{0<h\leq h_0}$ is the complement of the set of points $(z_0,\zeta_0)\in T^*\T^d$ for which there exists a symbol $c\in S^0$ such that $a(z_0,\zeta_0)\neq 0$ and $\mathrm{Op}_h(c)u_h=O_{L^2}(h^N)$ for all $N$. The \emph{semiclassical wavefront set of order} $m$, $\mathrm{WF}_h^m(u)$ is a modification of the definition $\mathrm{WF}_h(u)$. It is the complement of the set of points $(z_0,\zeta_0)\in T^*\T^d$ for which there exists $c\in S^0$ with $c(z_0,\zeta_0)\neq 0$ and $\mathrm{Op}_h(c)u_h=O_{L^2}(h^m)$.   
%%%%%%%%%%%%%%%%%%%%%%%%%%%%%%%%%%%%%%%%%%%%%%%%%%%%%%%%

%%%%%%%%%%%%%%%%%%%%%%%%%%%%%%%%%%%%%%%%%%%%%%%%%

\end{document}